\documentclass[12pt,reqno]{amsart}

\usepackage{graphicx}%
\usepackage{array,multirow,graphics}%
\usepackage{amsmath,amssymb,amsfonts}%
\usepackage{amsthm}%
\usepackage{thmtools}%
\usepackage{thm-restate}%
\usepackage{mathrsfs}%
\usepackage[title]{appendix}%
\usepackage{xcolor}%
\usepackage{colortbl}%
\usepackage{textcomp}%
\usepackage{manyfoot}%
\usepackage{booktabs}%
\usepackage{algorithm}%
\usepackage{algorithmicx}%
\usepackage{algpseudocode}%
\usepackage{listings}%
\usepackage{rotating}%
\usepackage[utf8]{inputenc}%
\usepackage{hyperref}%
\usepackage{tikz}%
\usetikzlibrary{fit}%
\usepackage{todonotes}%
\usepackage{float}%
\usepackage{subcaption}%

\usepackage[utf8]{inputenc}
\usepackage{hyperref}
\hypersetup
{
	colorlinks,	%
	citecolor=blue,%
	filecolor=black,%
	linkcolor=red,%
	urlcolor=blue
}

\usepackage[margin=1in]{geometry}

\newtheorem{theorem}{Theorem}[section]
\newtheorem{lemma}[theorem]{Lemma}

\newtheorem{cor}[theorem]{Corollary}
\theoremstyle{definition}
\newtheorem{definition}{Definition}[section]
\newtheorem{example}{Example}[section]
\newtheorem{remark}{Remark}[section]

\newcommand{\boundary}{\partial_{k,\ell}}
\newcommand{\im}{\mathrm{im}}

\newcommand{\E}{\mathbb{E}}
\newcommand{\R}{\mathbb{R}}

\newcommand{\rank}{\mathrm{rank}}
\newcommand{\len}{\mathrm{len}}

\definecolor{forest}{RGB}{34, 139, 34}
\definecolor{graphBlue}{RGB}{100, 143, 255}
\definecolor{graphRed}{RGB}{220, 38, 127}
\definecolor{graphYellow}{RGB}{255, 176, 0}

\usepackage{soul}
\usepackage{cancel}

\title[Eulerian magnitude homology] {Eulerian magnitude homology: subgraph structure and random graphs}
\author{Chad Giusti}
\address[CG]{Department of Mathematics, Oregon State University Corvallis, OR 97405}
\email{chad.giusti@oregonstate.edu}

\author{Giuliamaria Menara}
\address[GM]{Department of Mathematics and Geosciences, University of Trieste, Trieste, 34127, Italy}
\email{giuliamaria.menara@phd.units.it}


\begin{document}
\maketitle

\begin{abstract}In this paper we explore the connection between the ranks of the magnitude homology groups of a graph and the structure of its subgraphs. To this end, we introduce variants of magnitude homology called \emph{eulerian} magnitude homology and \emph{discriminant} magnitude homology. Leveraging the combinatorics of the differential in magnitude homology, we illustrate a close relationship between the ranks of the eulerian magnitude homology groups on the first diagonal and counts of subgraphs which fall in specific classes.
We leverage these tools to study limiting behavior of the eulerian magnitude homology groups for Erd\H{o}s-R\'{e}nyi random graphs and random geometric graphs, producing for both models a vanishing threshold for the eulerian magnitude homology groups on the first diagonal. This in turn provides a  characterization of the generators for the corresponding magnitude homology groups. Finally, we develop an explicit asymptotic estimate the expected rank of eulerian magnitude homology along the first diagonal for these random graph models.
\end{abstract}

\section{Introduction}\label{sec1}

In  \cite{leinster2013magnitude}, Leinster introduced the \emph{magnitude} of a metric space, a measurement of ``size" analogous to the Euler characteristic of a category \cite{leinster2006euler}. Magnitude provides a measure of size in the sense that it satisfies
\[
\text{Size}(A\cup B) = \text{Size}(A) + \text{Size}(B) - \text{Size}(A \cap B) 
\]
\[
\text{Size}(A \times B) = \text{Size}(A) \times \text{Size}(B),
\]
in analogy with other common notions of size like cardinality of sets, volumes of subsets of $\R^n$, dimensions of vector spaces, and Euler characteristic of topological spaces. Because it is intimately related to Euler characteristic, it is natural to consider categorification of the magnitude to a homology theory, as described in \cite{hepworth2015categorifying, leinster2021magnitude}. 
Homology theories for metric spaces are of general interest to the applied topology community, as evidenced by the success of persistent homology. The case of graphs equipped with the path metric is of particular utility, as new methods for interpreting and understanding the structure of networks are in general demand.

Explicitly, the magnitude chain complex of a simple graph is a bigraded complex described in terms of \emph{$k$-trails}, lists of $k+1$ \emph{landmark} vertices in the graph. A $k$-trail is of length $\ell$ if there is a walk in the graph of of length $\ell$ which passes through the listed landmarks in order, and this walk is minimal in length among all such. The differential in this chain complex deletes individual landmarks, with non-zero coefficient precisely when the deletion does not change the length of the trail. Such nonzero terms thus indicate the corresponding landmark is non-essential; a trail following the other landmarks must pass through at least as many vertices, even without this instruction. One can think of this as a form of ``non-convexity", as there are no shorter paths between some pair of landmarks \cite{leinster2021entropy}. However, beyond this term-wise intuition, the structure of magnitude homology groups remains hard to interpret. Recent developments in this direction are achieved in~\cite{asao2024magnitude}.

In this paper, we will make some progress toward elucidating the connection between magnitude homology of simple graphs equipped with the path metric and their combinatorial structure. The authors are certainly not the first to consider this problem.
In 2018 Gu confirmed the existence of graphs with the same magnitude but different magnitude homology \cite{gu2018graph}, while Kaneta and Yoshinaga have analyzed the structure and implications of torsion in magnitude homology \cite{kaneta2021magnitude}.
Torsion in the magnitude homology of graphs was also studied by Sazdanovic and Summers in \cite{sazdanovic2021torsion} and by Caputi and Collari in \cite{caputi2023finite}.
In \cite{asao2021girth} the authors explore the interaction between the magnitude homology of a graph and its diagonality and girth. 
Tajima and Yoshinaga investigate the connection between the homotopy type of the Asao-Izumihara CW complex (c.f. \cite{asao2020geometric}) and the diagonality of magnitude homology groups, \cite{tajima2021magnitude}.
As demonstrated by Cho in \cite{cho2019quantales}, magnitude and persistent homology can be thought of as endpoints on a spectrum of such theories.
Also, it is evident from \cite{hepworth2015categorifying} that calculating the magnitude homology of a graph is intricate and challenging, leading to the emergence of several technical approaches for its computation \cite{asao2020geometric, gu2018graph, sazdanovic2021torsion}. 

\medskip 
The approach we take in this work is to observe that a large portion of the magnitude chain complex is redundant, in the sense that the chains reflect combinatorial structure already recorded by chains of lower bigrading. To leverage this observation, in Section \ref{sec:EMH}, we define the subcomplex of \emph{eulerian magnitude chains}, supported on trails with no repeated landmarks. Focusing on the $k=\ell$ line, where the list of involved landmarks completely determines a trail, we develop the notion of a \emph{structure graph} for families of chains, which captures which chains share terms in their differentials. In Theorem \ref{thm:cycle_decomp}, we demonstrate that cycles in this complex can be decomposed into cycles with simple structure graphs. We leverage in this decomposition Corollary \ref{cor:cycle_spanning_set} to give a simple characterization of a generating set for $(k,k)$-graded eulerian magnitude homology of a graph. Combining this generating set with the language of \emph{class graphs} from \cite{yu2009computing}, we thus characterize which subgraphs of a graph support non-trivial cycles in $EMH_{k,k}(G)$ in terms of the corresponding structure graphs, providing a framework for computing these groups for graphs of interest.

Equipped with these observations, we consider the complementary set of trails, which must revisit a landmark, which we call the \emph{discriminant magnitude chains}, by analogy to the study of singular maps in the study of embeddings. Leveraging the corresponding long exact sequence in homology, in Theorems \ref{thm:relation MH EMH DMH} and \ref{thm:iso MH DMH}, we observe that the eulerian magnitude homology groups of lower bifiltration control the structure of the discriminant magnitude homology groups, and when lower eulerian magnitude homology groups vanishes we obtain a complete characterization of the generators of the $(k,k)$-magnitude homology groups.

In the interest of exploring what features of a graph the (eulerian) magnitude homology groups capture, we then turn our attention to two classes of random graphs: in Section \ref{sec:EMH_ER}, we study Erd\"os-R\'enyi random graphs, and in \ref{sec:EMH_RGG} we study random geometric graphs on the standard torus. By understanding these ``unstructured" examples, we aim to provide a backdrop for interpreting magnitude homology of ``structured" graphs including those observed in applications. In each context, we derive a vanishing threshold for the limiting expected rank of the $(k,k)$-eulerian magnitude homology in terms of the density parameter (Theorems \ref{thm:vanishingthreshER} and \ref{thm:vanishing threshold RGG}). In combination with our observation above about the relationship to discriminant magnitude homology, this provides a characterization of the structure of expected $(k,k)$-magnitude homology groups in this range. Further, adapting tools from \cite{kahle2013limit}, we develop a characterization of the limiting expected Betti numbers of the $(k,k)$-eulerian magnitude homology groups in terms of density (Theorems \ref{thm:asymp size BettiEMH} and \ref{thm:asympBettiRGG}) and corresponding central limit theorems.

\subsection{Outline}
The paper is organized as follows. 
We start by recalling in Section \ref{sec:Background} some general background about graphs, magnitude homology and random models.
Then in Section \ref{sec:EMH} we introduce \emph{eulerian} magnitude homology and \emph{discriminant} magnitude homology and highlight the advantages of the first in terms of interpretation. We then investigate the relationship between magnitude homology and the newly defined eulerian and discriminant magnitude homology.

In Sections \ref{sec:EMH_ER} and \ref{sec:EMH_RGG} we turn our attention to the computation of eulerian magnitude homology of Erd\H{o}s-R\'{e}nyi random graphs and random geometric graphs, respectively.
In both cases, we will develop an in-depth analysis of the eulerian magnitude homology groups of the $k=\ell$ diagonal by identifying the subgraphs induced by homology cycles, producing a vanishing threshold and providing an explicit formula for the asymptotic size of the eulerian magnitude homology groups.

\section{Background}
\label{sec:Background}

\emph{Magnitude} is a multiscale measure of ``size" for a metric space developed by Leinster in \cite{leinster2006euler}, and studied in the particular case of graphs in \cite{leinster2019magnitude}. In this latter paper, Leinster develops cardinality-like properties of the magnitude of a graph, including being multiplicative with respect to the Cartesian product and having (in some more restrictive cases) an inclusion-exclusion formula for the magnitude of a union. It is also shown that the magnitude of a graph is formally both a rational function over $\mathbb{Q}$ and a power series over $\mathbb{Z}$, and that it shares similarities with the Tutte polynomial. 

The \emph{magnitude homology} of a graph $MH_{k,\ell}(G)$, first introduced by Hepworth and Willerton in \cite{hepworth2015categorifying}, is a categorification of magnitude. We will now recall their definition, some simple results, and give an example of the computation of a magnitude homology group that will provide motivation for our definition of eulerian magnitude homology.

Throughout the paper we adopt the notation $[m] = \{1, \dots, m\}$ and $[m]_0 = \{0, \dots, m\}$ for common indexing sets.

\subsection{Graph terminology and notation}
We will write $G=(V,E)$ to denote an undirected, simple graph\footnote{As these are the only flavor of graph we will encounter, we will simply call these  ``graphs".}, with vertices $V$  and edges $E \subseteq {V \choose 2}$, unordered pairs of distinct vertices. Recall that a \emph{walk} in a graph $G$ is an ordered sequence of vertices $x_0,x_1,\ldots,x_k\in V$ such that there is an edge $\{x_i,x_{i+1}\}\in E$ for all $i \in [k]_0$, and a \emph{path} is a walk with no repeated vertices. A \emph{simple circuit} is a walk consisting of at least three vertices in $G$ for which $x_0 = x_k$ and there is no other repetition of vertices. We may view the set of vertices of a graph as an extended metric space by taking the \emph{path}  distance $d(u,v)$ to be equal to the length of a shortest path in $G$ from $u$ to $v$, if such a path exists, and taking $d(u,v) = \infty$ if $u$ and $v$ lie in different components of $G$. 

If $W \subseteq V,$ the \emph{full} subgraph of $G$ on $W$, written $G\vert_W,$ is the subgraph containing all edges of $G$ supported on $W$. 
Given a graph $H = (V', E'),$ we write $c(G, H)$ for the  number of full subgraphs of $G$  isomorphic to $H,$
    \[c(G, H) = \vert\{W \subseteq V\;\colon\; G\vert_{W} \cong H\}\vert.\]
Finally, write $\Delta_k$ for the complete graph on $k$ vertices, or $\Delta_V$ for the complete graph on a given set $V$ of vertices. A \emph{clique} in $G$ is a complete full subgraph of $G,$ and a $k$-clique is a clique supported on $k$ vertices.

\begin{definition}\label{def:ktrail}
Let $G = (V,E)$ be a graph, and $k$ a non-negative integer. A \emph{$k$-trail} $\bar{x}$ in $G$ is a $(k+1)$-tuple $(x_0,\dots,x_k) \in V^{k+1}$ of vertices for which $x_i \neq x_{i+1}$ and $d(x_i,x_{i+1})<\infty$ for every $i \in [k-1]_0$.
The \emph{length} of a $k$-trail $(x_0,\dots,x_k)$ in $G$ is defined as the minimum length of a walk that visits $x_0,x_1,\ldots,x_k$ in this order:
\[
\len(x_0,\dots,x_k) = d(x_0,x_1)+\cdots + d(x_{k-1},x_k).
\]
We call the vertices $x_0, \dots x_{k}$ the \emph{landmarks}, $x_0$ the \emph{starting point}, and $x_k$ the \emph{ending point} of the $k$-trail. Given a $k$-trail $\bar{x},$ write $L(\bar{x}) = \{x_0, \dots, x_k\}$ for the corresponding set of landmarks. Write $T_{k,\ell}(G)$ for the collection of all $k$-trails on $G$ of length $\ell.$
\end{definition}

\subsection{Magnitude homology}

The two parameters $k$ and $\ell$ for trails stratify walks in $G$ that pass through given sequences of vertices, and we can use this stratification to define a bigraded chain complex.

\begin{definition}[{\cite[Def 2.2]{hepworth2015categorifying}}]
    \label{differential}
    Given a graph $G$, the \emph{$(k,\ell)$-magnitude chain group}, $MC_{k,\ell}(G) = \mathbb{Z}\langle T_{k,\ell}(G)\rangle,$ is the free abelian group generated by $k$-trails in $G$ of length $\ell$.

    Denote by $(x_0,\dots,\hat{x_i},\dots,x_k)$ the $k$-tuple obtained by removing the $i$-th vertex from the $(k+1)$-tuple $(x_0,\dots,x_k)$.  We define the \emph{differential}
	\[
	\partial_{k,\ell}: MC_{k,\ell}(G) \to MC_{k-1,\ell}(G)
	\]
	to be the signed sum $\partial_{k,\ell}= \sum_{i\in[k-1]}(-1)^{i}\partial_{k,\ell}^i$ of chains corresponding to omitting landmarks without shortening the walk or changing its starting or ending points,
	\[
	\partial_{k,\ell}^i(x_0,\dots,x_k) = \begin{cases}
		(x_0,\dots,\hat{x_i},\dots,x_k) , &\text{ if } \len(x_0,\dots,\hat{x_i},\dots,x_k) = \ell, \\
		0, &\text{ otherwise.}\\
	\end{cases}
	\]

	For a non-negative integer $\ell$, we obtain the \emph{magnitude chain complex}, $MC_{*,\ell}(G),$ given by the following sequence of free abelian groups and differentials
	\[
	\cdots \to MC_{k+2,\ell}(G) \xrightarrow{\partial_{k+2,\ell}} MC_{k+1,\ell}(G) \xrightarrow{\partial_{k+1,\ell}} MC_{k,\ell}(G) \xrightarrow{\partial_{k,\ell}} MC_{k-1,\ell}(G) \to \cdots
	\]
\end{definition}

It is shown in \cite[Lemma 11]{hepworth2015categorifying} that the composition $\partial_{k,\ell} \circ \partial_{k+1,\ell}$ vanishes, justifying the name ``differential", and allowing them to define the corresponding bigraded homology groups of a graph.

\begin{definition}[{\cite[Def 2.4]{hepworth2015categorifying}}]
	\label{def_MH}
	The \emph{$(k,\ell)$-magnitude homology group} of the graph $G$ is 
	\[
	MH_{k,\ell}(G) = H_k(MC_{*,\ell}(G)) = \frac{\ker(\partial_{k,\ell})}{\im(\partial_{k+1,\ell})}.
	\]
\end{definition}

The magnitude homology of a graph is a rich graph invariant. However, understanding what the groups tell us about the structure of the graph is not straightforward. Hepworth and Willerton \cite[Proposition 9]{hepworth2015categorifying} show that the first two non-trivial groups simply count elements of $G$: $MH_{0,0}(G)$ is the free abelian group on $V,$ and $MH_{1,1}(G)$ is the free abelian group on the set of \emph{oriented} edges of $G.$ 

It is also straightforward to demonstrate that magnitude homology  vanishes when the length of the path is too short to support the necessary landmarks.

\begin{lemma}[{c.f. \cite[Proposition 10]{hepworth2015categorifying}}]
\label{lem:LowerTriangular}
Let $G$ be a graph, and $k > \ell$ non-negative integers. Then $MC_{k,\ell}(G) \cong 0.$
\end{lemma}

\begin{proof}
    Suppose $MC_{k,\ell}(G)\neq 0.$ Then, there must exist a $k$-trail $(x_0,\dots,x_k)$ in $G$ so that $\len(x_0,\dots,x_k)=d(x_0,x_1)+\cdots+d(x_{k-1},x_k) = \ell$.
	However, as consecutive vertices in a $k$-trail must be distinct, $d(x_i,x_{i+1}) \geq 1$ for $i \in [k-1]_0$, so $k$ can be at most $\ell$.
\end{proof}

We will make extensive use of the following immediate consequence of this result.

\begin{cor}
	\label{cor:homisker}
    Let $G$ be a graph and $k$ a non-negative integer. Then $MH_{k,k}(G) \cong ker(\partial_{k,k}).$
\end{cor}

However, even for uncomplicated graphs, those magnitude homology groups which do not vanish can be quite intricate.

\begin{example}
	\label{toyexampleMH}
	 Consider the graph $G$ in Figure \ref{fig:toyexampleMH}. We will compute $MH_{2,2}(G).$
	
	\begin{figure}[t]
		
		\centering
		\begin{tikzpicture}[node distance={15mm}, thick, main/.style = {draw, circle}]
			\node[main] (0) {$0$}; 
			\node[main] (1) [right of=0] {$1$};  
			\node[main] (2) [above of=1] {$2$};
			\node[main] (3) [right of=1] {$3$};
			\draw (0) -- (1);
			\draw (1) -- (2);
			\draw (0) -- (2);
			\draw (2) -- (3);
		\end{tikzpicture} 
		
		\caption{This graph will be used in Examples \ref{toyexampleMH} and \ref{ex:toyexampleEMH} to compare computations of magnitude homology and eulerian magnitude homology.}
        \label{fig:toyexampleMH}
	\end{figure}
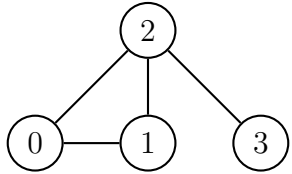
	
	$MC_{2,2}(G)$ is generated by the $2$-trails in $G$ of length $2$. There are eighteen such, consisting of all possible walks of length two in the graph: (0,1,0), (0,1,2), (0,2,0), (0,2,1), (0,2,3), (1,0,1), (1,0,2), (1,2,0), (1,2,1), (1,2,3), (2,0,1), (2,0,2), (2,1,0), (2,1,2), (2,3,2), (3,2,0), (3,2,1), (3,2,3).
    Similarly, $MC_{1,2}(G)$ is generated by $1$-trails in $G$ of length $2$. These are pairs of vertices for which the minimum length of a connecting path is 2, of which there are four: (0,3), (1,3), (3,0), (3,1).
    
    Because the differential $\partial_{2,2}$ consists of omitting only the center vertex, it is straightforward to check that it is surjective, and that the kernel is generated by the $14$ elements whose length diminishes when the middle vertex is removed; that is, all elements in $MC_{2,2}(G)$ except (0,2,3), (1,2,3), (3,2,0), (3,2,1).  

	On the other hand, by Lemma  \ref{lem:LowerTriangular}, $MC_{3,2}(G)$ is the trivial group, and thus the image of $\partial_{3,2}$ is $\langle 0 \rangle$. Therefore, $\rank(MH_{2,2}(G))=14,$ generated by those walks that do not have vertex 3 at exactly one endpoint. Of these, eight consist of walks back-and-forth across a single edge, while the remaining six are length 2 walks between vertices of the triangle with vertices 0, 1, and 2, and record the fact that there is a shorter path between the starting and ending points of the walks.
\end{example}

Cycles that record ``back-and-forth" trips across edges or similarly uninformative walks explode in number as the $k$ and $\ell$ grow. Indeed, as we will see in Theorem \ref{thm:relation MH EMH DMH}, in some regimes along the $k = \ell$ line, such cycles make up the entirety of the magnitude homology. On the other hand, the walks around the triangle detect structure in the graph akin to ``convexity" near the walk.

In comparison, in \cite{leinster2016magnitude} the authors consider the space $\ell_1^n=\mathbb{R}\times_1 \cdots \times_1 \mathbb{R}$ (i.e. $\mathbb{R}^n$ with the taxicab metric) and prove a connection between the magnitude of a convex body $A \subseteq \ell_1^n$ and some intrinsic geometric measures of $A$ such as volume.  This suggests that if we simplify the chain complex, it may be easier to interpret the resulting homology theory in terms of informative structure in the graph and relate that structure back to properties of metric spaces.

\section{Eulerian magnitude homology}
\label{sec:EMH}

In the magnitude chain complex, the differential of a single tuple $\boundary(x_0,\dots,x_k)$ vanishes precisely when $\len(x_{i-1},\hat{x_i},x_{i+1}) < \len   (x_{i-1},x_i,x_{i+1})$ for each $i$.
In other words, every landmark in the tuple enforces a longer walk than would otherwise be necessary based on the structure of the graph. So, the graph contains some smaller substructure \emph{witnessed} by this walk, which suggests there may be a meaningful relationship between the rank of magnitude homology groups of a graph and subgraph counting problems.

However, as we observed in Example \ref{toyexampleMH}, the relationship between the size of the magnitude homology groups and these structures is obscured by the fact that the constituent walks may revisit regions of the graph. For example, if $x_0$ and $x_1$ are adjacent in $G$, $(x_0,x_1,x_0,x_1,x_0) \in MC_{5,4}(G).$ As we will see in Theorem \ref{thm:relation MH EMH DMH}, cycles made of such chains can dominate $MH_{k,k}(G)$ under certain circumstances.

\subsection{Definition and intuition}
With this motivation, we introduce a natural sub-complex of $MC_{k,l}(G)$. 

\begin{definition}
    Let $G$ be a graph. We say a $k$-trail $(x_0, \dots, x_k) \in T_{k,\ell}(G)$ is \emph{eulerian} if $x_i \neq x_j$ for all $i \neq j.$ Denote the set of eulerian $k$-trails of $G$ by $ET_{k,\ell}(G)\subseteq T_{k,\ell}(G).$ 
    
    We define the \emph{$(k,\ell)$-eulerian magnitude chain group}, $EMC_{k,\ell}(G) = \mathbb{Z}\langle ET_{k,\ell}(G)\rangle,$ to be the free abelian group generated by eulerian $k$-trails $(x_0,\dots,x_k)$ of $G$ of length $\ell$. Throughout, we will mildly abuse notation by thinking of elements of $ET_{k,\ell}(G)$ as chains in $EMC_{k,\ell}(G).$
    
Regarding $EMC_{k,\ell}(G)$ as a subgroup of $MC_{k,\ell}(G)$, it follows directly from the definition that $\boundary(EMC_{k,\ell}(G)) \subseteq EMC_{k-1,\ell}(G),$ so $EMC_{*,\ell}(G)$ is the \emph{eulerian magnitude chain complex}, a sub-chain complex of the standard magnitude chain complex for each $k$. By abuse, we also will write $\partial_{k,\ell}$ for the restriction 
 $\partial_{k,\ell}\vert_{EMC_{k,\ell}(G)}$ unless it would create confusion.
\end{definition}

It is worth pausing to explicitly observe that the term ``eulerian" here is used to indicate that no \emph{landmark} is repeated. This does not, in general, imply that a minimal length walk in $G$ that visits those landmarks is eulerian; at best, it guarantees that the minimal walk between any two successive landmarks is distinct from all others. 

\begin{remark}
    However, in the special case $k = \ell$, the number of edges in a minimal walk is one less than the number of landmarks.
    Thus, as G is simple, there is a unique minimal walk that is entirely specified by its landmarks and so must, indeed, be eulerian (and hamiltonian).
\end{remark}

The following simple observation about the differential on the eulerian magnitude chains will drive a great deal of what we do in this paper.

\begin{lemma}
    \label{lem:diff_zero_so_edge}
    Let $G = (V, E)$ be a graph, $k \geq 2.$ Fix some $i \in [k-1]$ and $\bar{x} = (x_0, x_1, \dots, x_k)\in ET_{k,k}(G) \subseteq EMC_{k,k}(G).$ Then  $\partial^i_{k,\ell}(\bar{x}) = 0$ if and only if $\{x_{i-1}, x_{i+1}\}\in E.$     
\end{lemma}

\begin{remark}
Note, this result does not hold for standard magnitude homology, because of cycles like $(0,1,0)$ in Example \ref{fig:toyexampleMH}, for which  Lemma \ref{lem:diff_zero_so_edge} does not hold because there is no "edge" from vertex 0 to itself. In Section \ref{subsec:relationship}, we will investigate this discrepancy in more detail. For now, we simply note that throughout the paper, results which build on Lemma \ref{lem:diff_zero_so_edge} are in general true only in the eulerian case.
\end{remark}

\begin{proof}
        Let $G = (V, E)$ be a graph, $k \geq 2,$ $i \in [k-1],$ and $\bar{x} = (x_0, x_1, \dots, x_k)\in EMC_{k,k}(G).$ As $\len((x_0, \dots, x_k)) = k$ and each vertex is distinct, we have $d(x_j, x_{j+1}) = 1,$ so $\{x_j, x_{j+1}\}\in E$ for each $j \in [k-1]_0.$ As $\partial^i_{k,k}(\bar{x}) \in ET_{k-1,k}(G),$ if $\partial^i_{k,k}(\bar{x}) = 0,$ then $\len(\partial^i_{k,k}(\bar{x})) < \len(\bar{x}) = k,$ which can only occur if removing the landmark $x_i$ shortens the walk, which occurs precisely when $\{x_{i-1}, x_{i+1}\} \in E.$
\end{proof}

We now move on to our principal object of study.
    
\begin{definition}
The \emph{$(k,\ell)$-eulerian magnitude homology group} of a graph $G$ is
\[
EMH_{k,\ell}(G) = H_k(EMC_{*,\ell}(G)) = \frac{\ker(\partial_{k,\ell})}{\im(\partial_{k+1,\ell})}.
\]
\end{definition}

In this paper, we will focus our attention on the case $k = \ell$ to facilitate explicit descriptions of the relationship between the structure of the eulerian magnitude homology and that of the graph. However, many of the tools we develop should be applicable to the study of these groups in the case $k < \ell.$

\begin{example}
	\label{ex:toyexampleEMH}
	Consider again the graph $G$ of Figure \ref{fig:toyexampleMH}. We will compute $EMH_{2,2}(G)$ to compare with the computation in Example \ref{toyexampleMH}.
	
    The eulerian chain group $EMC_{2,2}(G)$ is generated by  ten $2$-trails: (0,1,2), (0,2,1), (0,2,3), (1,0,2), (1,2,0), (1,2,3), (2,0,1), (2,1,0), (3,2,0), (3,2,1). On the other hand, $EMC_{1,2}(G)= MC_{1,2}(G)$ is generated by the same four elements: (0,3), (1,3), (3,0), (3,1), and $EMC_{3,2}(G) \subseteq MC_{3,2}(G) = \langle 0 \rangle.$
	So, we have that $EMH_{2,2}(G)=\ker(\partial_{2,2}),$ and the group is generated by those six elements in $MC_{2,2}(G)$ which do not visit vertex $3$, and thus are precisely those six possible walks along the vertices of the triangle with vertices 0, 1, and 2. 
	
\end{example}

\begin{remark}
	\label{EMHsubgroup}
	Many of the definitions and properties regarding magnitude homology in \cite{hepworth2015categorifying} and \cite{leinster2021magnitude} carry over directly to eulerian magnitude homology, as it is defined on a subcomplex of the original chain complex.
	In particular, one can check that $EMH_{0,0}(G)\cong MH_{0, 0}(G)$ and $EMH_{1,1}(G)\cong MH_{1,1}(G),$ since the generators of the groups necessarily satisfy the condition of not revisiting vertices. Thus, these eulerian magnitude homology groups also count the number of vertices and edges in $G,$ respectively. 
\end{remark}

Naively, Lemma \ref{lem:diff_zero_so_edge} and Example \ref{ex:toyexampleEMH}, along with our experience with $EMH_{0,0}$ and $EMH_{1,1},$ may suggest that the rank of $EMH_{2,2}(G)$ should provide a count of triangles -- three vertex cliques -- in $G.$ And, indeed, generators of $EMC_{2,2}(G)$ are $2$-trails $(x_0,x_1,x_2)$ in $G$ of length 2. By Lemma \ref{lem:diff_zero_so_edge}, we see that if $\partial_{2,2}(x_0,x_1,x_2) = \partial_{2,2}^1(x_0, x_1, x_2) = 0,$ then $\{x_0, x_2\} \in E$ and $G\vert_{\{x_0, x_1, x_2\}} \cong C_3.$ Further, any $2$-trail given by a permutation of these three vertices gives rise to a generator of $EMC_{2,2}(G)$ with zero differential. Conversely, if for some subset $\{x_0, x_1, x_2\} \subseteq V$ we have $G\vert_{\{x_0, x_1, x_2\}} \cong C_3,$ we will have all six of these generators in $\ker(\partial_{2,2}) \subseteq EMC_{2,2}(G).$ 

However, this is not the complete story. For, if $\partial_{2,2}(x_0, x_1, x_2) = (x_0, x_2),$ then for any trail $(x_0, x_3, x_2) \in ET_{2,2}(G)$ with $x_1 \neq x_3,$ we also have $\partial_{2,2}(x_0, x_3, x_2) = (x_0, x_2).$ Thus, $(x_0. x_1, x_2) - (x_0, x_3, x_2) \in \ker(\partial_{2,2}).$ However, the same will hold for any choice of another trail, $(x_0, x_3', x_2) \in ET_{2,2}(G).$

\begin{definition}
    \label{def:fan subgraph}
    Let $G= (V,E)$ be a graph and $x_i, x_j \in V$ so that $\{x_i,x_j\}\not \in E$.
    Let $V_{x_i,x_j} = \{v \in V \;\vert\; (x_i, v, x_j)\in ET_{2,2}(G)\} \subset V$. We call the full subgraph of $G$ induced by the vertex set $\{x_i,x_j\}\cup V_{x_i,x_j} \subset V$ the \emph{fan subgraph of $G$ at the pair $\{x_i,x_j\}$}, written $\text{Fan}_{\{x_i,x_j\}}$. 
    See Figure \ref{fig:fan subgraph} for an example.
\end{definition}

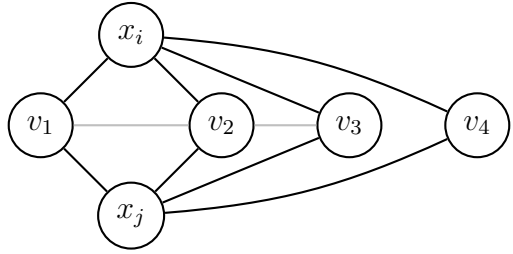
\begin{figure}[h]
        \centering
        \begin{tikzpicture}[node distance={17mm}, thick, main/.style = {draw, circle,minimum size=.85cm}]
	\node[main] (u1) {$v_1$}; 
	\node[main] (a) [above right of=u1] {$x_i$};  
	\node[main] (u2) [below right of=a] {$v_2$};
        \node[main] (b) [below right of=u1] {$x_j$};
        \node[main] (u3) [right of=u2] {$v_3$};
        \node[main] (u4) [right of=u3] {$v_4$};
            
	\draw (a) -- (u1);
	\draw (a) -- (u2);
        \draw (a) -- (u3);
        \path[every node/.style={font=\sffamily\small}]
			(a) edge [bend left=10] node {} (u4);
        \draw (b) -- (u1);
        \draw (b) -- (u2);
        \draw (b) -- (u3);
        \path[every node/.style={font=\sffamily\small}]
			(b) edge [bend right=10] node {} (u4);

        \draw[lightgray] (u1) -- (u2);
        \draw[lightgray] (u2) -- (u3);

	\end{tikzpicture} 
        \caption{Example of a fan subgraph induced by the vertex set $\{x_i, x_j, v_1,v_2,v_3,v_4\}$. Notice that any edges between the vertices $v_k$ may or may not be present; here $(v_1,v_2)$ and $(v_2,v_3)$ are drawn in grey to indicate they are incidental.}
        \label{fig:fan subgraph}
    \end{figure}

Now, for every pair of vertices $\{x_i,x_j\} \subseteq V$ so that $\{x_i, x_j\}\not\in E,$ there are $(|V_{x_i, x_j}| -1)$ linearly independent homology generators in $EMC_{2,2}(G)$ obtained by fixing $v \in V_{x_i, x_j}$ and allowing $v' \in V_{x_i, x_j} \setminus {v}$ to vary, producing cycles of the form $(x_i. v, x_j) - (x_i, v', x_j) \in \ker(\partial_{2,2}),$ $v,v' \in V_{x_i, x_j}$; other cycles of this form are linear combinations of those in this family. By symmetry, there are another $(|V_{x_i, x_j}| -1)$ generators for trails starting at $x_j$ and ending at $x_i.$ Letting the set $\{x_i, x_j\}$ vary, we obtain all other cycles in $EMH_{2,2}(G):$ 

\[
\text{rank}(EMH_{2,2}(G))=6\cdot c(G,C_3) + \sum_{\{x_i,x_j\}\in {V \choose 2}\setminus E} 2 \cdot (|V_{x_i, x_j}|-1).
\]

For each fan subgraph $\text{Fan}_{x_i,x_j}$ of $G,$ the full subgraph on vertices $\{x_i, x_j, v, v'\}$ is isomorphic either to $C_4$ or to \[F_4 = \left(\{v_0, v_1, v_2, v_3\}, \{\{v_0, v_1\}, \{v_0, v_3\}, \{v_0, v_2\}, \{v_1, v_2\}, \{v_2, v_3\}\}\right),\] the cycle graph on four vertices with one diagonal added; here, the added edge is $\{v, v'\}.$ In total, there are ${|V_{x_i, x_j}| \choose 2}$ such subgraphs. The difficulty in obtaining a complete count of these subgraphs comes from the fact that if the full subgraph supported on $\{x_i,x_j,v,v'\}$ is isomorphic to $C_4$, then $\text{Fan}_{x_i,x_j}$ and $\text{Fan}_{v,v'}$ intersect, and there is no general way to determine the intersection pattern of such subgraphs from a naive count of homology generators.

However, we can apply the above discussion, accounting for symmetry and the fact that $(n-1) \leq {n \choose 2}$ for $n \geq 2$, to provide a coarse bound on a subgraph count using $EMH_{2,2}(G).$
\begin{lemma}
\label{lem:emh22}
    Let $G = (V, E)$ be a graph. Write $C_3$ and $C_4$ for the cycle graphs on three and four vertices respectively, 
    Then 
    \[\rank(EMH_{2,2}(G)) \leq 6\cdot c(G, C_3) + 4\cdot c(G, C_4) + 2\cdot c(G, F_4).\]
\end{lemma}

\begin{remark}
    It is worth noting that Lemma \ref{lem:emh22} is analogous to some results proved in \cite{grigor2012homologies}, where the authors show a connection between the path homology groups of (di)graphs and substructure like cliques, binary hypercubes and other subgraphs reminiscent of polyhedra. However, the authors have not investigated any potential connections between these results.
\end{remark} 

Generalizing Lemma \ref{lem:emh22} to higher $k$ becomes difficult because of the increasingly complicated collection of isomorphism types of graphs which can support eulerian trails. 
A first result in this direction can be directly extracted from the proof of Lemma \ref{lem:emh22}. We observed that generators of $EMC_{2,2}(G)$ with zero differential were precisely those eulerian trails that walk around boundaries of 3-cliques. While this is not a complete characterization of such generators in $EMC_{k,k}(G)$ in general, it is the case that walks around cliques always have zero differential.

\begin{lemma}
	\label{lem:cliques bound}
    Let $G$ be a graph, and let $Z \subseteq ET_{k,k}(G)\subseteq EMC_{k,k}(G)$ be the collection of generators $\bar{x}$  for which $\partial_{k,k}(\bar{x}) = 0$. Write $\Delta_{k+1}$ for the complete graph on ${k+1}$ vertices, then
    \[c(G, \Delta_{k+1}) \leq \left\lfloor\frac{|Z|}{{k+1}!}\right\rfloor.\]
\end{lemma}

\begin{proof}
    Let $G$ and $Z$ be as in the statement of the lemma.   
	Observe that if vertices $x_0,\dots,x_k$ form a $(k+1)$-clique of $G,$ any $k$-trail passing through all nodes $x_0,\dots,x_k$ in any order has length $k$. Thus, for every permutation $\sigma \in \Sigma_k, (x_{\sigma(0)}, \dots, x_{\sigma(k)}) \in ET_{k,k}(G).$ Further, since all edges among the constituent vertices are present, removing any landmark from such a trail reduces its length. so $\partial_{k,k}(x_{\sigma(0)}, \dots, x_{\sigma(k)}) = 0.$ Thus, each such trail is an element of $Z,$ and the number of $(k+1)$-cliques in $G$ is bounded above by $\left\lfloor\frac{|Z|}{{k+1}!}\right\rfloor.$
\end{proof}

\begin{remark}
While these results provide us with some intuition about the meaning of the simplest classes in the eulerian magnitude homology groups, they also serve as a cautionary tale from a computational perspective. Any naive attempt to compute magnitude homology groups will run into clique enumeration problems, which grow exponentially in complexity with $k$ for even moderately sized $n.$ 
\end{remark}

One lesson that we should take away from this investigation is that the relationship between the combinatorics of the eulerian magnitude homology computation and that of subgraph counting are complicated by the fact that the presence or absence of some edges is irrelevant to computations in the differential. In the next section, we will develop language aimed at mitigating this difficulty.

\subsection{Families of graphs that support eulerian magnitude cycles}

We will now develop the language which we will need to fully describe the structure underlying nontrivial cycles in $EMH_{k,k}(G).$ To do so, we will study families of graphs that can support cycles of various forms. To this end, we introduce a useful way to describe families of graphs adopted from \cite{yu2009computing}.
We have chosen to slightly alter the terminology used in that paper to more clearly distinguish between individual graphs and collections thereof. 

\begin{definition}[c.f. \cite{yu2009computing}]
A \emph{class graph} $\mathcal{G} = (V,E_S,E_B)$ consists of a vertex set $V$ and two disjoint edge sets $E_S, E_B \subset {V \choose 2}.$ A \emph{complete} class graph is one for which $E_S \cup E_B = {V \choose 2}.$ 
\end{definition}

The sets $E_S$ and $E_B$ provide us with rules for constructing a family of graphs: edges in $E_S$ are mandatory, while edges in $E_B$ are optional, and edges that appear in neither $E_S$ nor $E_B$ cannot be included.

\begin{definition}
\label{def:alpha(G) and omega(G)}
Given a class graph $\mathcal{G} = (V, E_S, E_B)$, the set of \emph{graphs of class $\mathcal{G},$} $\Gamma(\mathcal{G}),$ consists of all graphs $G = (V, E)$ such that $E_S \subseteq E \subseteq E_S \cup E_B.$ Let $\mathcal{G}=(V,E_S,E_B)$ be a class graph. We denote the minimal and maximal graphs (under inclusion) in $\Gamma(\mathcal{G})$ as $\alpha(\mathcal{G})=(V,E_S)$ and $\omega(\mathcal{G})=(V,E_S \cup E_B)$.
\end{definition}

We visualize a class graph $\mathcal{G}$ by drawing edges in $E_S$ as solid lines and those in $E_B$ as dashed lines; graphs in $\Gamma(\mathcal{G})$ must have all solid edges and may have any dashed edges, but can have no other edges. See Figure \ref{fig:example representation class graph} for relevant examples.

\begin{figure}[h]
\centering
\begin{minipage}{0.18\linewidth}
\centering
    \begin{tikzpicture}[node distance={15mm}, thick, main/.style = {draw, circle}]
 		\node[main] (0) {}; 
 		\node[main] (1) [below of=0] {};  
 		\node[main] (2) [right of=0] {};
 		\node[main] (3) [below of=2] {};
	
 		\draw[dashed] (0) -- (1);

 		\draw[dashed] (0) -- (2);	
            \draw[dashed] (1) -- (2);
 		\draw (1) -- (3);
 		\draw[dashed] (2) -- (3);
 						
    \end{tikzpicture} 
\end{minipage} 
\hfill
\begin{minipage}{0.18\linewidth}
    \centering
    \begin{tikzpicture}[node distance={15mm}, thick, main/.style = {draw, circle}]
  		\node[main] (0) {}; 
 		\node[main] (1) [below of=0] {};  
 		\node[main] (2) [right of=0] {};
 		\node[main] (3) [below of=2] {};
	
 		\draw (1) -- (3);
 						
    \end{tikzpicture} 

\end{minipage} 
\hfill
\begin{minipage}{0.18\linewidth}
    \centering
    \begin{tikzpicture}[node distance={15mm}, thick, main/.style = {draw, circle}]
 		\node[main] (0) {}; 
 		\node[main] (1) [below of=0] {};  
 		\node[main] (2) [right of=0] {};
 		\node[main] (3) [below of=2] {};
	
 		\draw (0) -- (1);
 		\draw (0) -- (2);	
            \draw (1) -- (2);
 		\draw (1) -- (3);
 		\draw (2) -- (3);
 						
    \end{tikzpicture} 
\end{minipage} 
\hfill
\begin{minipage}{0.18\linewidth}
    \centering
    \begin{tikzpicture}[node distance={15mm}, thick, main/.style = {draw, circle}]
 		\node[main] (0) {}; 
 		\node[main] (1) [below of=0] {};  
 		\node[main] (2) [right of=0] {};
 		\node[main] (3) [below of=2] {};
	
 		\draw (0) -- (1);
 		\draw (0) -- (2);	
 		\draw (1) -- (3);
 		\draw (2) -- (3);
 						
    \end{tikzpicture} 
\end{minipage} 
    \caption{(far left) A class graph $\mathcal{G}$. (middle left) the minimal graph $\alpha(\mathcal{G})\in \Gamma(\mathcal{G})$. (middle right) the maximal graph $\omega(\mathcal{G})\in \Gamma(\mathcal{G})$. (far right)
    a graph $G \in \Gamma(\mathcal{G})$.}
\label{fig:example representation class graph}
\end{figure}
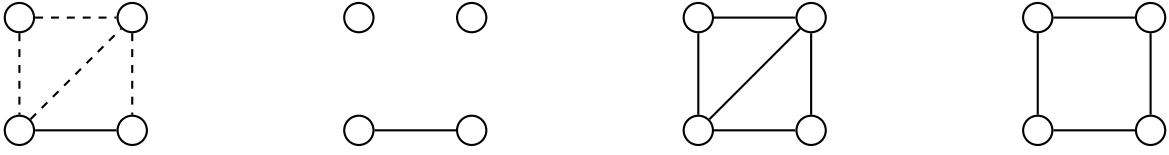

With this language, we can characterize which graphs support cycles in $EMC_{k,k}(G).$ We begin with the simple case of a single generator with zero differential. Recall from Definition \ref{def:ktrail} that for a $k$-trail $\bar{x}$, $L(\bar{x})$ is the corresponding set of landmarks. In the current setting, where $k = \ell,$ we have that $L(\bar{x})$ is precisely the set of vertices in the unique corresponding minimal walk in $G$, so in this case we will often call $L(\bar{x})$ the \emph{support} of $\bar{x}.$

\begin{lemma}
    \label{lem:class_graph_one_gen}
    Let $G=(V,E)$ be a graph. Suppose  $\bar{x} = (x_0, \dots x_k) \in ET_{k,k}(G) \subseteq EMC_{k,k}(G)$ has $\partial_{k,k}\bar{x} = 0.$ Consider the complete class graph \[\mathcal{H}(\{\bar{x}\}) = \left(L(\bar{x}), E_S = \{\{x_i, x_{i+1}\}\}_{i\in[k-1]_0} \cup \{\{x_{i-1}, x_{i+1}\}\}_{i\in[k-1]}, E_B = {L(\bar{x}) \choose 2} \setminus E_S\right).\] Then $G\vert_{L(\bar{x})} \in \Gamma(\mathcal{H}(\{\bar{x}\})).$ 
\end{lemma}

\begin{proof}
	Let $G=(V,E)$ be a graph, and $\bar{x} = (x_0, \dots x_k) \in EMC_{k,k}(G)$ a $k$-trail. So, every edge $\{x_i, x_i+1\}, i \in [k-1]_0$ is in $E$. By Lemma \ref{lem:diff_zero_so_edge}, for every $i\in [k-1]$ such that $\partial_{k,k}^i(\bar{x}) = 0$, the edge $\{x_{i-1},x_{i+1}\}$ is in $E$. Thus, if $\partial_{k,k}(\bar{x}) = 0$,  we have $\alpha(\mathcal{H}(\{\bar{x}\})) \subseteq G\vert_{L(\bar{x})},$ so $G\vert_{L(\bar{x})} \in \Gamma(\mathcal{H}(\{\bar{x}\})).$ \qedhere	    
\end{proof}

\begin{figure}[h]
    \begin{minipage}{0.45\linewidth}
	\centering
	\begin{tikzpicture}[node distance={15mm}, thick, main/.style = {draw, circle}]
			\node[main] (0) {$x_0$}; 
			\node[main] (1) [above right of=0] {$x_1$};  
			\node[main] (2) [below right of=1] {$x_2$};
			\node[main] (3) [above right of=2] {$x_3$};
			\node[main] (4) [below right of=3] {$x_4$};
			\node[main] (5) [above right of=4] {$x_5$};
			\draw (0) -- (1);
			\draw (1) -- (2);
			\draw (2) -- (3);
			\draw (3) -- (4);
			\draw (4) -- (5);
			
			\draw[lightgray] (0) -- (2);
			\draw[lightgray] (1) -- (3);
			\draw[lightgray] (2) -- (4);
			\draw[lightgray] (3) -- (5);

			\draw[dashed] (0) -- (3);
			\draw[dashed] (1) -- (4);
            \draw[dashed] (2) -- (5);

           \path[every node/.style={font=\sffamily\small}]
			(0) edge [bend right, dashed] node {} (4);
            \path[every node/.style={font=\sffamily\small}]
			(0) edge [bend right=75, dashed] node {} (5);
            \path[every node/.style={font=\sffamily\small}]
			(1) edge [bend left, dashed] node {} (5);
    \end{tikzpicture} 
    \end{minipage}
    \begin{minipage}{0.45\linewidth}
		\centering
		\begin{tikzpicture}[node distance={15mm}, thick, main/.style = {draw, circle}]
			\node[main] (0) {$x_0$}; 
			\node[main] (1) [above right of=0] {$x_1$};  
			\node[main] (2) [below right of=1] {$x_2$};
			\node[main] (3) [above right of=2] {$x_3$};
			\node[main] (4) [below right of=3] {$x_4$};
			\node[main] (5) [above right of=4] {$x_5$};
			\draw (0) -- (1);
			\draw (1) -- (2);
			\draw (2) -- (3);
			\draw (3) -- (4);
			\draw (4) -- (5);
			
			\draw[lightgray] (0) -- (2);
			\draw[lightgray] (1) -- (3);
			\draw[lightgray] (2) -- (4);
			\draw[lightgray] (3) -- (5);

        \end{tikzpicture} 
    \end{minipage}
    
    \caption{(left) The class graph $\mathcal{H} = \mathcal{H}(\{(x_0, \dots, x_5)\}).$ (right) The minimal element $\alpha(\mathcal{H}) \in \Gamma(\mathcal{H}).$ We have  $\partial_{5,5}(x_0,\dots,x_5) = 0$ in $EMC_{k,k}(G),$ precisely when $G\vert_{\{x_0, \dots, x_5\}} \in \Gamma(\mathcal{H}).$ Black edges lie in the support of the trail, while gray edges must be present to force terms in the differential to be zero. Dashed edges do not play a role in the computation of the differential.}
    \label{fig:subgraph_emh}
\end{figure}
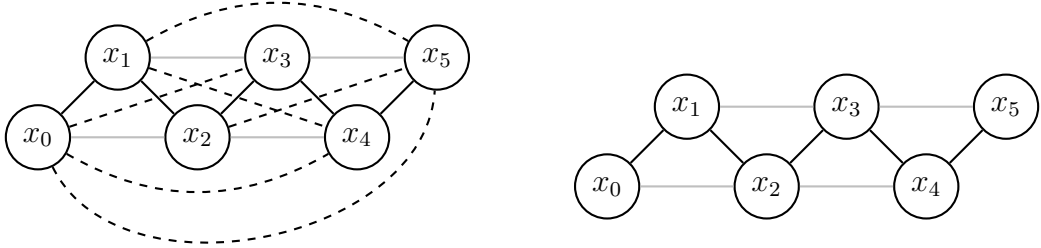

We now consider the somewhat more intricate combinatorics that arise for linear combinations of generators whose differentials cancel. The first non-trivial case, described in the following example, is instructive. 

\begin{example}\label{ex:two_trail_cycle_graph}
    Let $G=(V, E)$ be a graph, and let $\bar{x}^i = (x_0^i, \dots, x_k^i), i=1, 2,$ be trails in $\Delta_V.$ We now extend the construction of $\mathcal{H}(\{\bar{x}\})$ from Lemma \ref{lem:class_graph_one_gen} to construct a class graph $\mathcal{H}(\{\bar{x}^1, \bar{x}^2\}) = (W^{\{1,2\}}, E_S^{\{1,2\}}, E_B^{\{1,2\}})$ so that if $G\vert_{W^{\{1,2\}}} \in \mathcal{H}(\{\bar{x}^1, \bar{x}^2\}),$ then $\bar{x}^1$ and $\bar{x}^2$ are generators of $EMC_{k,k}(G)$ for which  $\partial_{k,k}(\bar{x}^i) \neq 0, i=1, 2$ but $\partial_{k,k}(\bar{x}^1 - \bar{x}^2) = 0.$
    
    From the definition of the differential, if $\bar{x}^1 \neq \bar{x}^2$ and $\partial_{k,k}(\bar{x}^1 - \bar{x}^2) = 0$ then both trails agree in all landmarks except one, say $x_r^1 \neq x_r^2$ for some $r \in [k-1],$ and the vertex $x_r^2$ cannot appear as a landmark in $\bar{x}^1$ nor vice versa.
    Indeed, suppose they differ in two landmarks, say $x_r^1 \neq x_r^2$ and $x_s^1 \neq x_s^2$ for some $r,s \in [k-1]$, and say $r<s$.
    Then if we compute the boundary $\partial_{k,k}(\bar{x}^1 - \bar{x}^2)$ we get
    \begin{align*}
        \partial_{k,k}(\bar{x}^1 - \bar{x}^2) &= \sum_{i=1}^{k-1} (-1)^i \partial_{k,k}^i(\bar{x}^1 - \bar{x}^2) \\
        &= (-1)^r \partial_{k,k}^r(\bar{x}^1 - \bar{x}^2) + (-1)^s \partial_{k,k}^s(\bar{x}^1 - \bar{x}^2) \\
        &= (-1)^r ((x_0^1,\dots, \hat{x}_r^1, \dots,x_s^1, \dots, x_k^1 ) - (x_0^2,\dots, \hat{x}_r^2, \dots,x_s^2, \dots, x_k^2 )) \\
        &+ (-1)^r ((x_0^1,\dots, x_r^1, \dots,\hat{x}_s^1, \dots, x_k^1 ) - (x_0^2,\dots, x_r^2, \dots,\hat{x}_s^2, \dots, x_k^2 )) \\
        &\neq 0,
    \end{align*}
    because we assumed $x_r^1 \neq x_r^2$ and $x_s^1 \neq x_s^2$, and thus the subtuples do not simplify.
     Note that, in particular, if this differential vanishes then $x_0^1 = x_0^2$ and $x_k^1 = x_k^2,$ so the trails have the same starting and ending points.
    
    Let $\mathcal{H}(\{\bar{x}^1\}) = (W^{\{1\}}, E^{\{1\}}_S, E^{\{1\}}_B)$ be the class graph of Lemma \ref{lem:class_graph_one_gen}. The set $E_S$ necessarily contains all of the edges in the support of both trails except $\{x_{r-1}^1, x_r^2\}$ and $\{x_r^2, x_{r+1}^1.\}$
    Further, per the proof of Lemma \ref{lem:class_graph_one_gen}, the edges already present in $\mathcal{H}(\{\bar{x}^1\})$ enforce $\partial_{k,k}^i(\bar{x}^1) = 0, i \neq r,$ and $\partial_{k,k}^i(\bar{x}^2) = 0$ for $i \neq r-1, r, r+1,$ due to agreement of the trails away from these vertices. However, we must introduce new edges to ensure $\partial_{k,k}^{r-1}(\bar{x}^2) = \partial_{k,k}^{r+1}(\bar{x}^2) = 0.$ So, we define $\mathcal{H}(\{\bar{x}^1, \bar{x}^2\}) = (W^{\{1,2\}}, E_S^{\{1,2\}}, E_B^{\{1,2\}}),$ where 
    \begin{align*}W^{\{1,2\}} = &W^{\{1\}} \cup \{x_r^2\}\\E_S^{\{1,2\}} = &\left(E_S^{\{1\}} \cup \{\{x_{a}^1, x_r^2\} \;\colon\; a \in \{r-2, r-1, r+1, r+2\} \cap [k]_0\}\right)\setminus \{ \{x_{r-1}^1, x_{r+1}^1\}\}\\
    E_B^{\{1,2\}}=& {W^{\{1,2\}} \choose 2} \setminus \left(E_S^{\{1,2\}} \cup \{\{x_{r-1}^1, x_{r+1}^1\}\}\right) \end{align*}

    See Figure \ref{fig:other kernel EMH_2,2} for an illustration of the portion of this new class graph that differs from $\mathcal{H}(\{\bar{x}^1\}).$ The new vertex $x_r'$ and the two new edges $\{x_{r-1}^1, x_r^2\}$ and $\{x_{r+1}^1, x_r^2\}$ support the trail $\bar{x}^2$ and enforce the agreement of $\partial_{k,k}^r$ on the two generators. The other one or two new edges are diagonals in newly introduced subgraphs isomorphic to $C_4,$ and so are added to enforce that $\partial_{k,k}^{r-1}(\bar{x}^i)=0$ and $\partial_{k,k}^{r+1}(\bar{x}^i)=0,$ as needed. Finally, by Lemma \ref{lem:diff_zero_so_edge}, the edge $\{x_{r-1}, x_{r+1}\}$ must be absent from $G$ to ensure that $\partial_{k,k}^r(\bar{x}^1) = \partial_{k,k}^r(\bar{x}^2) \neq 0$. Per Lemma \ref{lem:class_graph_one_gen}, all other terms in the differential of both chains are zero due to edges in $E_S^{\{1\}}.$ 
    \label{ex:class_graph_construction}
\end{example}
   
\begin{figure}[h]
\begin{minipage}{\linewidth}
    \centering
	\begin{tikzpicture}[node distance={15mm}, thick, main/.style = {draw, circle,minimum size=1.2cm}]		
		\node[main] (1) {$x_{r-2}$};  
		\node[main] (2) [below right of=1] {$x_{r-1}$};
		\node[main] (3) [above right of=2] {$x_r$};
		\node[main] (4) [below right of=3] {$x_{r+1}$};
		\node[main] (5) [above right of=4] {$x_{r+2}$};
		\node[main] (6) [below right of=2] {$x_r'$};

		\draw (1) -- (2);
		\draw (2) -- (3);
		\draw (3) -- (4);
           \draw (4) -- (5);
			
		\draw[black] (4) -- (6);
		\draw[black] (2) -- (6);
		\draw[lightgray] (1) -- (3);
		\draw[lightgray] (3) -- (5);
		\draw[dashed] (3) -- (6);
			
		\path[every node/.style={font=\sffamily\small}]
			(1) edge [bend right=50, lightgray] node {} (6);
	    \path[every node/.style={font=\sffamily\small}]	
            (6) edge [bend right=50, lightgray] node {} (5);
        \path[every node/.style={font=\sffamily\small}]
			(2) edge [bend left=75, dashed] node {} (5);
	    \path[every node/.style={font=\sffamily\small}]
			(1) edge [bend left=75, dashed] node {} (4);
	    \path[every node/.style={font=\sffamily\small}]
			(1) edge [bend left=60, dashed] node {} (5);

	\end{tikzpicture} 
\end{minipage}
	\captionof{figure}{Neighborhood $U \subseteq W^{\{1,2\}}$ of the vertices $x_r$ and $x_r'$ in the class graph $\mathcal{H}(\{\bar{x}^1, \bar{x}^2\}).$ If $G|_W \in \Gamma(\mathcal{H}(\{\bar{x}^1, \bar{x}^2\})),$ then $\partial_{k,k}\left(\bar{x}^1 - \bar{x}^2\right) = 0$. Outside of this neighborhood, the class graph is identical to $\mathcal{H}(\bar{x})$ from Figure \ref{fig:subgraph_emh}. Black edges lie in the support of the trail, while gray edges must be present to force terms in the differential to be zero.}
	\label{fig:other kernel EMH_2,2}
\end{figure}
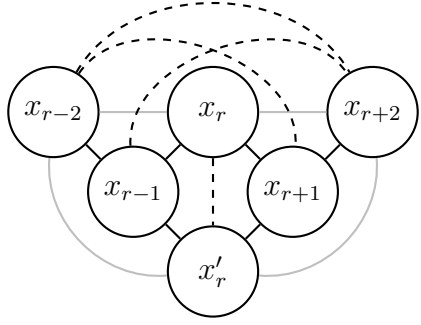

Our goal now is to generalize the construction in Example \ref{ex:class_graph_construction} to characterize classes of subgraphs of a graph $G$ which can support cycles in $EMC_{k,k}(G).$ To do so, we will first develop a decomposition of these cycles via a manageable spanning set.

\begin{definition}
\label{def:class_graph_construction}
    Let $G$ be a graph and fix integers $k, m, n \geq 1$. 
    Consider a collection of eulerian $k$-trails $X= \{\bar{x}^i\}_{i \in [m]}\subseteq ET_{k,k}(G).$ We say such a collection $X$ is \emph{local} if all of the $\bar{x}^i$ have the same starting and ending points, and for each $i\in[m]$ there are $j \in [m]\setminus \{i\}$ and $r \in [k-1]$ so that $x_r^i \neq x_r^j$ and $x_a^i = x_a^j$ for all $a \in [k-1]\setminus\{r\}.$
\end{definition}

Trails in local collections are those which can potentially share terms in their differentials, so we can hope to construct linear combinations of these trails which are cycles. However, such linear combinations may decompose into cycles supported on smaller local collections of trails.  We will need to decompose cycles as much as possible to obtain our desired spanning set.

\begin{definition}\label{def:minsupp}
    Let $G$ be a graph and suppose $X = \{\bar{x}^i\}_{i \in [m]}\subseteq ET_{k,k}(G).$ If there are coefficients $a_i \in \mathbb{Z}\setminus\{0\}, i \in [m]$ so that $d_k(\sum_{i \in I}a_i\bar{x}^i) = 0,$ we say that $\gamma = \sum_{i \in I}a_i\bar{x}^i$ is a cycle \emph{supported on X}. If, in addition, there is no $J \subsetneq [m]$ so that $\partial_{k,k}(\sum_{j \in J}a_j\bar{x}^j) = 0,$ we say that $\gamma$ is an \emph{$X$-minimal cycle for $G$}, and that $\gamma$ is \emph{minimally supported on $X$.}
\end{definition}

Now, we will take the (perhaps discomfiting) step of introducing an auxiliary graph that encodes structure in cycles in $EMC_{k,k}(G)$  supported on local collections of trails, and studying its structure.

\begin{definition}\label{def:structuregraph}
    Let $G$ be a graph, fix integers $k, m, n \geq 1$, and let
    $X= \{\bar{x}^i\}_{i \in [m]}$ be a local collection of eulerian $k$-trails in $G.$ Define the \emph{structure graph} for $X$, $s(X),$ to be the graph with vertices $X$ and an edge $\{\bar{x}^i, \bar{x}^j\}$ if there exists $r \in [k-1]$ so that $x_r^i \neq x_r^j$ and $x_a^i = x_a^j$ for all $a \in [k-1]\setminus\{r\}.$
\end{definition}

Suppose $X = \{\bar{x}^i\}_{i\in[m]}$ is a local collection of trails in $ET_{k,k}(G)$ for some graph $G$. The cliques and circuits in the structure graph $s(X)$ encode relations on coefficients of linear combinations of the trails in $X$  for which the differential $\partial_{k,k}$ can vanish. 
Observe first that since every edge in the structure graph indicates that the corresponding trails differ by precisely one landmark, a maximal clique in $s(X)$ corresponds to a maximal family $\{\bar{x}^j\}_{j\in J}\subseteq X, J \subseteq [m]$ which are pairwise identical except in their $r$th landmark for some fixed $r$. Thus, since we assumed vanishing differential, there is a single nonzero term in $\partial_{k,k}^r (\sum_{i\in[m]}a_i\bar{x}^i),$ corresponding to this family, and that term vanishes precisely when $\sum_{j\in J}a_j = 0.$ In the particular case when the maximal clique has two vertices -- that is, it is an edge in $s(X)$ which is not included in a larger clique -- this equation becomes $a_1 + a_2 = 0,$ so the  coefficients assigned to the two trails must differ by a sign. 

Knowing that the cliques in the structure graph determine the relations between coefficients in cycles supported on a local collection $X$ allows us to describe the class of graphs that support $X$-minimal cycles. To do so, for each maximal clique in $s(X)$, we assemble required and excluded edges as we did in Example \ref{ex:two_trail_cycle_graph} including a set $E_{supp}$ of the unions of the supports of the trails, a set $E_\text{diff}$ of the unions of the sets of edges which make each trail's differential cancel; and, a set $E_\text{rem}$ which indicates which edges must be missing from the graph so that trails have non-zero differentials indicated by the structure of $s(X)$. If a collection $X$ would cause a conflict between required and excluded edges, we cannot construct a graph which would minimally support the corresponding cycle, so we must exclude such collections. We collect all of these observations in Definition \ref{def:fullclassgraph}; Example \ref{ex:edge_sets} below describes the constituent edge sets for a particular graph and local collection $X$ of trails.

\begin{definition}
    \label{def:fullclassgraph}
    Fix integers $k, m \geq 1$. 
    Consider a local collection $X = \{\bar{x}^i\}_{i \in [m]}\subseteq ET_{k,k}(\Delta_n).$ Write $\bar{x}^i = (x_0^i, \dots, x_k^i).$ Let $\{Q_t = \{\bar{x}^j\}_{j\in J_t} \subseteq X\}_{t\in [r]}$ be the collection of supports of maximal cliques in $s(X),$ so $s(X)\vert_{Q_t}$ is a maximal clique of $s(X)$ for all $t \in [r].$ For each $t,$ write $b(t)$ for the landmark at which the constituent trails of $Q_t$ differ; that is, for all $i \neq j \in J_t, x_{b(t)}^i \neq x_{b(t)}^j$ and $x_s^i = x_s^j$ for all $s \in [k]_0 \setminus \{b(t)\}.$
    
    Now, write
    \begin{align*}
    E_{\text{supp}} &= \{\{x_a^i, x_{a+1}^i\}\;\colon\;a\in [k-1]_0, i \in [m]\}\\
    E_\text{diff} &=  \{\{x^{i}_a, x^{i}_{a+2}\}\;\colon\; a \in [k-2]_0, i \in [m]\}\\
    E_\text{rem} &= \{\{x^{j(t)}_{b(t)-1}, x^{j(t)}_{b(t)+1}\}\;\colon\;\text{for any choice of }j(t) \in J_t, \text{ for all }t \in [r]\}\\
    \end{align*}
    If $E_\text{supp} \cap E_\text{rem} = \varnothing,$ we say $X$ is \emph{compatible}, and define the \emph{minimal class graph supporting $X$,} written $\mathcal{H}(X) = (W^X, E_S^X, E_B^X),$ with the following elements:    
    \begin{align*}
        W^X &= \bigcup_{i \in [m]}\{x_0^i, \dots, x_k^i\}\\
        E_S^X &= E_\text{supp} \cup E_\text{diff} \setminus (E_\text{diff} \cap E_\text{rem})\\
        E_B^X &= {W \choose 2} \setminus E_\text{rem}
    \end{align*}
    
\end{definition}

\begin{example}\label{ex:edge_sets}
    For the graph in Figure~\ref{fig:example four sets of edges}, consider the local collection $X = \{\bar{x}^i\}_{i=1}^4$, with  $\bar{x}^1=(0,1,3,5,6)$, $\bar{x}^2=(0,1,4,5,6)$, $\bar{x}^3=(0,2,4,5,6)$ and $\bar{x}^4=(0,2,3,5,6).$ Applying Definition~\ref{def:fullclassgraph} to this collection, we have $k = m = 4$, while the structure graph $s(X)$ is isomorphic to $C_4,$ with $r=4$ maximal cliques. From this structure, we collect:
    
    \begin{align*}
        E_{\text{supp}} &=  \bigcup_{i\in [4]} \{ \{x_a^i, x_{a+1}^i\} : a \in [3]_0\} \\
        &= \{ \{0,1\}, \{1,3\}, \{3,5\}, \{5,6\}\} \cup \{ \{0,1\}, \{1,4\}, \{4,5\}, \{5,6\} \}  \\
        &\qquad\cup \{\{0,2\}, \{2,4\}, \{4,5\}, \{5,6\} \} \cup \{ \{0,2\}, \{2,3\}, \{3,5\}, \{5,6\} \} \\
        &= \{ \{0,1\}, \{0,2\}, \{1,3\}, \{1,4\}, \{2,3\}, \{2,4\}, \{3,5\}, \{4,5\}, \{5,6\} \}\\
        E_\text{diff} &= \bigcup_{i \in [4]} \{\{x^{i}_a, x^{i}_{a+2}\}\;\colon\; a \in [2]_0\}\\
        &= \{ \{0,3\}, \{1,5\}, \{3,6\} \} \cup \{ \{0,4\}, \{1,5\}, \{4,6\}\}  \\
        &\qquad\cup \{ \{0,4\}, \{2,5\}, \{4,6\}\} \cup \{ \{0,3\}, \{2,5\}, \{3,6\}\} \\
        &= \{ \{0,3\}, \{0,4\}, \{1,5\}, \{2,5\}, \{3,6\}, \{4,6\}  \} 
    \end{align*}
    \begin{align*}
        E_\text{rem} &= \bigcup_{i \in [4]} \{\{x^{j(t)}_{b(t)-1}, x^{j(t)}_{b(t)+1}\}\;\colon\;\text{for any choice of }j(t) \in J_t, \text{ for all }t \in [r]\}\\
        &= \{ \{0,3\}, \{1,5\}\} \cup \{\{0,4\}, \{1,5\}\} \cup \{\{0,4\}, \{ 2,5\}\} \cup \{\{0,3\}, \{2,5\}\} \\
        &= \{ \{0,3\}, \{0,4\}, \{1,5\}, \{2,5\}\} \\
    \end{align*}
    
    \begin{figure}[h]
			\centering
			\begin{tikzpicture}[node distance={20mm}, thick, main/.style = {draw, circle,minimum size=.85cm}]
	 	\node[main] (0) {$0$}; 
	 	\node[main] (1) [above right of=0] {$1$};  
		  \node[main] (2) [below right of=1] {$4$};
		  \node[main] (3) [above right of=2] {$5$};
		  \node[main] (1') [below right of=0] {$2$};
            \node[main] (2') [above right of=1] {$3$};
            \node[main] (4) [below right of=3] {$6$};
            
		  \draw (0) -- (1);
		  \draw (1) -- (2);
		  \draw (2) -- (3);
            \draw (0) -- (1');
		  \draw (1') -- (2);
		  \draw (1) -- (2');
		  \draw (2') -- (3);
            \draw (1') -- (2');
            \draw (3) -- (4);            

            \path[every node/.style={font=\sffamily\small},lightgray]
			(0) edge [bend left=33] node {} (2');
            \path[every node/.style={font=\sffamily\small},lightgray]
			(0) edge [bend right=20] node {} (2);
            \path[every node/.style={font=\sffamily\small},lightgray]
			(1) edge [bend right=20] node {} (3);
            \path[every node/.style={font=\sffamily\small},lightgray]
			(1') edge [bend right=33] node {} (3);

            \path[every node/.style={font=\sffamily\small},dashed, red]
			(0) edge [bend left=50] node {} (2');
            \draw[dashed, red] (0) -- (2);
            \draw[dashed, red] (1) -- (3);
            \path[every node/.style={font=\sffamily\small},dashed, red]
			(1') edge [bend right=50] node {} (3);

            \path[every node/.style={font=\sffamily\small},lightgray]
			(2') edge [bend left=33] node {} (4);
            \draw[lightgray] (2) -- (4);           
			
		  \end{tikzpicture} 

	\caption{Let $\bar{x}^1=(0,1,3,5,6)$, $\bar{x}^2=(0,1,4,5,6)$, $\bar{x}^3=(0,2,4,5,6)$ and $\bar{x}^4=(0,2,3,5,6)\}$. $E_{\text{supp}}$ is represented in black, 
 $E_{\text{diff}}$ in gray and $E_{\text{rem}}$ is dashed-red.}
	\label{fig:example four sets of edges}
\end{figure}
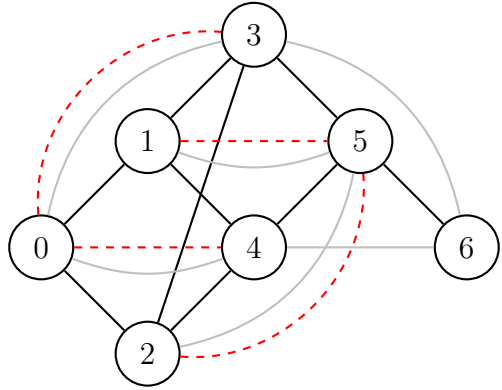

We see from the computations above that $E_\text{supp} \cap E_\text{rem} = \emptyset$.
Thus $X$ is a compatible collection and we can define the minimal class graph supporting $X$, $\mathcal{H}(X)$:

\begin{align*}
    W^X &= \bigcup_{i \in [4]} \{x_0^i,x_1^i,x_2^i,x_3^i,x_3^i\} \\
    &= \{0,1,2,3,4,5,6\} \\
    E_S^X &= E_\text{supp} \cup E_\text{diff} \setminus (E_\text{diff} \cap E_\text{rem})\\
    &= \{ \{0,1\}, \{0,2\}, \{1,3\}, \{1,4\}, \{2,3\}, \{2,4\}, \{3,5\}, \{3,6\}, \{4,5\}, \{4,6\}, \{5,6\} \} \\
    E_B^X &= {W \choose 2} \setminus E_\text{rem} \\
    &= {W \choose 2} \setminus \{ \{0,3\}, \{0,4\}, \{1,5\}, \{2,5\}\}.
\end{align*}
\end{example}

Observe that the class graph $\mathcal{H}(X)$ is ``minimal" in the sense that for a cycle minimally supported on $X$, every edge in $\alpha(\mathcal{H}(X))$ is either in $E_\text{supp}$, and so needed to support one of the trails, or in $E_\text{diff}$ but not $E_\text{rem}$, and so ensures that a boundary term that is not canceled by other trails is zero. In order to demonstrate that the family $X$ describes at least one element in $EMC_{k,k}(\alpha(\mathcal{H}(X))),$ we will first observe that the structure graph provides us with tools for decomposing such cycles into simpler elements. 

\begin{definition}\label{def:clique-tree}
    Let $G$ be a graph. A simple circuit \emph{of minimal-length} is a full subgraph of $G$ isomorphic to some cycle graph $C_k, k \geq 3$. A connected structure graph $G$ is a \emph{clique-tree} if every minimal-length simple circuit in $G$ has three vertices\footnote{Clique trees are chordal graphs. However, as structure graphs they have a restricted, tree-like form we describe below, which motivates our choice of terminology.}. A \emph{clique-forest} is a disjoint union of clique-trees.
\end{definition}

Observe that two maximal cliques in a structure graph cannot intersect in more than one vertex.
Indeed, we noticed that vertices $\{\bar{x}^j\}_{j \in J}$ of a maximal clique are all pairwise identical except in their $r$th landmark, for some fixed $r$.
Now, suppose by contradiction that two distinct maximal cliques, $Q_1$ and $Q_2$, share two vertices $\bar{x}^a$ and $\bar{x}^b$.
Notice that, because the two maximal cliques are distinct, the landmark in which vertices of $Q_1$ differ is different then the one in which vertices of $Q_2$ differ.
So, say $\{\bar{x}^j\}_{j \in J} \in Q_1$ differ in their $r$th landmark and $\{\bar{x}^j\}_{j \in J} \in Q_2$ differ in their $s$th landmark, $r \neq s$.
Now, because $\bar{x}^a, \bar{x}^b \in Q_1$ they both differ only in their $r$th landmark, and similarly because $\bar{x}^a, \bar{x}^b \in Q_2$ they both differ only in their $s$th landmark.
But then they differ in two landmarks, so by definition of structure graph they are not connected by edge, and thus they cannot be contained in the same maximal clique, which gives us a contradiction.

Therefore, clique-trees are connected graphs for which every simple circuit is contained in a maximal clique, and so successively collapsing the maximal cliques of three or more vertices into vertices will result in a tree. 

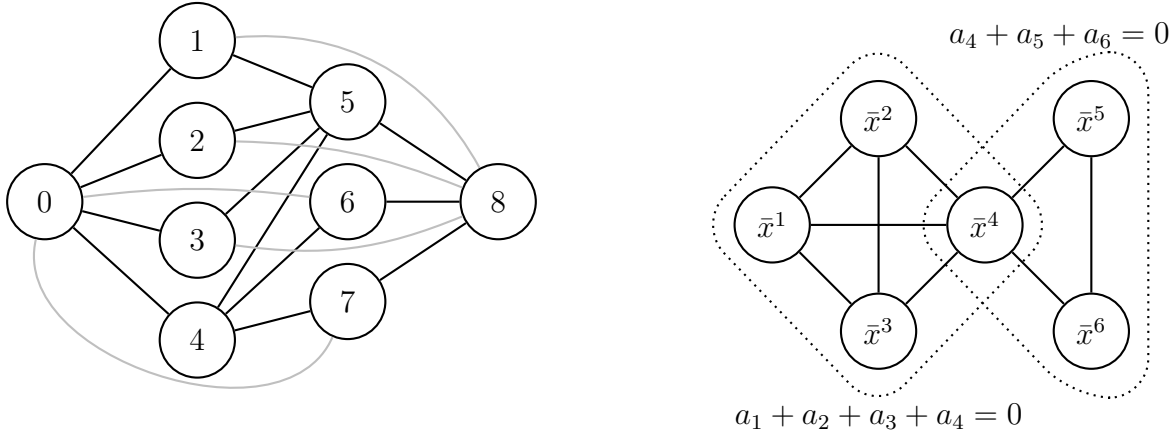
\begin{figure}[h]
	\begin{minipage}{\linewidth}
		\begin{minipage}{0.45\linewidth}
			\centering
			\begin{tikzpicture}[node distance={20mm}, thick, main/.style = {draw, circle,minimum size=1cm}]
	 	\node[main] (1) {$1$};  
		  \node[main] (2) [below=0.3cm of 1] {$2$};
  		\node[main] (3) [below=0.3cm of 2] {$3$};
		  \node[main] (4) [below=0.3cm of 3] {$4$};
            \node (level) [below=0.15cm of 1]	{};
            \node[main] (5) [right of=level] {$5$};
            \node[main] (6) [below=0.3cm of 5] {$6$};
            \node[main] (0) [left=3cm of 6] {$0$}; 
            \node[main] (7) [below=0.3cm of 6] {$7$};
            \node[main] (8) [right of=6] {$8$};

			\draw (0) -- (1);
			\draw (1) -- (5);
            \draw (0) -- (2);
			\draw (2) -- (5);
			\draw (0) -- (3);
			\draw (3) -- (5);
			\draw (0) -- (4);
			\draw (4) -- (5);
			\draw (4) -- (6);
			\draw (4) -- (7);
			\draw (8) -- (5);
			\draw (8) -- (6);
			\draw (8) -- (7);

                \path[every node/.style={font=\sffamily\small},lightgray]
			(1) edge [bend left=33] node {} (8);
                \path[every node/.style={font=\sffamily\small},lightgray]
			(2) edge [bend left=10] node {} (8);
                \path[every node/.style={font=\sffamily\small},lightgray]
			(3) edge [bend right=15] node {} (8);
                \path[every node/.style={font=\sffamily\small},lightgray]
			(0) edge [bend left=7] node {} (6);
                \path[every node/.style={font=\sffamily\small},lightgray]
			(0) edge [bend right=85] node {} (7);
            	
		\end{tikzpicture} 
	\end{minipage}
	\hspace{0.07\linewidth}
	\begin{minipage}{0.45\linewidth}
		\centering
		\begin{tikzpicture}[node distance={20mm}, thick, main/.style = {draw, circle,minimum size=1cm}]
			\node[main] (0) {$\bar{x}^1$}; 
			\node[main] (1) [above right of=0] {$\bar{x}^2$};  
			\node[main] (2) [below right of=1] {$\bar{x}^4$};
			\node[main] (3) [below right of=0] {$\bar{x}^3$};
			\node[main] (4) [below right of=2] {$\bar{x}^6$};
			\node[main] (5) [above right of=2] {$\bar{x}^5$};

            \node (sum1) [below=0.3cm of 3] {$a_1 + a_2 + a_3 + a_4 = 0$};
            \node (sum2) [above right=0.6 of 1] {$a_4 + a_5 + a_6 = 0$};

            \draw[rounded corners=15pt, dotted]
                (-1,0)  --
                ++(2.4,2.5)  -- 
                ++(2.4,-2.5)  -- 
                ++(-2.4,-2.5)  -- 
                cycle;

            \draw[rounded corners=15pt, dotted]
                (1.8,0)  --
                ++(2.2,2.3)  -- 
                ++(1,0)  -- 
                ++(0,-4.6)  -- 
                ++(-1,0)  -- 
                cycle;

			\draw (0) -- (1);
			\draw (1) -- (2);
			\draw (2) -- (3);
			\draw (0) -- (3);
            \draw (0) -- (2);
            \draw (1) -- (3);
            \draw (2) -- (4);
            \draw (2) -- (5);
            \draw (4) -- (5);
				
			\end{tikzpicture} 
			
		\end{minipage}
	\end{minipage}
	\caption{Write $\bar{x}^1=(0,1,5,8)$, $\bar{x}^2=(0,2,5,8)$, $\bar{x}^3=(0,3,5,8)$, $\bar{x}^4=(0,4,5,8)$, $\bar{x}^5=(0,4,6,8)$, $\bar{x}^6=(0,4,7,8)$. and  $X = \{\bar{x}^i\}_{i\in[6]}$. (left) In this graph $G \in \mathcal{H}(X)$, the cycle $\gamma = \bar{x}^1 - \bar{x}^2 + \bar{x}^3 - \bar{x}^4 - \bar{x}^5 + 2\bar{x}^6$ is minimally supported on $X.$ The set $E_\text{supp}$ is represented in black, while the edges in $E_\text{diff} \setminus (E_\text{diff} \cap E_\text{rem})$ are gray. (right) The structure graph $s(X)$ is a clique-tree, and its maximal cliques describe relations on the coefficients of supported cycles.}
	\label{fig:clique tree}
\end{figure}

It is straightforward to construct eulerian magnitude cycles minimally supported on clique-trees. Suppose that for some local, compatible collection $X = \{\bar{x}^i\}_{i\in[m]} \subseteq ET_{k,k}(\Delta_n),$ $G \in \mathcal{H}(X)$ and $s(X)$ is a clique-tree. Select any vertex $\bar{x}^i\in X$ and assign to it a non-zero coefficient $a_i$. For each maximal clique $\{\bar{x}^j\}_{j\in J},$ $J\subseteq[m]$ in $s(X)$ such that $i \in J,$ assign non-zero coefficients $a_j, j\in J \setminus \{i\}$ so that $\sum_{j \in J}a_j = 0.$ Repeat inductively for every maximal clique containing a vertex assigned a coefficient by this process. Since $s(X)$ is connected this process will assign a coefficient to every element of $X.$ By construction, the resulting linear combination will be a cycle in $EMC_{k,k}(G).$ In addition, the cycle is $X$-minimal because we have constructed $G$ to force other terms in the boundary to vanish, so terms in the differential corresponding to cliques in $s(X)$ cancel only due to the relations on the coefficients. See  Figure \ref{fig:clique tree} for an example. 

\medskip
We now come to one of the core results of this paper: given a general local, compatible collection $X$ of $k$-trails in $G$, we can decompose cycles minimally supported on $X$ into cycles minimally supported on clique-trees and cycles minimally supported on minimal circuits of $s(X).$ This decomposition bears more than a passing resemblance to finding a minimal spanning tree for $s(X),$ though here we will remove vertices in cycles rather than edges.

\begin{theorem}\label{thm:cycle_decomp}
    Let $G$ be a graph and $X = \{\bar{x}^i\}_{i\in[m]} \subseteq ET_{k,k}(G)$ a local, compatible collection of trails in $G.$ Suppose $\gamma\in EMC_{k,k}(G)$ is a cycle minimally supported on $X.$ Then there are a (possibly empty) family of local, compatible collections $Y_1, \dots Y_r \subseteq X$ and disjoint subsets $Z_a \subseteq Y_a, a\in[r]$ so that
    \begin{enumerate}
        \item $s(X)\vert_{Y_a}$ is a minimal-length simple circuit for each $a \in [r]$ with $|Y_a| \geq 4$ and even,
        \item writing $Y_T = X \setminus \left(\bigsqcup_{a\in[r]} Z_a\right)$, $s(X)\vert_{Y_T}$ is a clique-forest composed of disjoint local, compatible clique-trees $s(X)\vert_{Y_T^1}, \dots s(X)\vert_{Y_T^p}$, and
        \item there are cycles $\gamma_a \in EMC_{k,k}(G), a\in[r]$ minimally supported on $Y_a$ and $\gamma_T^b\in EMC_{k,k}(G), b\in[p]$ minimally supported on $Y_T^b,$ so that $\gamma = \sum_{a \in [r]}\gamma_a + \sum_{b \in [p]}\gamma_T^b.$
    \end{enumerate}
\end{theorem}

\begin{proof}

Fix a graph $G$ and $X = \{\bar{x}^i\}_{i\in[m]} \subseteq ET_{k,k}(G)$ a local, compatible collection of $k$-trails in $G,$ and let $\gamma = \sum_{i\in[m]}a_i\bar{x},$ $a_i \neq 0,$ be a cycle in $EMC_{k,k}(G)$ which is minimally supported on $X.$

Suppose $s(X)$ contains a minimal-length simple circuit involving four or more vertices, so we have $Y = \{\bar{x}^j\}_{j\in J}$ for $J \subseteq [m], |J| \geq 4$ such that $s(X)\vert_{Y} \cong C_{|J|}.$ Thus, each vertex $\bar{v}$ of $s(X)\vert_{Y}$ is implicated in two edges $(\bar{u},\bar{v})$ and $(\bar{v},\bar{w})$, and these edges cannot lie in the same maximal clique in $s(X)$ or the edge $(\bar{u},\bar{w})$ would also be present in the subgraph $s(X)\vert_{Y}$. 

Fixing $j_0 \in J$, the fact that the edge $(\bar{u},\bar{w})$ is not present in the subgraph $s(X)\vert_{Y}$ says that there are precisely two terms in the differential $\partial_{k,k}(\bar{x}^{j_0})$ which do not vanish, say $\partial_{k,k}^t(\bar{x}^{j_0})$ and $\partial_{k,k}^{t'}(\bar{x}^{j_0}),$ $t < t'.$ Thus, because $\gamma$ is a cycle and therefore $\partial(\gamma)=0$, from Example 3.2 we see that there are $\bar{x}^{j_1}$ and $\bar{x}^{j_2}$ in $Y$ so that $\bar{x}^{j_0}$ and $\bar{x}^{j_1}$ differ precisely in their $t$th landmark, and similarly $\bar{x}^{j_0}$ and $\bar{x}^{j_2}$ differ in their $t'$th landmark. Thus, $G$ contains the walks $\bar{x}^{j_0} = (x^{j_0}_0, \dots x^{j_0}_t, \dots , x^{j_0}_{t'}, \dots , x^{j_0}_k)$, $\bar{x}^{j_1} =(x^{j_0}_0, \dots x^{j_1}_t, \dots , x^{j_0}_{t'}, \dots , x^{j_0}_k)$,  and $\bar{x}^{j_2} =(x^{j_0}_0, \dots x^{j_0}_t, \dots , x^{j_2}_{t'}, \dots , x^{j_0}_k),$ and necessarily does not contain the edges $\{x^{j_0}_{t-1}, x^{j_0}_{t+1}\}$ and $\{x^{j_0}_{t'-1}, x^{j_0}_{t'+1}\}.$  
As before, because $s(X)\vert_{Y}$ is a cycle graph, each of the chains $\bar{x}^{j_1}$ and $\bar{x}^{j_2}$ has exactly two non-vanishing terms in its differential. And, because they share all landmarks except $x^{j_1}_t$ or $x^{j_2}_{t'}$ with $\bar{x}^{j_0},$ the missing edges in $G$ force those to be the $t$th and $t'$th terms of the differential, per Lemma \ref{lem:diff_zero_so_edge}. There are two possibilities: 

\begin{enumerate}
    \item The trail $\bar{x}^{j_3} =(x^{j_0}_0, \dots x^{j_1}_t, \dots , x^{j_2}_{t'}, \dots , x^{j_0}_k)$ is in $Y.$ In this case, by minimality of the circuit, $s(X)\vert_Y \cong C_4.$ An example is illustrated in Figure \ref{fig:square minimally supported on Y}.
    \item The set $Y$ contains two distinct trails $\bar{x}^{j_1'}=(x^{j_0}_0, \dots x^{j_1'}_t, \dots , x^{j_2}_{t'}, \dots , x^{j_0}_k)$ and $\bar{x}^{j_2'}=(x^{j_0}_0, \dots x^{j_1}_t, \dots , x^{j_2'}_{t'}, \dots , x^{j_0}_k)$ in order to cancel these terms. In this case, we have an equivalent pair of options again for canceling the corresponding differential terms, and this process proceeds inductively. As $|J|$ is finite, eventually we must select the option where a single trail shares both of the unresolved differential terms. As each step added an even number of trails to $Y$, we conclude that $|J|$ is even. Such a circuit is illustrated in Figure \ref{fig:even cycle minimally supported on Y}.
\end{enumerate}

\begin{figure}[h]
	\begin{minipage}{\linewidth}
		\begin{minipage}{0.45\linewidth}
			\centering
			\begin{tikzpicture}[node distance={15mm}, thick, main/.style = {draw, circle,minimum size=.75cm}]
	 	\node[main] (0) {$0$}; 
	 	\node[main] (1) [above right of=0] {$1$};  
			\node[main] (2) [below right of=1] {$4$};
			\node[main] (3) [above right of=2] {$5$};
			\node[main] (1') [below right of=0] {$2$};
            \node[main] (2') [above right of=1] {$3$};
			\draw (0) -- (1);
			\draw (1) -- (2);
			\draw (2) -- (3);
            \draw (0) -- (1');
			\draw (1') -- (2);
			\draw (1) -- (2');
			\draw (2') -- (3);
            \draw (1') -- (2');
			
		  \end{tikzpicture} 
			
		\end{minipage}
		\hfill
		\begin{minipage}{0.45\linewidth}
			\centering
			\begin{tikzpicture}[node distance={20mm}, thick, main/.style = {draw, circle,minimum size=0.75cm}]
				\node[main] (0) {$\bar{x}^1$}; 
				\node[main] (1) [above of=0] {$ \bar{x}^2$};  
				\node[main] (2) [right of=1] {$ \bar{x}^3$};
				\node[main] (3) [right of=0] {$ \bar{x}^4$};
				
				\draw (0) -- (1);
				\draw (1) -- (2);
				\draw (2) -- (3);
				\draw (0) -- (3);

        \draw[rounded corners=15pt, dotted]
            (-0.7,2.7)  --
            ++(1.4,0)  -- 
            ++(0,-3.4)  -- 
            ++(-1.4,0)  -- 
            cycle;
            
        \node (level)[above=0.45 cm of 0] {};
        \node (sum1)[left=0.8cm of level]{$a_2 = -a_1$};
            \end{tikzpicture} 

		\end{minipage}
	\end{minipage}
	\caption{Let $\bar{x}^1=(0,1,3,5)$, $\bar{x}^2=(0,1,4,5)$, $\bar{x}^3=(0,2,4,5)$ and $\bar{x}^4=(0,2,3,5)\}$, and let $Y = \{\bar{x}^1, \bar{x}^2, \bar{x}^3, \bar{x}^4\}.$ (left) A graph $G$ for which the cycle $\gamma = \bar{x}^1 - \bar{x}^2- \bar{x}^3+ \bar{x}^4$ is minimally supported on $Y$. (right) The structure graph $s(Y)$ is isomorphic to $C_4$. Its maximal cliques, the edges, induce alternating signs on the coefficients of the cycle $\gamma.$}
	\label{fig:square minimally supported on Y}
\end{figure}
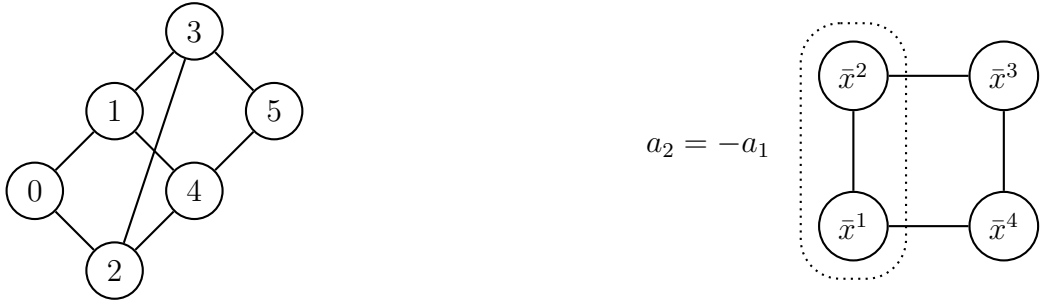

\begin{figure}[h]
	\begin{minipage}{\linewidth}
		\begin{minipage}{0.4\linewidth}
			\centering
			\begin{tikzpicture}[node distance={17mm}, thick, main/.style = {draw, circle,minimum size=.75cm}]
	 	\node[main] (0) {$0$}; 
	 	\node[main] (1) [above right of=0] {$1$};  
			\node[main] (2) [below right of=1] {$4$};
			\node[main] (3) [above right of=2] {$5$};
			\node[main] (4) [below right of=3] {$8$};
			\node[main] (1') [right=0.45cm of 0] {$2$};
            \node[main] (3') [right=0.45cm of 2] {$6$};
            \node[main] (1'') [below right of=0] {$3$};
            \node[main] (3'') [below right of=2] {$7$};
			\draw (0) -- (1);
			\draw (1) -- (2);
			\draw (2) -- (3);
			\draw (3) -- (4);
			
			\draw (0) -- (1');
			\draw (1') -- (2);
			\draw (2) -- (3');
            \draw (3') -- (4);
			\draw (0) -- (1'');
            \draw (0) -- (1'');
            \draw (1'') -- (2);
            \draw (2) -- (3'');
			\draw (3'') -- (4);

            \draw[lightgray] (1) -- (3);
            \draw[lightgray] (1') -- (3);
            \draw[lightgray] (1) -- (3');
            \draw[lightgray] (1') -- (3'');
            \draw[lightgray] (1'') -- (3');
            \draw[lightgray] (1'') -- (3'');
			
		\end{tikzpicture} 
		\end{minipage}
		\hfill
		\begin{minipage}{0.5\linewidth}
			\centering
			\begin{tikzpicture}[node distance={20mm}, thick, main/.style = {draw, circle,minimum size=.75cm}]
				\node[main] (1) {$ \bar{x}^1$}; 
				\node[main] (2) [above right of=1] {$ \bar{x}^2$};  
				\node[main] (3) [right of=2] {$\bar{x}^4$};
				\node[main] (4) [below right of=3] {$ \bar{x}^6$};
                \node[main] (5) [below left of=4] {$ \bar{x}^5$};
                \node[main] (6) [left of=5] {$ \bar{x}^3$};
				
				\draw (1) -- (2);
				\draw (2) -- (3);
				\draw (3) -- (4);
                \draw (4) -- (5);
                \draw (5) -- (6);
                \draw (6) -- (1);
				
			\end{tikzpicture} 

		\end{minipage}
	\end{minipage}
	\caption{Take $\bar{x}^1=(0,1,4,5,8), \bar{x}^2= (0,2,4,5,8), \bar{x}^3=(0,1,4,6,8), \bar{x}^4=(0,2,4,7,8), \bar{x}^5=(0,3,4,6,8), \bar{x}^6=(0,3,4,7,8)$ and $Y =  \{\bar{x}^1,\bar{x}^2,\bar{x}^3,\bar{x}^4,\bar{x}^5,\bar{x}^6\}.$ (left) A graph $G$ for which the cycle 
   $\gamma=\bar{x}^1 - \bar{x}^2 - \bar{x}^3 + \bar{x}^4 + \bar{x}^5 - \bar{x}^6$ is minimally supported on $Y.$ The set $E_\text{supp}$ is represented in black, while the edges in $E_\text{diff} \setminus (E_\text{diff} \cap E_\text{rem})$ are gray. (right) The structure graph $s(Y).$}
        \label{fig:even cycle minimally supported on Y}
\end{figure}
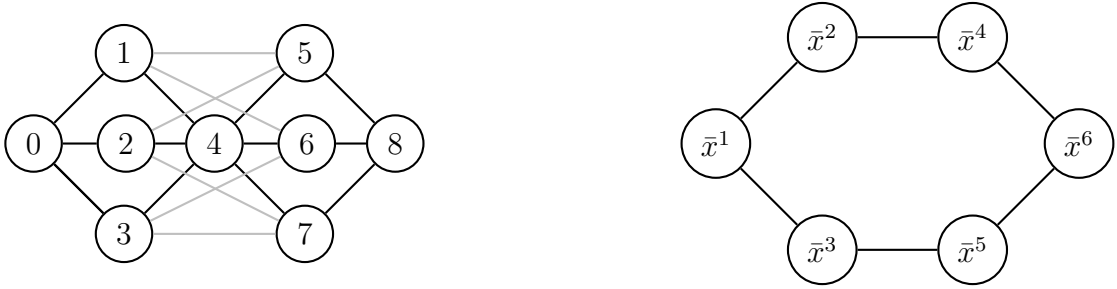

Now, select any $\bar{x}^{j_0} \in Y,$ and take $\gamma_Y = \sum_{j\in J}(-1)^{d(\bar{x}^j, \bar{x}^{j_0})}a_{j_0}\bar{x}^{j}$, where $d$ is the path distance in $s(X)\vert_{Y}.$ This is a cycle in $EMC_{k,k}(G)$ minimally supported on $Y,$ per the discussion following Definition \ref{def:structuregraph}, as illustrated in Figures \ref{fig:square minimally supported on Y} and Figures \ref{fig:even cycle minimally supported on Y}. Now, there is a nonempty collection of trails $Z \subseteq Y$ for which the coefficients in $\gamma_Y$ matches its coefficient in $\gamma.$ Thus, the cycle $\gamma$ can be written as $\gamma = \gamma' + \gamma_Y$ for some $\gamma'$  minimally supported on $X \setminus Z,$ which does not contain the circuit $s(X)\vert_{Y}.$

The structure graph $s(X \setminus Z)$ has strictly fewer minimal-length simple circuits than $s(X).$ By repeating this process inductively, we can decompose any cycle supported on a local, compatible family $X$ into summands minimally supported on local subsets $Y_1, \dots Y_r$ for which $s(X)_{Y_a}$ is a minimal-length simple circuit or such a circuit along with an adjoined 3-clique, along with a summand supported on a clique-forest subgraph of $s(X).$ This final summand decomposes into summands minimally supported on its constituent clique-trees, which each local and compatible by construction.
\end{proof}

\begin{remark}
    Theorem \ref{thm:cycle_decomp} does not quite provide a cycle basis for $EMH_{k,k}$, but it does allow us to enumerate possible cycles. In the special case $k=2$, the decomposition produces a structure graph with a single clique, as there is only one vertex which can differ from trail to trail in any given local compatible family of $2$-trails. Each edge in this clique witnesses a pair of trails which are supported on a full subgraph of $G$ isomorphic to either $C_4$ or $F_4.$ Thus, applying this decomposition to the collection of all such families in $G$ recovers Lemma \ref{lem:emh22}. 
    
    Careful accounting of the way such subgraphs interact could in principle be used to generalize Lemma \ref{lem:emh22} to higher $k$, but the necessary combinatorics quickly become complicated.
    For example, in the case $k=3$, the only possible minimal circuits in the structure graph for a local compatible family are isomorphic to $C_4$, as illustrated in Figure \ref{fig:square minimally supported on Y}. These witness subgraphs which are either isomorphic to the graph illustrated in that figure, or to variants of that graph which include edges $\{1, 2\}$ and/or $\{3, 4\}$. The cycles resulting from such families may interact in non-trivial ways. Rather than pursuing these combinatorics directly, our goal will be to apply these decompositions as obstructions, allowing us to study when $EMH_{k,k}(G)$ vanishes.
    
\end{remark}

\begin{remark}
    In the proof of Theorem \ref{thm:cycle_decomp}, when considering minimal length circuits in $s(X)$, we saw that each such circuit would have had minimal length 4 if the trail $\bar{x}^{j_3} =(x^{j_0}_0, \dots x^{j_1}_t, \dots , x^{j_2}_{t'}, \dots , x^{j_0}_k)$ were in the collection $Y.$ It is interesting to observe that this trail is always in $T_{k,k}(G)$ for $G$ which supports the three trails $\bar{x}^{j_0}, \bar{x}^{j_1},$ and $\bar{x}^{j_2}$ in the proof, so there are always eulerian magnitude cycles with structure graphs isomorphic to $C_4$ whenever larger cyclic structure graphs can be found.
\end{remark}

If we combine Theorem \ref{thm:cycle_decomp} with the observation that every cycle in $EMC_{k,k}(G)$ can be decomposed into cycles minimally supported on disjoint local collections of trails\footnote{Due to the the block structure of the differential.}, we obtain our desired spanning set for $\ker(\partial_{k,k}) \cong EMH_{k,k}(G).$ 

\begin{cor}\label{cor:cycle_spanning_set}
    Let $G$ be a graph. Then $\ker(\partial_{k,k})\subseteq EMC_{k,k}(G)$ is spanned by cycles $\gamma$ minimally supported on local, compatible collections $X_\gamma \subseteq ET_{k,k}(G)$ so that $s(X_\gamma)$ is a clique-tree, or $s(X_\gamma) \cong C_{d}$ for $d$ even.
\end{cor}

For general graphs $G$, it is difficult to count the number of subgraphs for which the corresponding structure graphs lie in these classes, or to count the number of decompositions a particular cycle might have along these lines. However, as we will see in Sections \ref{sec:EMH_ER} and \ref{sec:EMH_RGG}, for some important families of random graphs it is possible to obtain useful bounds on when such subgraphs can arise. To apply those results to the study of magnitude homology, however, we will need to think about how it is related to this new object.

\subsection{Relationship to magnitude homology}
\label{subsec:relationship}

As the eulerian magnitude chain complex is a subcomplex of the usual magnitude chain complex, we can construct a long exact sequence relating the two in the usual way. As the generators for the eulerian magnitude chain groups are a subset of those for magnitude chains, it is relatively easy to characterize the quotient: it is generated by those $k$-trails which repeat at least one vertex. Drawing inspiration from the study of spaces of singular curves, we make the following definition.

\begin{definition}
	Let $G$ be a graph.
	The \emph{discriminant $(k,\ell)$-magnitude chain group} $DMC_{k,\ell}(G)$ is the quotient
	\[
	DMC_{k,\ell}(G) = \frac{MC_{k,\ell}(G)}{EMC_{k,\ell}(G)}
	\]
	equipped with the usual quotient differential.
\end{definition}

By definition of $EMC_{k,\ell}(G)$ and $DMC_{k,\ell}(G)$ we have the following short exact sequence of chain complexes
\[
0 \to EMC_{\ast,\ell}(G) \xrightarrow{\iota} MC_{\ast,\ell}(G) \xrightarrow{\pi} DMC_{\ast,\ell}(G) \to 0.
\]
where $\iota$ and $\pi$ are the induced inclusion and quotient maps, respectively. Therefore, we obtain a long exact sequence in homology
\begin{equation}
	\label{eq:LES}
	\cdots \to DMH_{k+1,\ell}(G) \xrightarrow{\delta_{k+1}} EMH_{k,\ell}(G) \xrightarrow{\iota_*} MH_{k,\ell}(G) \xrightarrow{\pi_*} DMH_{k,\ell}(G) \xrightarrow{\delta_{k}} \cdots,
\end{equation}
where the map $\delta_{k+1}$ as given by the Snake Lemma is defined as

\[
\begin{matrix}
	\delta_{k+1} : &DMH_{k+1,\ell}(G) 			&\to 		&EMH_{k,\ell}(G) \\
	&[(x_0,\dots,x_{k+1})]		&\mapsto	&[\bar{\partial}_{k+1,\ell}(x_0,\dots,x_{k+1})],
\end{matrix}
\]
where $\bar{\partial}_{k+1,\ell} = \sum_{i\in[k-1]}(-1)^i \bar{\partial}^i_{k+1,\ell}$ with
\[
\bar{\partial}^i_{k+1,\ell}(x_0,\dots,x_{k+1})=
\begin{cases}
	(x_0,\dots,\hat{x}_i,\dots,x_k)& \text{ if }(x_0,\dots,\hat{x}_i,\dots,x_k)\in ET_{k,\ell}(G)\\
	0& \text{ otherwise.}
\end{cases}
\]

Note that, in general, it is not true that $EMH_{k,\ell}(G)$ is a subgroup of $MH_{k,\ell}(G),$ so the long exact sequence in equation (\ref{eq:LES}) does not split.
Consider, for example the graph $G$ in Figure \ref{fig:example relation} and consider the element $(0,1,2,3,1,4) \in MC_{5,5}(G)$.
Applying the differential map, we have
\[
\partial_{5,5}(0,1,2,3,1,4) = (0,1,2,3,4).
\]

So the generator $(0,1,2,3,4) \in MC_{4,5}(G)$ is a boundary and will be trivial in $MH_{4,5}(G)$. On the other hand, the same tuple $(0,1,2,3,4)$ cannot be a boundary in $EMC_{4,5}(G)$ since the $5$-trail already contains all vertices of $G$. Hence, in general the map $EMH_{k, \ell} \to MH_{k, \ell}$ is not injective.

\begin{figure}[h]
	
	\centering
	\begin{tikzpicture}[node distance={15mm}, thick, main/.style = {draw, circle}]
		\node[main] (0) {$0$}; 
		\node[main] (1) [right of=0] {$1$}; 
		\node[main] (3) [above right of=1] {$3$};
		\node[main] (2) [below right of=3] {$2$};
		\node[main] (4) [below right of=1] {$4$};
		\draw (0) -- (1);
		\draw (1) -- (2);
		\draw (1) -- (3);
		\draw (2) -- (3);			
		\draw (2) -- (4);
		\draw (1) -- (4);
		
		\path[every node/.style={font=\sffamily\small}]
		(0) edge [bend right] node {} (2);
		
	\end{tikzpicture} 
	\caption{In this graph, the element $[(0,1,2,3,4)]$ is trivial in $MH_{4,5}(G)$ but not in $EMH_{4,5}(G)$.}
	\label{fig:example relation}
\end{figure}
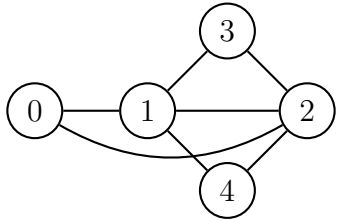

Despite the fact that the long exact sequence in equation (\ref{eq:LES}) fails to split in general, along the $k = \ell$ line we can still leverage the relationships 
between generators to make some interesting observations. As $EMH_{k,k}(G)$ is the first potentially non-trivial group in the long exact sequence of equation (\ref{eq:LES}), for every $k$ it holds that the map $i_* : EMH_{k,k}(G) \to MH_{k,k}(G)$ is injective.
For the same reason, if $EMH_{k,k}(G)$ is trivial, then $MH_{k,k}(G)$ is a subgroup of $DMH_{k,k}(G)$. In fact, in the regime where $EMH_{k,k}(G)$ is trivial, we are able to explicitly describe the generators of $MH_{k,k}(G)$ and, if $k \geq 5$ establish an isomorphism $MH_{k,k}(G) \cong DMH_{k,k}(G)$. 

\begin{theorem}
\label{thm:relation MH EMH DMH}
Fix an integer $k \geq 2.$ Let $G=(V,E)$ be a graph and suppose $EMH_{n,n}(G)\cong\langle 0 \rangle$ for $2 \leq n \leq k$.
Then every generator of $MH_{k,k}(G) \cong \ker(\partial_{k,k})$ is a $k$-trail of the form $(x_0, x_1, x_0, x_1, \dots)$ which visits only one pair of vertices $x_0\neq x_1 \in V,$
\end{theorem}

In Sections \ref{sec:EMH_ER} and \ref{sec:EMH_RGG}, we will describe regimes in which this theorem holds for important classes of random graphs.

\begin{proof}
Fix $k \geq 2.$ Consider a graph $G=(V,E)$ and assume $EMH_{n,n}(G)\cong\langle 0 \rangle$ for  $2\leq n\leq k$.

First, we consider the case of cycles supported on individual trails. Let $\bar{x}=(x_0,\dots,x_k) \in MC_{k,k}(G)$ and suppose $\partial_{k,k}(\bar{x})=0.$ Suppose we have $i < j \in [k]_0$ such that there is an eulerian sub-trail $(x_i, \dots, x_j) \in ET_{j-i,j-i}(G)$ of $\bar{x}.$ If $j - i \geq 2,$ then as $EMH_{j-i, j-i}(G) \cong \ker(\partial_{j-i,j-i}\vert_{EMC_{j-i,j-i}(G)})$ vanishes, for some $a \in [j-i-1]$ we have $\partial_{j-i,j-i}^a(x_i, x_{i+1}, \dots x_j) \neq 0,$ and thus $\partial_{k,k}^{i+a}\bar{x} \neq 0.$ Thus $j = i+1.$ However, for any $k$-trail, pairs of sequential vertices are distinct and thus must form a maximal  eulerian subtrail, so $x_i = x_{i+2}$ for every $i \in [k-1]_0.$ Thus, $\bar{x} = (x_0, x_1, x_0, \dots).$

Consider now an element $\sum_{i\in[m]} a_i\bar{x}^i\in MC_{k,k}(G)$ so that $\partial_{k,k}(\sum_{i\in[m]} a_i\bar{x}^i) = 0,$ and for which $\partial_{k,k}(\bar{x}^i) \neq 0$ for each $i\in [m].$ By the contrapositive of the argument above, we see that each involved $k$-trail $\bar{x}^i$ visits at least three vertices, and indeed non-vanishing terms in the differential can only occur away from subtrails which repeatedly cross the same edge. Suppose $\partial^r_{k,k}(\bar{x}^i) \neq 0$ then $r\in[k-1]$, $x_{r-1}^i, x_{r}^i,$ and $ x_{r+1}^i$ are all distinct. Consider $(x_{r-1}^i, x_{r}^i, x_{r+1}^i)\in ET_{2,2}(G)\subseteq EMC_{2,2}(G),$ and observe that $\partial^r_{k,k}(\bar{x}^i) \neq 0$ implies $\partial^1_{2,2}((x_{r-1}^i, x_{r}^i, x_{r+1}^i)) \neq 0.$ However, if this boundary is to cancel, there must be $j \neq i \in [m]$ so that $\partial^r_{k,k}(\bar{x}^i) = \partial^r_{k,k}(\bar{x}^j)\neq 0$. These two trails agree except at the term $x_r^i \neq x_r^j,$ so we have $(x_{r-1}^i, x_{r}^j, x_{r+1}^i)\in ET_{2,2}(G)$ as well, and their difference forms a non-zero homology class, contradicting the fact that $EMH_{2,2}(G) \cong \langle 0 \rangle.$ Thus, no such elements can exist, so all cycles in $MH_{k,k}(G)$ are of the required form. \qedhere
\end{proof}

The following is an immediate consequence of Theorem \ref{thm:relation MH EMH DMH}.

\begin{cor}
Fix an integer $k \geq 2.$ If $G=(V,E)$ is such that $EMH_{n,n}(G)=\langle 0 \rangle$ for $2 \leq n\leq k$, then $|MH_{k,k}(G)|=2|E|=|MH_{1,1}(G)|.$
\end{cor}

In a similar spirit to Theorem \ref{thm:relation MH EMH DMH}, under weaker vanishing conditions, we can prove the following isomorphism.

\begin{theorem}
\label{thm:iso MH DMH}
Fix $k\geq 5.$ Suppose $EMH_{2,2}(G) \cong EMH_{k,k}(G) \cong \langle 0 \rangle,$ then $MH_{k,k}(G) \cong DMH_{k,k}(G)$.
\end{theorem}

\begin{proof}
Let $k \geq 5$ and suppose $EMH_{2,2}(G)\cong EMH_{k,k}(G) \cong \langle 0 \rangle.$ By the long exact sequence in equation (\ref{eq:LES}), to prove that $MH_{k,k}(G) \cong DMH_{k,k}(G)$ it suffices to show that $EMH_{k-1,k}(G)\cong\langle 0 \rangle$.

Suppose that $\bar{x} = (x_0,\dots,x_{k-1}) \in ET_{k-1,k} \subseteq EMC_{k-1,k}(G)$. Observe that in such a  $(k-1)$-trail, we must have $d(x_i, x_{i+1}) = 1$ for all but one pair of consecutive vertices, for which $d(x_j, x_{j+1}) = 2.$ Since $k \geq 5$, regardless of the value of $j\in[k-1]_0$, the trail $\bar{x}$ will have at least three consecutive vertices $x_{r-1},x_r,x_{r+1}, r \in [k-2]$ for which $d(x_{i-r},x_r)=d(x_r,x_{r+1})=1$. That is, $(x_{r-1}, x_r, x_{r+1}) \in ET_{2,2}(G).$

\sloppy
Now, suppose by contradiction $\partial_{k-1,k}(\bar{x}) = 0$. Then $\partial^r_{k-1,k}(\bar{x}) = 0,$ whence $\partial^1_{2,2}(x_{r-1}, x_r, x_{r+1}) = 0$ and so $\{x_{r-1},x_{r+1}\} \in E$ by Lemma \ref{lem:diff_zero_so_edge}. However, then $G\vert_{\{x_{r-1}, x_r, x_{r+1}\}} \cong C_3$, which by Lemma \ref{lem:emh22} contradicts the assumption that $EMH_{2,2}(G) \cong \langle 0 \rangle.$ So, we have $\{x_{r-1},x_{r+1}\} \not\in E$ and thus $\partial^r_{k-1,k}(\bar{x}) \neq 0.$

Suppose now that there exists some linear combination $\sum_{i\in [m]}a_i\bar{x}^i  \in EMC_{k-1,k}(G)$ for which $\partial_{k-1,k}(\sum_{i\in [m]}a_i\bar{x}^i) = 0.$ Consider $\bar{x}^1 = (x_0^1, \dots, x_{k-1}^1)$ and apply the above argument to again find $r \in [k-1]$ so that $d(x_{i-r}^1,x_r^1)=d(x_r^1,x_{r+1}^1)=1.$ As $\partial^r_{k-1,k}(\bar{x}^1) \neq 0$, in order to cancel this term there must be some $\bar{x}^j$, $j \neq 1$ for which $\partial^r(\bar{x}^j) = \partial^r(\bar{x}^1),$ $x_r^1 \neq x_r^j$ and $x_i^1 = x_1^j$ for $i \neq r.$ However, then $G\vert_{\{x_{r-1}^1, x_{r}^1, x_{r+1}^1, x_{r}^j\}}$ is isomorphic to either $C_4$ or $F_4,$ as in Lemma \ref{lem:emh22}, in either case again contradicting the assumption that $EMH_{2,2}(G)\cong\langle0\rangle.$ Thus, we must have $EMH_{k-1,k}(G) \cong \langle0\rangle$ as required. \qedhere
\end{proof}

Note that in Theorem \ref{thm:iso MH DMH} the hypothesis $k\geq 5$ was essential. In $EMC_{2,3}(G)$ and $EMC_{3,4}(G),$ graphs such as those illustrated in Figure \ref{fig:EMH for lower k} support trails which generate nontrivial homology classes.

\begin{figure}[h]
	\begin{minipage}{\linewidth}
		\begin{minipage}{0.45\linewidth}
			\centering
			\begin{tikzpicture}[node distance={15mm}, thick, main/.style = {draw, circle}]
				\node[main] (0) {$0$}; 
				\node[main] (1) [above right of=0] {$1$};  
				\node[main] (2) [right of=1] {\textcolor{white}{2}};
				\node[main] (3) [below right of=2] {$3$};
				\node[node distance={18mm}, thick, style={draw,circle},right of=0] (4) {\textcolor{white}{4}};
				\draw (0) -- (1);
				\draw (1) -- (2);
				\draw (2) -- (3);
				\draw (0) -- (4);
				\draw (3) -- (4);
			\end{tikzpicture} 
			\subcaption{In the graph $C_5,$ the generator $(0,1,3) \in EMC_{2,3}(G)$ descends to  a nontrivial homology cycle in $EMH_{2,3}(G).$}
		\end{minipage}
		\hfill
		\begin{minipage}{0.45\linewidth}
			\centering
			\begin{tikzpicture}[node distance={15mm}, thick, main/.style = {draw, circle}]
				\node[main] (0) {$0$}; 
				\node[main] (1) [above of=0] {$1$};  
				\node[main] (2) [right of=1] {\textcolor{white}{2}};
				\node[main] (3) [right of=2] {$3$};
				\node[main] (4) [above  of=3] {$4$};
				\node[main] (5) [right of=0] {\textcolor{white}{5}};
				\node[main] (6) [above  of=2] {\textcolor{white}{6}};
				\draw (0) -- (1);
				\draw (1) -- (2);
				\draw (2) -- (3);
				\draw (3) -- (4);
				\draw (0) -- (5);
				\draw (3) -- (5);
				\draw (1) -- (6);
				\draw (4) -- (6);
			\end{tikzpicture} 
			\subcaption{In the pictured graph, the generator $(0,1,3,4) \in EMC_{3,4}(G)$ represents a nontrivial a homology cycle in $EMH_{3,4}(G).$}
		\end{minipage}
	\end{minipage}
	\caption{Graphs for which $EMH_{k,k}(G)\cong\langle 0 \rangle$ but $EMH_{k-1,k}(G) $ is non-trivial for some $k< 5$.}
	\label{fig:EMH for lower k}
\end{figure}
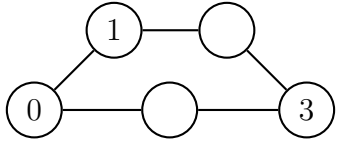
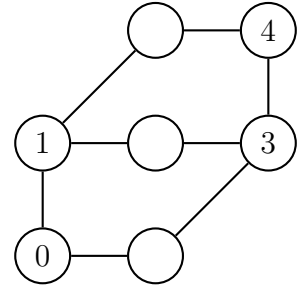

One last question about the relationship between eulerian and standard magnitude homology arises from the example in Figure \ref{fig:example relation}, where we have a non-trivial eulerian magnitude homology cycle that becomes trivial in standard magnitude homology.  Using the tools we have developed, we can provide a partial answer in the first non-trivial case, when $\ell = k+1$, determining which non-trivial eulerian magnitude homology cycles generated by a single tuple $\bar{x}$ become trivial in standard magnitude homology. 

\begin{theorem}
\label{thm:lesnotsplit}
Let $G$ be a graph and fix a nonnegative integer $k.$ Let $\bar{x}=(x_0,\dots,x_{k-1}) \in ET_{k-1,k}(G)$ such that $[\bar{x}] \in EMH_{k-1,k}(G)$ is non-trivial.
Then $[\bar{x}]$ is trivial in $MH_{k-1,k}(G)$ if and only if there are some $i,r\in [k-2]$ with $r - i \geq 2$ so that the $k$-trail \[(x_0,\dots, x_{i},\dots,  x_{r},x_{i}, x_{r+1},\dots,x_{k-1})\] appears in $G,$ in which case $G\vert_{\{x_i, x_{i+1}, \dots, x_r\}} \in \Gamma(\mathcal{C}_{r-i}),$ where $\mathcal{C}_{r-i}$ is the complete class graph for which $\alpha(\mathcal{C}_{r-i}) \cong C_{r-i}.$ is cyclic on $(r-i)$ vertices\footnote{Here, we take $C_2$ to be the graph with two vertices and one edge to simplify the statement.}.
\end{theorem}

\begin{proof}
Suppose $G= (V,E)$ is a graph and $\bar{x}=(x_0,\dots,x_{k-1}) \in ET_{k-1,k}(G)$ such that $[\bar{x}] \in EMH_{k-1,k}(G)$ is non-trivial. So, there exists no $\bar{x}' \in ET_{k,k}(G)$ with $\partial_{k,k}(\bar{X}') = \bar{x}.$ Thus, for $[\bar{x}]$ to be trivial in $MH_{k-1,k}(G),$ there must exist $\bar{x}' \in T_{k,k}(G) \setminus ET_{k,k}(G)$ so that $\partial^r_{k,k}(\bar{x'})= \bar{x}$ for some $r\in[k-1].$ Thus, $\bar{x}' = (x_0, \dots, x_{r-1}, y, x_{r}, \dots, x_{k-1}),$ as pictured in Figure \ref{fig:vanishing tuple in MH}. However, since $\bar{x}' \not\in ET_{k,k}(G),$ we must have $y = x_i$ for some $i \in [k-2].$ 

In this case, $\bar{x}'' = (x_0, \dots, x_{i-1}, x_{i+1}, \dots, x_{r-1}, x_i, x_{r}, \dots, x_{k-1})$\footnote{Supposing $r<i.$ If $i > r$, reverse the appearance order in this sequence.} is also an eulerian $(k-1)$-trail that could appear as a term in $\partial_{k,k}(\bar{x}').$ However, for the same reason as in Lemma \ref{lem:diff_zero_so_edge}, the fact that $\partial_{k-1,k}^i\bar{x} = 0$ implies that $\{x_{i-1}, x_{i+1}\}\in E,$ so $\len(\bar{x}'') = k-1.$ Thus, this term is zero in the differential. All other terms in $\partial_{k,k}(\bar{x}')$ are zero because they agree with those of $\bar{x}$, indeed when computing all other terms in $\partial_{k,k}(\bar{x}')$ we are removing the same vertices we remove when computing $\partial_{k,k}(\bar{x})$. Thus, $[\bar{x}] = 0$ in $MH_{k-1,k}(G)$.
\end{proof}

\begin{figure}[h]
	\begin{minipage}{\linewidth}
		\begin{minipage}{0.45\linewidth}
			\centering
			\begin{tikzpicture}[node distance={15mm}, thick, main/.style = {draw, circle, minimum size=1.1cm}]
				\node[main] (0) {$x_0$}; 
                \node[main] (7) [right of=0] {$x_{i-1}$}; 
				\node[main] (1) [below right of=7, xshift=0.5cm] {$x_i$};  
				\node[main] (2) [below right of=1] {$x_{i+1}$};
				\node[main] (3) [right of=2] {$x_{i+2}$};
                \node[main] (6) [above right of=1] {$x_{r}$};
                \node[main] (9) [right of=6] {$x_{r-1}$};
                \node[main] (5) [below left of=1,xshift=-0.5cm] {$x_{r+1}$};
				\node[main] (8) [left of=5] {$x_{k-1}$};
				
				\draw[dotted] (0) -- (7);
                \draw[dotted] (5) -- (8);
                \draw (1) -- (7);
				\draw (1) -- (2);
    			\draw (1) -- (6);
				\draw (2) -- (3);
                \draw (1) -- (5);
                \draw (6) -- (9);
				\path[every node/.style={font=\sffamily\small}]
				(9) edge [bend right=15, dotted] node {} (3);
				\path[every node/.style={font=\sffamily\small}]
				(2) edge [bend left=35, lightgray] node {} (7);
    		  \path[every node/.style={font=\sffamily\small}]
				(1) edge [bend left=15, lightgray] node {} (3);
    \path[every node/.style={font=\sffamily\small}]
				(1) edge [bend right=15, lightgray] node {} (9);
			\end{tikzpicture} 
		\end{minipage}
		\hspace{.7cm}
		\begin{minipage}{0.45\linewidth}
			\centering
			\begin{tikzpicture}[node distance={15mm}, thick, main/.style = {draw, circle, minimum size=1.1cm}]
                \node[main] (1)  {$x_{i-1}$};
				\node[main] (0) [above left of=1] {$x_0$}; 
				\node[main] (2) [right of=1] {$x_{i}$};
				\node[main] (3) [right of=2] {$x_{i+2}$};
				\node[main] (4) [above of=2] {$x_{i+1}$};
				\node[main] (5) [above right of=3] {$x_{k-1}$};
				
				\draw[dotted] (0) -- (1);
				\draw (1) -- (2);
				\draw (2) -- (3);
				\draw (2) -- (4);
                \draw[dotted] (3) -- (5);
                \draw[lightgray] (1) -- (4);
			\end{tikzpicture} 
		\end{minipage}
	\end{minipage}
    \vspace{.7cm}
	\caption{Relevant subgraphs of minimal graphs in classes which support non-trivial eulerian magnitude homology cycles which become trivial in standard magnitude homology, per the proof of Theorem \ref{thm:lesnotsplit}. Note the asymmetry of those gray edges around $x_i$ which must be present to enforce the vanishing differential for $(x_0, \dots, x_{k-1}) \in EMC_{k-1,k}(G).$}
	\label{fig:vanishing tuple in MH}
\end{figure}
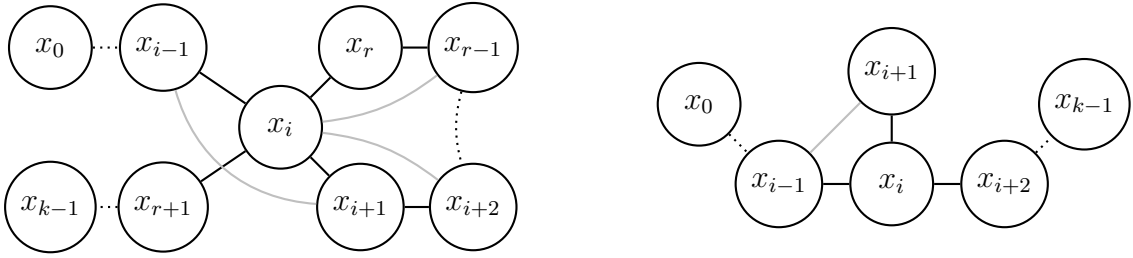

\section{Eulerian magnitude homology of Erd\H{o}s-R\'{e}nyi graphs}
\label{sec:EMH_ER}

The \emph{Erd\H{o}s-R\'{e}nyi (ER) model} for random graphs, $G(n,p),$ introduced in \cite{erdHos1960evolution}, is one of the most widely studied and applied models for  random graphs. As the maximum entropy distribution on graphs with an expected proportion of edges, the ER model serves as a useful null model in a broad range of scientific and engineering applications. For this reason, clique complexes of ER graphs have long been an object of interest in the stochastic topology community \cite{kahle2009topology, kahle2011random, kahle2013limit}. We expect it to serve a similar function as a baseline model for understanding magnitude homology. 

\begin{definition}
	The \emph{Erd\H{o}s-R\'{e}nyi (ER) model} 
 $G(n, p) = (\Omega, P)$ is the probability space where $\Omega$ is the discrete space of all graphs on $n$ vertices, and
 $P$ is the probability measure that assigns to each graph $G \in \Omega$ with $m$ edges probability
	\[
	P(G)=p^m(1-p)^{{n \choose 2}-m}.
	\]
\end{definition}

We can sample an ER graph $G \sim G(n, p)$ on $n$ vertices with parameter $p\in [0,1]$ by determining whether each of the $n \choose 2$ potential edges is present via independent draws from a Bernoulli distribution with probability $p.$ In order to study the limiting behavior of these models as $n \to \infty$, it is often useful to change variables so that $p$ is a function of $n.$ In Section \ref{subsec:van_threshold}, we will take $p=n^{-q}$, $q \in [0,\infty)$, following \cite{kahle2009topology}.
\smallskip 

In what follows, we will study the first diagonal ($k = \ell$) for the eulerian magnitude homology groups of ER graphs as $n \to \infty$. We would like to think of $EMH_{k,k}(G(n, p))$ as a random variable valued in finitely generated abelian groups. As we know from Corollary \ref{cor:homisker}, in this case the groups are free abelian, and counting generators is sufficient to completely understand these groups. Write $\beta_{k,k}(G) = \text{rank}(EMH_{k,k}(G)))$ for the $(k,k)$-EMH Betti number of $G,$ and $\beta_{k,k}(n,p) = \text{rank}(EMH_{k,k}(G(n,p)))$ for the corresponding random variable. In Theorem \ref{thm:vanishingthreshER}, we use the spanning set from Theorem \ref{thm:cycle_decomp} to establish a threshold in terms of $q$ beyond which $\mathbb{E}[\beta_{k,k}(n,p)]$ vanishes, as illustrated in Figure \ref{fig:non trivial emh}. 
Then, in Theorem \ref{thm:asymp size BettiEMH}, we present an explicit formula for the asymptotic size of $\E[\beta_{k,k}(n,p)]$, and in Theorem \ref{thm:central limit thm} we prove a Central Limit Theorem for $\beta_{k,k}(n,p)$.	

\subsection{A vanishing threshold for $EMH_{k,k}(G)$}	
\label{subsec:van_threshold}
Here, we leverage the connection between eulerian magnitude chains along the $k = \ell$ line and the structure of classes $\Gamma(\mathcal{H})$ of graphs that support those chains established in Section \ref{sec:EMH} to study how the expected homology groups for $G(n, n^{-q})$ behave as $n \to \infty.$ By counting subgraphs that can support chains in the kernel of the differential, we are able to establish a $q$ threshold beyond which the limiting $\ker(\partial_{k,k})$ vanishes in expectation, meaning the groups $EMH_{k,k}(G(n, n^{-q}))$ vanish. We will break the argument into two stages. First, we will consider individual basis elements $\bar{x}  \in ET_{k,k}(G) \subseteq EMC_{k,k}(G)$ such that $\partial_{k,k}\bar{x}=0$, such as those we studied in Lemma \ref{lem:cliques bound}. Then, we will demonstrate that in all other cases, the corresponding chains have a higher vanishing threshold. 

First, we require a simple observation, which follows from the fact that edges are drawn independently via Bernoulli random trials.

\begin{lemma}
    \label{lem:prob_class_ER}
    Fix positive $n$ and $p\in [0,1].$ Let $G = (V, E) \sim G(n,p)$ be an ER random graph, and $W \subseteq V$, let $\mathcal{H} = (W, E_S, E_B)$ a class graph. Then the probability that $G\vert_W \in \Gamma(\mathcal{H})$ is $p^{|E_S|}(1-p)^{{|W|\choose 2} - |E_S| - |E_B|}.$
\end{lemma}

Now, we can return to the issue at hand.

\begin{lemma}
	\label{lem:vanishing threshold_single tuple}
    Fix a non-negative integer $k$, and let $q > \frac{k+1}{2k-1}$. Then
	\[\E\left[\left|\left\{\bar{x} \in ET_{k,k}(G(n, n^{-q})) \;\colon\;\partial_{k,k}\bar{x}  = 0\right\}\right|\right] \to 0\] as $n \to \infty.$
\end{lemma}

\begin{proof}	
Sample a graph $G \sim G(n, n^{-q})$ and let $\bar{y} = (y_0, \dots, y_k) \in ET_{k,k}(G)$ be a $k$-trail of length $k$ in $G.$ We wish to estimate the probability of the corresponding generator being a cycle in $EMC_{k,k}(G).$ Taking $\mathcal{H}(\{\bar{y}\}) = (L(\bar{y}), E_S, E_B)$ to be the complete class graph from Lemma \ref{lem:class_graph_one_gen}, we see that the set $E_S$ has $2k-1$ elements, so by Lemma \ref{lem:prob_class_ER}, the probability that $G\vert_{\{y_0, \dots, y_k\}} \in \Gamma(\mathcal{H}(\{\bar{y}))$ is $p^{2k-1}.$ 

However, more than one such generator may be supported on $G\vert_{\{y_0, \dots, y_k\}};$ each such generator corresponds to an isomorphic copy of $H_k 
 =\alpha(\mathcal{H}(\{\bar{y}))).$ We can bound above the number of isomorphic copies of $H_k$ that appear on any collection of $(k+1)$ vertices in $G$ by $c(\Delta_{k+1}, H_k)$, which we recall denotes the  number of full subgraphs of $\Delta_{k+1}$ isomorphic to $H_k$. We apply these estimates to provide an upper bound on the expectation of the number of copies of $H_k$ appearing in such a sampled $G$ as $n\to \infty.$
	\begin{align*}
        \E[c(G(n, n^{-q}), H_k)]  &\leq \binom{n}{k+1} c(\Delta_{k+1}, H_k) p^{2k-1} \\
		&\sim  \frac{n^{k+1}}{(k+1)!}c(\Delta_{k+1}, H_k)  n^{q(1-2k)} \\
		&= \frac{c(\Delta_{k+1}, H_k) }{(k+1)!} n^{q(1-2k)+(k+1)} \xrightarrow{n\to \infty}
		\begin{cases}
			0, \text{ if } q > \frac{k+1}{2k-1} \\
			\infty, \text{ if } 0 < q < \frac{k+1}{2k-1}.
		\end{cases}
	\end{align*}
	Thus, we conclude that no such generators are expected to exist when $q > \frac{k+1}{2k-1}$ as $n\to \infty.$
\end{proof}

Note that the situation described in Lemma \ref{lem:vanishing threshold_single tuple}, where a single tuple $\bar{x}$ generates an $EMH$-cycle, corresponds to a structure graph $s(\{\bar{x}\})$ consisting of a single vertex. As we will see in the following lemma, the limiting probability of observing cycles minimally supported on families of trails with more complex structure graphs goes to zero in a larger $q$ range than this singleton case.

\begin{lemma}
	\label{lem:vanishing threshold_combination tuples}
Fix non-negative integers $k, n$, and fix $m \geq 2$. Let $q > \frac{k+1}{2k-1}$. Let $X = \{\bar{x}^i\}_{i \in [m]} \subseteq ET_{k,k}(\Delta_n)$ be a local, compatible collection of trails so that $s(X) \cong C_d$ for some even $d$ or $s(X)$ is a clique-tree. Then, as $n\to \infty$, $$\mathbb{E}[c(G(n, n^{-q}), \alpha(\mathcal{H}(X)))] \to 0.$$ 
\end{lemma}

\begin{proof}
Consider first the case where $s(X) \cong C_d,$
which will be similar to those depicted in Figures \ref{fig:square minimally supported on Y} or \ref{fig:even cycle minimally supported on Y}, with $|X| \geq 4.$
Taking $t < t'$ and $\bar{x}^{j_0} = (x^{j_0}_0, \dots x^{j_0}_t, \dots , x^{j_0}_{t'}, \dots , x^{j_0}_k),$ as in the proof of Theorem \ref{thm:cycle_decomp}, and following the construction of the cycle therein, we see that all but two of the elements of $X$ introduce exactly one new vertex to $\mathcal{H}(X).$ The other is the initial trail, consisting of $k+1$ new vertices, and the last does not introduce anything new. So, by construction $\mathcal{H}(X)$  must have precisely $k + |X| - 1$ vertices. Further, following  Definition \ref{def:fullclassgraph}, $E_{\text{supp}}$ will consist of exactly the requisite $(k + 1) + 2(|X| - 2) + (|X| -3) = k+3|X|-6$ edges implicated in these trails, where the last $(|X|-3)$ edges connect the $(|X|-1)$ newly added vertices. $E_\text{diff}$ can be decomposed as follows. Start with the $(k-1)$ edges for the first trail $\bar{x}^1$. Iteratively, select a new maximal clique in $s(X)$ containing $\bar{x}^1$, and observe that each trail in that clique differs from $\bar{x}^1$ in precisely one vertex, and thus there are at most two edges including that vertex that must be added to $E_\text{diff}$ to ensure the corresponding differential terms for those trails vanish. Iterating through cliques in this way, we see that each other trail in $X$ adds between one and two edges, and so $(k-1) + |X| - 1 \leq E_\text{diff} \leq (k-1) + 2(|X| - 1)$, depending on the values of $t$ and $t'.$ Finally,  $E_{\text{rem}}$ will contain precisely the edges $\{x^{j_0}_{t-1}, x^{j_0}_{t+1}\}$ and $\{x^{j_0}_{t'-1}, x^{j_0}_{t'+1}\}.$ Thus, as $|E_\text{diff} \cap E_{\text{rem}}| \leq 2$, $|E_S^{X}| \geq (k+3|X|-6) + (k-1) + |X|-2 \geq 2k + 2|X| - 1$ for every $|X| \geq 4$. Taking $W$ to be any collection of $k + |X| - 1$ vertices in $G \sim G(n, p)$, Lemma \ref{lem:prob_class_ER} then says that the probability that $G|_{W} \in \Gamma(\mathcal{H}(X))$ is bounded above by $p^{2k+2|X|-1}$.

Now, we can apply the same reasoning as in the proof of Lemma \ref{lem:vanishing threshold_single tuple} to count subgraphs of $G$ isomorphic to $H_k \cong \alpha(\mathcal{H}(X)).$ 

	\begin{align*}
        \E[c(G(n, n^{-q}), H_k)]  &\leq \binom{n}{k+|X|} c(\Delta_{k+|X|}, H_k) p^{2k+2|X|-1} \\
		&\sim  \frac{n^{k+|X|}}{(k+|X|)!}c(\Delta_{k+|X|}, H_k)  n^{q(1-2|X|-2k)} \\
		&= \frac{c(\Delta_{k+2|X|}, H_k) }{(k+|X|)!} n^{q(1-2|X|-2k)+(k+|X|)} \\
        &\xrightarrow{n\to \infty}
		\begin{cases}
			0, \text{ if } q > \frac{k+|X|}{2k+2|X|-1} \\
			\infty, \text{ if } 0 < q < \frac{k+|X|}{2k+2|X|-1}.
		\end{cases}
	\end{align*}

 In particular, as $  \frac{k+1}{2k-1} > \frac{k+|X|}{2k+2|X|-1}$, we conclude that no local compatible collection of trails with cyclic structure graph is expected to be supported in $G$ as $n\to \infty$ when $q > \frac{k+1}{2k-1}.$ 

The case where $s(X)$ a clique-tree, as in Figure \ref{fig:clique tree}, is similar. Each vertex in $s(X)$ beyond the first adds at most one vertex to $\mathcal{H}(X).$ Suppose $y$ such vertices are added, so $\mathcal{H}(X)$ has $k + y + 1$ vertices. For each such new vertex there are two edges added to $E_\text{supp}.$ The set $E_\text{rem}$ contains no more edges than the number of maximal cliques, $r$. And, $E_\text{diff}$ contains at least $((k-1) + y)$ edges. Then, since $y \geq r,$ $|E_S^{X}| \geq (k + 2y) + (k-1)  + y - r \geq 2k + 2y - 1.$ So, again taking $W$ to be a subset of the vertices of $G \sim G(n,p)$ of the appropriate size and invoking Lemma \ref{lem:prob_class_ER}, we have the probability that $G|_{W} \in \Gamma(\mathcal{H}(X))$ is bounded above by $p^{2k+2y-1}$. We now apply an identical counting argument to conclude the  required vanishing bound for collections of trails with clique-tree structure graphs.

\end{proof}
  
\noindent Now, by Theorem \ref{thm:cycle_decomp}, any cycle in $EMH_{k,k}(G)$ is either minimally supported on a singleton $\{\bar{x}^1\}$ or can be decomposed into cycles minimally supported on collections with structure graphs given by clique-trees and cycles. Thus, Lemma \ref{lem:vanishing threshold_single tuple} and Lemma \ref{lem:vanishing threshold_combination tuples} together provide a vanishing threshold for eulerian magnitude homology on the $k=\ell$ line.

\begin{theorem}
\label{thm:vanishingthreshER}
	Fix $k$ and $q > \frac{k+1}{2k-1}.$ 
	As $n \to \infty,$ \[\E\left[\beta_{k,k}(n, n^{-q})\right] \to 0.\]
 
\end{theorem}

Figure \ref{fig:non trivial emh} illustrates the $q$ range where the first diagonal of the eulerian magnitude homology of an Erd\H{o}s-R\'{e}nyi random graph $G(n,n^{-q})$ is expected to vanish as $n\to \infty.$ 

\begin{figure}[H]
	\centering
	\includegraphics[scale=0.65]{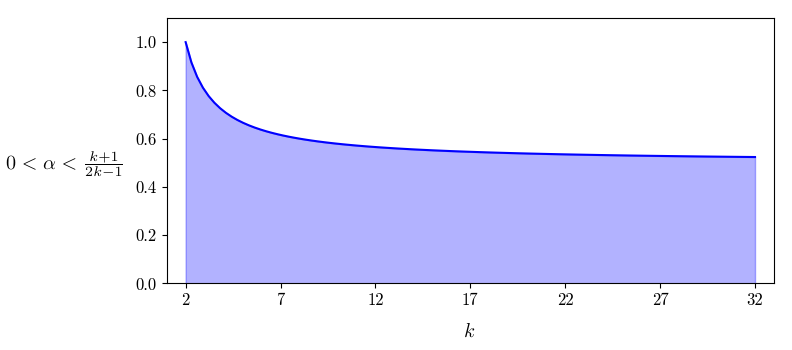}
	\caption{The shaded region below the curve is the $q$ vs $k$ region for which we can have  non-vanishing $EMH_{k,k}(G)$ in expectation as $n \to \infty$ for graphs $G\sim G(n,n^{-q}).$ By Theorems \ref{thm:relation MH EMH DMH} and \ref{thm:iso MH DMH}, in the non-shaded region asymptotically almost surely $MH_{k,k}(G) \leq DMH_{k,k}(G)$, and if also $k \geq 5,$ then $MH_{k,k}(G) \cong DMH_{k,k}(G)$.}
	\label{fig:non trivial emh}
\end{figure}

Combining the vanishing threshold in Theorem \ref{thm:vanishingthreshER} with the behavior of $MH_{k,k}(G)$ when eulerian magnitude homology vanishes as in Theorem \ref{thm:relation MH EMH DMH}, we obtain the following characterization of the expected behavior of $MH_{k,k}(G(n,p))$ as $n \to \infty.$ 

\begin{cor}
Let $G=G(n,n^{-q})$ be an Erd\H{o}s-R\'{e}nyi random graph.
If $q > \frac{k+1}{2k-1}$ then the magnitude homology group $MH_{k,k}(G)$ is the subgroup of the discriminant magnitude homology group $DMH_{k,k}(G)$ generated by tuples of the form $(x_0,x_1,x_0,\dots)$ for which the induced path only revisits the same edge $(x_0,x_1)$.
\end{cor}

\subsection{Asymptotic behavior of $\beta_{k,k}(G)$ for ER random graphs}

Given the prominence of cliques in our computations in Section \ref{subsec:van_threshold}, it may come as little surprise that the pioneering work of Kahle and Meckes on the limiting behavior of Betti numbers in random clique complexes would have application in this context. Following \cite[Theorem 2.3]{kahle2013limit} and \cite[Theorem 1.1]{kahle2015erratum}, in this section we compute expectations of these values for Erd\H{o}s-R\'{e}nyi random graphs as $n\to\infty,$ see Figure \ref{fig:expected betti ER} Because we rely on many of the same combinatorial structures, in the name of brevity we will refer the reader to computations in these papers when the details are identical.

 \begin{figure}[H]
	\centering
	\includegraphics[scale=0.65]{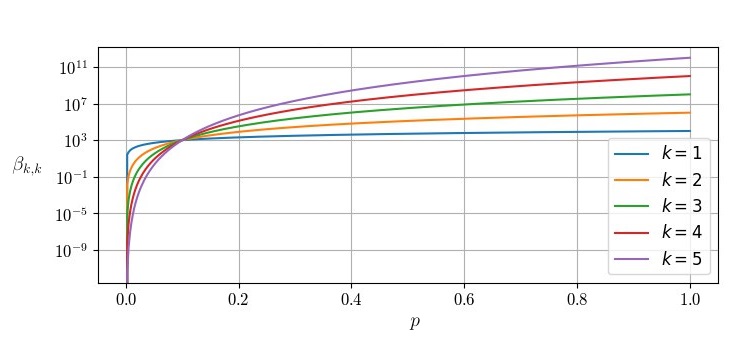}
	\caption{The expected value for the Betti numbers $\beta_{k,k}$, $k \in [1,5]$, of an Erd\H{o}s-R\'{e}nyi random graph is plotted vertically against the probability $p$. In this example $n = 100$ and the $y$-axis is on a base 10 logarithmic scale to better visualize the large differences in values. Notice that the the curves all intersect at roughly $p=0.1=n^{-\alpha}$ with $n=100$ and $\alpha=\frac{1}{2} \sim \frac{k+1}{2k-1}$.}
	\label{fig:expected betti ER}
\end{figure}

\begin{theorem}
	\label{thm:asymp size BettiEMH}
	Fix $k$ nonnegative and write $\beta_{k,k} = \rank(EMH_{k,k}(G(n,p)))$. 
 
 If $p = 1-O\left( \frac{\ln(n)}{n} \right)$ then 
	\[
	\lim\limits_{n\to \infty} \frac{\mathbb{E}[\beta_{k,k}]}{n^{k+1}p^{2k-1}} = 1.
	\]
\end{theorem}

\begin{proof}
    In the regime where eulerian magnitude homology does not vanish we know that $0<q<\frac{k+1}{2k-1}$.
    From Lemma \ref{lem:vanishing threshold_single tuple} the Betti number $\beta_{k,k}$ is at least the number of single vertex structure graphs,
    
    \begin{equation}
    \label{eq:expected value left inequality}
    \binom{n}{k+1}(k+1)!p^{2k-1} \leq \beta_{k,k},
    \end{equation}
    
    and from Lemma and \ref{lem:vanishing threshold_combination tuples} it is at most the number of combinations obtained by replacing $m$ edges $(x_{i-1},x_{i+1})$ with $2m$ edges $(x_{i-1},x'_{i})$, $(x'_{i},x_{i+1})$
    
    \begin{align}
    \label{eq:expected value right inequality 1}
    \begin{split}
        \beta_{k,k} &\leq \sum_{m=0}^{k-1} \binom{n}{k+1+m}(k+1+m)!p^{2k-1-m+2m}(1-p)^m \\
                    & = \sum_{m=0}^{k-1} \binom{n}{k+1+m}(k+1+m)!p^{2k-1+m}(1-p)^m,
    \end{split}
    \end{align}
    
    which for $p=n^{-q}$ is equal to
    \begin{align}
    \label{eq:expected value right inequality 2}
    \begin{split}
    &\sum_{m=0}^{k-1} \binom{n}{k+1+m}(k+1+m)!n^{-q(2k+m-1)}(1-n^{-q})^m \\
    \sim
    &\sum_{m=0}^{k-1} n^{k+1+m}n^{-q(2k+m-1)}(1-n^{-q})^m \\
    = &n^{(k+1) - q(2k-1)} + \sum_{m=1}^{k-1} n^{k+1+m}n^{-q(2k+m-1)}(1-n^{-q})^m.
    \end{split}
    \end{align}

    Now, the assumption that $p=1-O\left(\frac{\ln(n)}{n} \right)$ is telling us that $p$ behaves asymptotically as $n^{-\frac{1}{n}}$. Therefore $(1-n^{-q}) \sim (1-n^{-\frac{1}{n}}) \xrightarrow{n \to \infty} 0$, and we see that the only subgraphs asymptotically contributing to eulerian magnitude homology are the ones for which $m=0$, i.e. the ones induced by a single tuple $(x_0,\dots,x_k)$.

    Therefore, putting together equations \ref{eq:expected value left inequality}, \ref{eq:expected value right inequality 1} and \ref{eq:expected value right inequality 2} we obtain
    \[
    \lim\limits_{n\to \infty} \frac{\mathbb{E}[\beta_{k,k}]}{n^{k+1}p^{2k-1}} = 1.
    \]
\end{proof}

We can also prove a central limit type result for $\beta_{k,k}(G(n,p))$. In what follows we will denote by $\Phi(\cdot)$ the distribution function of the standard normal distribution. Recall that a sequence $\{X_n\}_n$ of random variables converges weakly to a random variable $X,$ and we write $X_n \Rightarrow X,$ if $\lim\limits_{n \to \infty} \mathbb{E}[f(X_n)] = \mathbb{E}[f(X)]$  for all bounded continuous functions $f$.

\begin{theorem}
	\label{thm:central limit thm}
    \sloppy
	Fix $k$ nonnegative and write $\beta_{k,k} = \rank(EMH_{k,k}(G(n,p)))$.
 If $p = 1-O\left(\frac{\ln(n)}{n} \right)$, then as $n \to \infty$ we have that the sequence $\{\mathbb{E}[\beta_{k,k}]\}_k$ converges weakly to the Betti number $\beta_{k,k}$,
	\[
	\frac{\beta_{k,k} -  \mathbb{E}[\beta_{k,k}]}{\sqrt{Var(\beta_{k,k})}} \Rightarrow N(0,1).
	\]
\end{theorem}

\begin{proof}
    Let $G$ be a graph.
    Write $t_{k,\ell} = |ET_{k,\ell}(G(n,p))|$ for the random variable providing the number of eulerian $(k)$-trails  of length $\ell$ in a graph sampled from $G(n,p).$
    By Lemma \ref{lem:LowerTriangular}, $EMC_{k+1,k}(G)=\langle0\rangle$, so $EMH_{k,k}(G)$ is free abelian. From the rank-nullity theorem, we conclude that
	\begin{equation}
		\label{eq:morselikeineq}
		-t_{k-1,k} + t_{k,k} \leq \beta_{k,k} \leq t_{k,k}.
	\end{equation}
	
	Observe that $t_{k,k}$ is the number of $(k+1)$-tuples of distinct vertices that induce in the graph $G$ a path of length $k$. There are $\binom{n}{k+1}(k+1)!$ such paths possible for a graph on $n$ vertices, and for $G \sim G(n,p)$ they each appear with probability $p^k$.
    For $t_{k-1,k}$, notice that this term counts $(k)$-tuples of vertices that induce in the graph $G$ a path of length $k$. So we still need to account for $k+1$ visited vertices (even if one of those is not explicitly mentioned in the tuple). Further, we need to account for two facts: first, each edge of the induced path appears with probability $p^k$; and second, at least one of the edges $(x_{i-1},x_{i+1})$ will not appear in the induced graph.
    This leaves us with $\binom{n}{k+1}(k+1)!p^k(1-p^{k-1})$ such tuples.
    
    To prove the result we will need the following three intermediate claims:
	\begin{enumerate}
		\item $\lim\limits_{n \to \infty} \frac{Var(t_{k,k})}{Var(t_{k,k} - t_{k-1,k})} =1$
		\item $\frac{t_{k,k} -  \mathbb{E}[t_{k,k}]}{\sqrt{Var(t_{k,k})}} \Rightarrow N(0,1)$ as $n \to \infty$
		\item $\frac{(t_{k,k} -t_{k-1,k}) -  \mathbb{E}[t_{k,k} -t_{k-1,k}]}{\sqrt{Var(t_{k,k} - t_{k-1,k})}} \Rightarrow N(0,1)$ as $n \to \infty$.
	\end{enumerate}

 Deferring their proofs, we now apply these three claims to prove the theorem.  Take $\alpha \in \mathbb{R}$. Then from the inequalities in equation (\ref{eq:morselikeineq}), we get
	
	\[
	\mathbb{P} \left[ \frac{t_{k,k} -  \mathbb{E}[t_{k,k}]}{\sqrt{Var(t_{k,k})}} \leq \alpha \right] \leq
	\mathbb{P} \left[ \frac{\beta_{k,k} -  \mathbb{E}[t_{k,k}]}{\sqrt{Var(t_{k,k})}} \leq \alpha \right] \leq
	\mathbb{P} \left[ \frac{(t_{k,k} -t_{k-1,k}) -  \mathbb{E}[t_{k,k}]}{\sqrt{Var(t_{k,k})}} \leq \alpha \right].
	\] 
	
	Because of Claim 2, the left term $\mathbb{P} \left[ \frac{t_{k,k} -  \mathbb{E}[t_{k,k}]}{\sqrt{Var(t_{k,k})}} \leq \alpha \right]$ tends, as $n \to \infty$, to $\Phi(\alpha),$ the distribution function of the standard normal distribution.
	For the right-hand-side we can proceed analogously as in \cite[Theorem 2.3]{kahle2013limit}.
 
	Thus, note that
	\begin{align}
		\label{eq:prob from Morse ineq}
			\mathbb{P} &\left[ \frac{(t_{k,k} -t_{k-1,k}) -  \mathbb{E}[t_{k,k}]}{\sqrt{Var(t_{k,k})}} \leq \alpha \right] \leq 
			\mathbb{P} \left[ \frac{(t_{k,k} -t_{k-1,k}) -  \mathbb{E}[t_{k,k} -t_{k-1,k}]}{\sqrt{Var(t_{k,k} -t_{k-1,k})}} \leq \alpha - \varepsilon \right] + \nonumber\\
			+ &\mathbb{P} \left[ \left| \frac{(t_{k,k} -t_{k-1,k}) -  \mathbb{E}[t_{k,k} -t_{k-1,k}]}{\sqrt{Var(t_{k,k} -t_{k-1,k})}} - \frac{(t_{k,k} -t_{k-1,k}) -  \mathbb{E}[t_{k,k}]}{\sqrt{Var(t_{k,k})}} \right| > \varepsilon \right] +  \nonumber\\
			&+ \mathbb{P} \left[ \frac{(t_{k,k} -t_{k-1,k}) -  \mathbb{E}[t_{k,k}]}{\sqrt{Var(t_{k,k})}} \leq \alpha,
			\left| \frac{(t_{k,k} -t_{k-1,k}) -  \mathbb{E}[t_{k,k} -t_{k-1,k}]}{\sqrt{Var(t_{k,k} -t_{k-1,k})}}\right| \leq \varepsilon \right].
	\end{align}
	
	Using Claim 3 we see that the first summand of the right-hand-side of equation (\ref{eq:prob from Morse ineq}) tends to $\Phi(\alpha - \varepsilon)$, and the last summand is bounded above by $\Phi(\alpha + \varepsilon) - \Phi(\alpha - \varepsilon)$.
	For the middle term,
 notice that from the assumption $p=1-O\left(\frac{\ln(n)}{n} \right)$ we have $p \sim n^{-\frac{1}{n}}$, and thus
	\[
	\lim\limits_{n\to \infty}\frac{\mathbb{E}[t_{k-1,k}]}{\mathbb{E}[t_{k,k}]} = 
	\lim\limits_{n\to \infty}\frac{\binom{n}{k+1}(k+1)!p^k(1-p^{k-1})}{\binom{n}{k+1}(k+1)!p^{k}} = 0.
	\]
	\noindent It follows that 
	\begin{equation}
    \label{eq:limit mean values quotient}
	\lim\limits_{n\to \infty}\frac{\mathbb{E}[t_{k,k}]}{\mathbb{E}[-t_{k-1,k}+t_{k,k}]} =1,
    \end{equation}
 and so using the limit in equation \ref{eq:limit mean values quotient} and in claim 1 we get that

    \begin{equation*}
        \begin{aligned}
            &\lim\limits_{n\to \infty} \frac{\E[t_{k,k}]}{\sqrt{Var(t_{k,k})}} - 
		\frac{\E[t_{k,k} -t_{k-1,k}]}{\sqrt{Var(t_{k,k} -t_{k-1,k})}} = \\
            &= \frac{\E[t_{k,k} -t_{k-1,k}]}{\sqrt{Var(t_{k,k} -t_{k-1,k})}} \cdot
            \lim\limits_{n\to \infty} \frac{\E[t_{k,k}]}{\sqrt{Var(t_{k,k})}}\frac{\sqrt{Var(t_{k,k} -t_{k-1,k})}}{\E[t_{k,k} -t_{k-1,k}]} -1 \\
            &= \frac{\E[t_{k,k} -t_{k-1,k}]}{\sqrt{Var(t_{k,k} -t_{k-1,k})}} \cdot 0 \\
            &=0.
        \end{aligned}    
    \end{equation*}
 
    It is thus possible to take $n$ large enough so that
	\begin{equation}
		\label{eq:condition for second term}
		\left| \frac{\E[t_{k,k}]}{\sqrt{Var(t_{k,k})}} - 
		\frac{\E[t_{k,k} -t_{k-1,k}]}{\sqrt{Var(t_{k,k} -t_{k-1,k})}} \right| < \frac{\varepsilon}{2}.
	\end{equation}
	
	Now recall that, by Chebyshev's inequality, for a random variable $X$ with finite non-zero variance $\sigma^2$ (and finite expected value $\mu$) it holds that for any positive real number $r$
	\[
	\mathbb{P} \left( \left|X - \mu \right| \geq r \sigma \right) \leq \frac{1}{r^2}.
	\]
	
\noindent This, together with the condition in equation (\ref{eq:condition for second term}), implies that
	\begin{align*}
		&\mathbb{P} \left[ \left| \frac{(t_{k,k} -t_{k-1,k}) -  \mathbb{E}[t_{k,k} -t_{k-1,k}]}{\sqrt{Var(t_{k,k} -t_{k-1,k})}} - \frac{(t_{k,k} -t_{k-1,k}) -  \mathbb{E}[t_{k,k}]}{\sqrt{Var(t_{k,k})}} \right| > \varepsilon \right] \\
        &= \mathbb{P} \left[\left| \left( 
        \frac{t_{k,k} -t_{k-1,k}}{\sqrt{Var(t_{k,k} -t_{k-1,k})}} - \frac{t_{k,k} -t_{k-1,k}}{\sqrt{Var(t_{k,k})}}\right) + \left(
        \frac{\E[t_{k,k}]}{\sqrt{Var(t_{k,k})}} - 
		\frac{\E[t_{k,k} -t_{k-1,k}]}{\sqrt{Var(t_{k,k} -t_{k-1,k})}} \right)  \right| > \varepsilon\right] \\
		& \leq \mathbb{P} \left[ (t_{k,k} -t_{k-1,k})\left| \frac{1}{\sqrt{Var(t_{k,k})}} - \frac{1}{\sqrt{Var(t_{k,k} -t_{k-1,k})}} \right| > \frac{\varepsilon}{2} \right] \\
		&\leq 4 \varepsilon^{-2} \left(\frac{\sqrt{Var(t_{k,k} -t_{k-1,k})}}{\sqrt{Var(t_{k,k})}}-1 \right)^2,
	\end{align*}
	\noindent which by Claim 1 goes to zero for any fixed $\varepsilon >0$. In conclusion, the right side of equation (\ref{eq:prob from Morse ineq}) is bounded above by $\Phi(\alpha + \varepsilon)$ in the limit of $n \to \infty$, and therefore the central limit result for $\beta_{k,k}$ follows.

\medskip	
	\noindent We proceed now to prove the three claims the above argument relies on.

	\noindent\underline{Claim 1:} $\lim\limits_{n \to \infty} \frac{Var(t_{k,k})}{Var(t_{k,k} - t_{k-1,k})} =1$\\
  By definition, $Var(t_{k,k}) = \mathbb{E}[t_{k,k}^2] - \mathbb{E}[t_{k,k}]^2$.
		
		Let $\chi_I$ be the indicator function taking value $1$ if a tuple $I$ spans a face in $EMC_{k,k}(G)$, i.e. if $I$ induces in $G$ a path $p_I$ of length $k$.
		So,
		\[
		t_{k,k} = \sum_{\substack{I \subseteq \{0,\dots,n\} \\ |I|=k+1 \\ len(p_I)=k}} \chi_I
		\]
		and
		\[
		\mathbb{E}[t_{k,k}^2] = \sum_{I,J} \mathbb{E}[\chi_I \chi_J].
		\]
		Therefore, $Var(t_{k,k}) = \sum_{I,J} \mathbb{E}[\chi_I \chi_J] - \mathbb{E}[t_{k,k}]^2$.			
		
		The second summand is, as we noticed in Theorem \ref{thm:asymp size BettiEMH}, $\left(\binom{n}{k+1}(k+1)! p^k\right)^2$.
		The first summand requires a little more work.
		Indicating by $j$ the number of vertices that the tuples $I$ and $J$ might share we get
		\[
		\sum_{I,J} \mathbb{E}[\chi_I \chi_J] = 
		\binom{n}{k+1}(k+1)! \sum_{j=0}^{k+1} \binom{k+1}{j} \binom{n - (k+1)}{k+1 -j} p^{2k - f(j)},
		\] 
		where $f(j)$ is the number of edges shared by the induced paths $p_I$ and $p_J$.
		Since the number $f(j)$ of shared edges will always be at most $k$ (the length of the path) and being $p<1$ we can write
		
		\begin{align*}
			&\binom{n}{k+1}(k+1)! \sum_{j=0}^{k+1} \binom{k+1}{j} \binom{n - (k+1)}{k+1 -j} p^{2k - f(j)}\\
			&\leq 
			\binom{n}{k+1}(k+1)! \sum_{j=0}^{k+1} \binom{k+1}{j} \binom{n - (k+1)}{k+1 -j} p^{2k - k} \\
			&= \binom{n}{k+1}(k+1)! p^k \sum_{j=0}^{k+1} \binom{k+1}{j} \binom{n - (k+1)}{k+1 -j} \\
			&= \binom{n}{k+1}^2 (k+1)!p^k,
		\end{align*}
		where the last equality holds because of Vandermond's identity.
		
		So now, 
		
		\begin{align*}
			Var(t_{k,k}) &= \sum_{I,J} \mathbb{E}[\chi_I \chi_J] - \mathbb{E}[t_{k,k}]^2 \\
			&\leq  \binom{n}{k+1}^2(k+1)! p^k - \binom{n}{k+1}^2((k+1)!)^2 p^{2k} \\
			&= \binom{n}{k+1}^2((k+1)!)^2 p^{2k} (((k+1)!)^{-1}p^{-k}-1) \\
			&\sim n^{2(k+1)}p^{2k}\left(\frac{p^{-k}}{(k+1)!}-1\right).
		\end{align*}
		
		Thus
		\[
		\lim\limits_{n\to \infty} \frac{Var(t_{k,k})}{n^{2(k+1)}p^{2k}(p^{-k}/(k+1)!-1)} = 1.
		\]

        Performing similar computations to estimate $Var(t_{k-1,k})$ we obtain
        \[
        Var(t_{k-1,k}) \leq  \binom{n}{k+1}^2(k+1)! (1-p^{k-1}) p^k - \binom{n}{k+1}^2((k+1)!)^2 (1-p^{k-1})^2 p^{2k},
        \]
        
		and from this it follows that $\frac{Var(t_{k-1,k})}{Var(t_{k,k})} = o(1)$.
		
		Now, call $I$ and $J$ tuples in $EMC_{k-1,k}(G)$ and $EMC_{k,k}(G)$ respectively.
		Expanding the same way as above we find
        
		\begin{align*}
			Cov(t_{k-1,k},t_{k,k}) &= \sum_{I,J} \mathbb{E}[\chi_I \chi_J] - \mathbb{E}[t_{k-1,k}] \mathbb{E}[t_{k,k}] \\
			&=\binom{n}{k+1}(k+1)! (1-p^{k-1}) \sum_{j=0}^{k+1} \binom{k}{j} \binom{n-(k+1)}{k+1-j} p^{2k - f(j)} - \\
                & - \binom{n}{k+1}(k+1)!p^k(1-p^{k-1}) \binom{n}{k+1}(k+1)!p^k \\
			&\leq \binom{n}{k+1}^2(k+1)! p^k (1-p^{k-1}) - \binom{n}{k+1}^2\left((k+1)!\right)^2 p^{2k} (1-p^{k-1})\\
			&= \binom{n}{k+1}^2 \left((k+1)!\right)^2 p^{2k} (1-p^{k-1}) \left(\frac{p^{-k}}{(k+1)!} -1\right).
		\end{align*} 
		
		Therefore
		\[
		\lim\limits_{n\to \infty} \frac{Cov(t_{k-1,k},t_{k,k})}{n^{2(k+1)}p^{2k}(1-p^{k-1})(p^{-k}/(k+1)! -1)} =
		1,
		\]
		which implies that $\frac{Cov(t_{k-1,k},t_{k,k})}{Var(t_{k,k})} = o(1)$, completing the proof of the claim.
		
	\medskip	\noindent\underline{Claims 2 and 3:}  $\frac{t_{k,k} -  \mathbb{E}[t_{k,k}]}{\sqrt{Var(t_{k,k})}} \Rightarrow N(0,1)$ as $n \to \infty,$
 
		\hspace{1.3cm} and $\frac{(t_{k,k} -t_{k-1,k}) -  \mathbb{E}[t_{k,k} -t_{k-1,k}]}{\sqrt{Var(t_{k,k} - t_{k-1,k})}} \Rightarrow N(0,1)$ as $n \to \infty$.
  
  The proofs of claims 2 and 3 are a consequence of an abstract normal approximation theorem for dissociated random variables showed in \cite{barbour1989central}. 
		
		We recall that, given a set $A$ of $n$-tuples, a set of random variables $\{X_{\boldsymbol{a}}:\boldsymbol{a}=(a_1,\dots,a_n)\in A\}$ is \emph{dissociated} if there exist two subsets of random variables $\{X_{\boldsymbol{a}}:\boldsymbol{a} \in A'\}$ and $\{X_{\boldsymbol{a}}:\boldsymbol{a} \in A''\}$ that are independent whenever $(\cup_{\boldsymbol{a} \in A'} (a_1,\dots,a_n)) \cap (\cup_{\boldsymbol{a} \in A''} (a_1,\dots,a_n)) = \emptyset$. 
		
		Now let $W=\sum_{\boldsymbol{a} \in A} X_{\boldsymbol{a}}$, and for each $\boldsymbol{a} \in A$ define $D_{\boldsymbol{a}}=\{\boldsymbol{\alpha} \in A: (\alpha_1,\dots,\alpha_n) \cap (a_1,\dots,a_n) \neq \emptyset \}$, i.e. $D_{\boldsymbol{a}}$ is a dependency neighborhood for $\boldsymbol{a}$.
		
		It is proved in \cite[Theorem 1]{barbour1989central} that if $\mathbb{E}[X_{\boldsymbol{a}}]=0$ and $\mathbb{E}[W^2]=1$, then
		\begin{equation}
        \label{eq:d(W,Z)}
		d_1(W,Z) \leq K \sum_{\boldsymbol{a} \in A} \sum_{\boldsymbol{\alpha},\boldsymbol{\beta} \in D_{\boldsymbol{a}}} 
		\left( \mathbb{E}|X_{\boldsymbol{a}}X_{\boldsymbol{\alpha}}X_{\boldsymbol{\beta}}| 
		+  \mathbb{E}|X_{\boldsymbol{a}}X_{\boldsymbol{\alpha}}|\mathbb{E}|X_{\boldsymbol{\beta}}|\right),
		\end{equation}
		\noindent where $d_1(\cdot,\cdot)$ is the \emph{$L_1$-Wasserstein distance}, $Z$ is a standard normal random variable, $K$ is a universal constant.
		
		To show that $(t_{k,k} - t_{k-1,k})$ satisfies a central limit theorem, take the index set $A$ to be the potential edge set of $k$-paths induced by simplices in $EMC_{k-1,k}(G)$ and $EMC_{k,k}(G)$. 
		That is, the potential edge set of $k$-paths spanning a set of $(k+e)$ ($e \in \{0,1\}$) vertices in $G(n,p)$.
		We then associate to each $\boldsymbol{a} \in A$ the set $V_{\boldsymbol{a}}$ of corresponding vertices and we define
		\[
		X_{\boldsymbol{a}} = \frac{1}{\sqrt{Var_{t_{k,k}}}}(\chi_{V_{\boldsymbol{a}}} - \mathbb{E}[\chi_{V_{\boldsymbol{a}}}]).
		\]
		
	\noindent With this definition, the $\{X_{\boldsymbol{a}}\}$ are dissociated, $W=\sum_{\boldsymbol{a} \in A} X_{\boldsymbol{a}}$, $\mathbb{E}[W]=0$ and $Var(W)=\mathbb{E}[W^2]=1.$ The proof of the claim is analogous to \cite[Theorem 2.3]{kahle2013limit}, and for completeness we summarize here the main steps. 
    Specifically, what is left to do is to bound the first and second half of the the sum in equation \ref{eq:d(W,Z)} and show that they both tend to zero as $n\to\infty$.

    In both cases the technique used by Kahle and Meckes is to partition the dependency neighborhood $D_{\boldsymbol{a}}$ for each $\boldsymbol{a}$ into subsets of indices $D_{\boldsymbol{a}}^e$ whose corresponding spanning set of vertices $V_{\boldsymbol{a}}^e$ has size $k+e$.
    Then, decomposing the two sums (as in the variance estimate) by the size of the intersection of the vertex sets, it is possible to bound the first and second half of the sum with $o\left(\frac{1}{n}\right)$, and conclude that the whole quantity tends to zero as $n$ tends to infinity.
  \qedhere
		
\end{proof}

\begin{remark}
    Note that we are referring to the original central limit theorem proof \cite[Theorem 2.3]{kahle2013limit} and not the erratum presented in \cite[Theorem 1.1]{kahle2015erratum}. 
    This is because Kahle and Meckes needed to slightly restrict the $p$-interval where the estimate for the expected value holds by setting $p=\omega\left(n^{-1/k+\delta}\right)$ and $p=o\left(n^{-1/(k+1)-\delta}\right)$ with $\delta >0$ to avoid that the difference in means of the upper and lower bounds for the Betti numbers under consideration was too large relative to the normalization.
    In our case we were able to leave the assumption that $p=1-O\left(\frac{\ln(n)}{n}\right)$. 
\end{remark}

\section{Eulerian magnitude homology for random geometric graphs on $T^2$}
\label{sec:EMH_RGG}

Like Erd\H{o}s-R\'{e}nyi graphs, \emph{random geometric graphs} are a fundamental model for random graphs across a variety of disciplines. First introduced by Edward N. Gilbert in \cite{gilbert1961random} to model communications between radio stations, the original model involved sampling points from a Poisson point process and connecting points that fall within a fixed distance. Another common model involves selecting a fixed number of points independently and identically distributed on the space. The two models are closely connected, as observed by Penrose in \cite[Section 1.7]{penrose2003random}. Because it is somewhat more commonly studied in the context of stochastic topology, here we will use the latter model.

\begin{definition}
	Let $(X ,d, m)$ be a metric space equipped with be a Borel probability measure $m$.
	Given a positive real number $r >0$ and positive integer $n$, the \emph{Random Geometric Graph (RGG) model} $RGG(n, r, (X, d, m))$ is the probability distribution on graphs with $n$ vertices $\{x_1, \dots, x_n\}$ given by
    
	\begin{itemize}
		\item selecting a collection of points $\{p_i\}_{i=1}^n$ in $X$ according to the to the probability measure $m^{\otimes n}$ on $X^n$, and
		\item for any $i\neq j \in \{1, \dots, n\},$ taking the edge $\{x_i, x_j\}$ to be in $G$ if and only if $d(p_i,p_j)\leq r$.
	\end{itemize}
	
\end{definition}

We are particularly interested in bounded regions in Euclidean spaces with the uniform probability measure, as these are of central interest in many applications. However, to avoid boundary effects we will follow \cite{yu2009computing}, and instead consider the flat torus of area $T^2_A = [0,\sqrt{A})^2,$ with the uniform probability measure, $RGG(n, r,(T^2_A, d_T, u)),$ which we will abbreviate to $G(n, r, A).$ 
By moving to the torus, we homogenize the probability of a single edge being part of our random geometric graph, as every point $x$ lies at the center of a euclidean ball of radius $r$, each of which is equally likely to contain other points. As we are again interested in studying limiting phenomena as $n \to \infty,$ we will take $r = n^{-q}$ for a fixed parameter $q.$ In particular, this will ensure that $r << \sqrt{A},$ so our euclidean balls around points do not self-intersect.

Our results in this context have similar flavor to those about ER graphs from Section \ref{sec:EMH_ER}, however they rely on somewhat different subgraph counting methods. However, in Theorem \ref{thm:vanishing threshold RGG}, we again find a threshold for $q$ beyond which $EMH_{k,k}(G(n,n^{-q},A))$ vanishes in expectation as $n \to \infty.$ And, in Theorem \ref{thm:asympBettiRGG} we again present an explicit formula for the asymptotic size of the Betti numbers $\E[\beta_{k,k}]$ as $n \to \infty,$ as well as a corresponding Central Limit Theorem in Theorem \ref{thm:RGG_CLT}.

\medspace

We start by recalling results presented by Yu in \cite{yu2009computing} which will provide us with fundamental tools in our analysis of RGGs. 

\begin{theorem}[Theorem 1 \cite{yu2009computing}]
	\label{thm:prob_edge_RGG}
	Let $G\sim G(n,r,A)$ and let $x_i\neq x_j$ be vertices of $G$. The probability of the edge $\{x_i,x_j\}$ appearing in $G$ is $\frac{\pi r^2}{|A|}$.
\end{theorem}

It is shown in \cite[Theorem 2]{yen2004link} that the occurrences of arbitrary pair-wise edges in RGGs are independent even if they share one end vertex. Combining this fact with the statement of Theorem \ref{thm:prob_edge_RGG} we obtain the following.

\begin{cor}[Corollary 3 \cite{yu2009computing}]
	\label{cor:prob_E2}
	Let $G\sim G(n,r,A)$ with vertices $V = \{x_1, \dots, x_n\}$, and let $\{x_{i_1}, x_{i_2}\} \neq \{x_{i_3}, x_{i_4}\}\in {V \choose 2}$. The probability of both $\{x_{i_1}, x_{i_2}\}$ and $\{x_{i_3}, x_{i_4}\}$  appearing as edges of $G$ is $\left(\frac{\pi r^2}{|A|}\right)^2$.
\end{cor}

By inductively applying this corollary, we can extend this argument to any  $k$-trail. 

\begin{lemma}
	\label{lem:RGG_trails_indep}
	Let $G\sim G(n,r,A)$ with vertices $V = \{x_1, \dots, x_n\}$, and let $(x_0, \dots, x_k) \in ET_k(G)$ be a  $k$-trail. The probability of the edges $\{x_i, x_{i+1}\}_{i=0}^{k-1}$ all appearing  as edges of $G$ is $\left(\frac{\pi r^2}{|A|}\right)^{k}$.
\end{lemma}

\begin{proof}
Starting from the base step stated in Corollary \ref{cor:prob_E2}, consider the eulerian $(k-1)$-trail $(x_0,\dots,x_{k-1}) \in ET_{k-1}(G)$ and assume the probability of the edges $\{x_i, x_{i+1}\}_{i=0}^{k-2}$ all appearing  as edges of $G$ is $\left(\frac{\pi r^2}{|A|}\right)^{k-1}$.

Consider now the the  $k$-trail $(x_0, \dots, x_{k-1}, x_k) \in ET_k(G)$.
Since the edge $(x_{k-1}, x_k)$ shares only the end vertex $x_{k-1}$ with the the $(k-1)$-trail $(x_0,\dots,x_{k-1})$, we can conclude that the probability of the edges $\{x_i, x_{i+1}\}_{i=0}^{k-1}$ all appearing  as edges of $G$ is $\left(\frac{\pi r^2}{|A|}\right)^{k-1} \cdot \frac{\pi r^2}{|A|} = \left(\frac{\pi r^2}{|A|}\right)^{k}$.
\end{proof}

To prove our vanishing thresholds, we will leverage this computation, along with families of graphs called ``Y-graphs" in \cite{yu2009computing}. We require only a particular subset of this collection of families of graphs which we will call \emph{restricted Y-classes}. In particular, our families  will satisfy the conditions of \cite[Corollary 18]{yu2009computing}, which we will restate in our context in Corollary \ref{cor:prob_Y}. First, we require some terminology.

\begin{figure}[t]
	\begin{minipage}{0.18\linewidth}
		\centering
\begin{tikzpicture}[node distance={15mm}, thick, main/.style = {draw, circle}]
 		\node[main] (0) {}; 
 		\node[main] (1) [below of=0] {};  
 		\node[main] (2) [right of=0] {};
 		\node[main] (3) [below of=2] {};
	
 		\draw (0) -- (1);
 		\draw (0) -- (3);
 		\draw[dashed] (0) -- (2);	
            \draw[dashed] (1) -- (2);
 		\draw (1) -- (3);
 		\draw[dashed] (2) -- (3);
 						
 	\end{tikzpicture} 
  \subcaption*{$\mathcal{G}_1$}
	\end{minipage}
	\hfill
	\begin{minipage}{0.18\linewidth}
		\centering
\begin{tikzpicture}[node distance={15mm}, thick, main/.style = {draw, circle}]
 		\node[main] (0) {}; 
 		\node[main] (1) [below of=0] {};  
 		\node[main] (2) [right of=0] {};
 		\node[main] (3) [below of=2] {};
	
 		\draw[dashed] (0) -- (1);
 		\draw (0) -- (3);
 		\draw[dashed] (0) -- (2);	
            \draw[dashed] (1) -- (2);
 		\draw (1) -- (3);
 		\draw[dashed] (2) -- (3);
 						
 	\end{tikzpicture} 
  \subcaption*{$\mathcal{G}_2$}
	\end{minipage}
 \hfill
 \begin{minipage}{0.18\linewidth}
 \centering
 \begin{tikzpicture}[node distance={15mm}, thick, main/.style = {draw, circle}]
 		\node[main] (0) {}; 
 		\node[main] (1) [below of=0] {};  
 		\node[main] (2) [right of=0] {};
 		\node[main] (3) [below of=2] {};
	
 		\draw[dashed] (0) -- (1);
 		\draw[dashed] (0) -- (3);
 		\draw[dashed] (0) -- (2);	
            \draw[dashed] (1) -- (2);
 		\draw (1) -- (3);
 		\draw[dashed] (2) -- (3);
 						
 	\end{tikzpicture} 
  \subcaption*{$\mathcal{G}_3$}
	\end{minipage}
  \hfill
    \begin{minipage}{0.18\linewidth}
 \centering
 \begin{tikzpicture}[node distance={15mm}, thick, main/.style = {draw, circle}]
 		\node[main] (0) {}; 
 		\node[main] (1) [below of=0] {};  
 		\node[main] (2) [right of=0] {};
 		\node[main] (3) [below of=2] {};
	
 		\draw[dashed] (0) -- (1);
 		\draw[dashed] (0) -- (3);
 		\draw[dashed] (0) -- (2);	
            \draw[dashed] (1) -- (2);
 		\draw[dashed] (1) -- (3);
 		\draw[dashed] (2) -- (3);
 						
 	\end{tikzpicture} 
  \subcaption*{$\mathcal{G}_4$}
	\end{minipage}
 \hfill
  \begin{minipage}{0.18\linewidth}
 \centering
 \begin{tikzpicture}[node distance={15mm}, thick, main/.style = {draw, circle}]
 		\node[main] (0) {}; 
 		\node[main] (1) [below of=0] {};  
 		\node[main] (2) [right of=0] {};
 		\node[main] (3) [below of=2] {};
	
 		\draw[dashed] (0) -- (1);
 		\draw[dashed] (0) -- (3);
 		\draw (0) -- (2);	
            \draw[dashed] (1) -- (2);
 		\draw (1) -- (3);
 		\draw[dashed] (2) -- (3);
 						
 	\end{tikzpicture} 
  \subcaption*{$\mathcal{G}_5$}
	\end{minipage}
 \hfill
   \begin{minipage}{0.18\linewidth}
 \centering
 \begin{tikzpicture}[node distance={15mm}, thick, main/.style = {draw, circle}]
 		\node[main] (0) {}; 
 		\node[main] (1) [below of=0] {};  
 		\node[main] (2) [right of=0] {};
 		\node[main] (3) [below of=2] {};
	
 		\draw[dashed] (0) -- (1);
 		\draw (0) -- (3);
 		\draw[dashed] (0) -- (2);	
            \draw (1) -- (2);
 		\draw (1) -- (3);
 		\draw[dashed] (2) -- (3);
 						
 	\end{tikzpicture} 
  \subcaption*{$\mathcal{G}_6$}
	\end{minipage} 
 \hfill
     \begin{minipage}{0.18\linewidth}
 \centering
 \begin{tikzpicture}[node distance={15mm}, thick, main/.style = {draw, circle}]
 		\node[main] (0) {}; 
 		\node[main] (1) [below of=0] {};  
 		\node[main] (2) [right of=0] {};
 		\node[main] (3) [below of=2] {};
	
 		\draw (0) -- (1);
 		\draw[dashed] (0) -- (3);
 		\draw[dashed] (0) -- (2);	
            \draw (1) -- (2);
 		\draw (1) -- (3);
 		\draw[dashed] (2) -- (3);
 						
 	\end{tikzpicture} 
  \subcaption*{$\mathcal{G}_7$}
	\end{minipage}
 \hfill  
  \begin{minipage}{0.18\linewidth}
 \centering
 \begin{tikzpicture}[node distance={15mm}, thick, main/.style = {draw, circle}]
 		\node[main] (0) {}; 
 		\node[main] (1) [below of=0] {};  
 		\node[main] (2) [right of=0] {};
 		\node[main] (3) [below of=2] {};
	
 		\draw (0) -- (1);
 		\draw[dashed] (0) -- (3);
 		\draw (0) -- (2);	
            \draw (1) -- (2);
 		\draw[dashed] (1) -- (3);
 		\draw (2) -- (3);
 						
 	\end{tikzpicture} 
  \subcaption*{$\mathcal{G}_8$}
	\end{minipage}
 \hfill
  \begin{minipage}{0.18\linewidth}
 \centering
 \begin{tikzpicture}[node distance={15mm}, thick, main/.style = {draw, circle}]
 		\node[main] (0) {}; 
 		\node[main] (1) [below of=0] {};  
 		\node[main] (2) [right of=0] {};
 		\node[main] (3) [below of=2] {};
	
 		\draw (0) -- (1);
 		\draw[dashed] (0) -- (3);
 		\draw (0) -- (2);	
            \draw (1) -- (2);
 		\draw (1) -- (3);
 		\draw (2) -- (3);
 						
 	\end{tikzpicture} 
  \subcaption*{$\mathcal{G}_{10}$}
	\end{minipage}
	\caption{Class graphs on four vertices for which $\Gamma(\mathcal{G})$ is defined to be a restricted $Y$-class. Numbering is preserved from \cite{yu2009computing}. Recall from Definition \ref{def:alpha(G) and omega(G)} that solid lines identify the edges in the set $E_S$ and dashed lines denote the edges in $E_B$, and for every class graph $\mathcal{G}_i$ the minimal and maximal graphs (under inclusion) in $\Gamma(\mathcal{G}_i)$ are given by $\alpha(\mathcal{G}_i)=(V,E_S)$ and $\omega(\mathcal{G}_i)=(V,E_S \cup E_B)$.}
	\label{fig:four_vertices_class_graphs}
\end{figure}
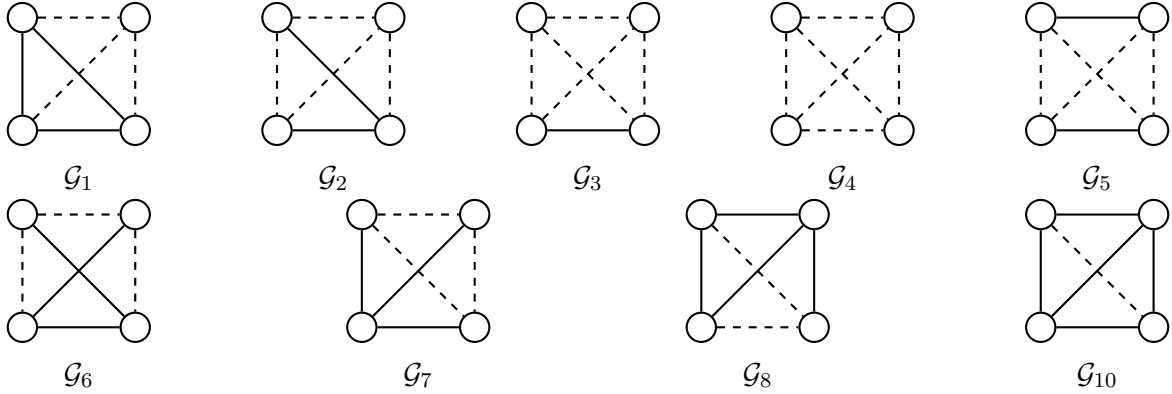

\begin{definition}
        Let $\mathcal{G}=(V, E_S, E_B)$ be a class graph, and let $V' \subseteq V$. The \emph{full class subgraph of $\mathcal{G}$ on $V'$} is $\mathcal{G}\vert_{V'} = \left(V', E_S \cap {V' \choose 2}, E_B \cap {V' \choose 2}\right).$ The \emph{contraction of $\mathcal{G}$ along $V'$} is the class graph $\mathcal{G} / V' = (V / V', E_S^{V'}, E_B^{V'})$ where, writing $v'$ for the element of the quotient set $V / V'$ corresponding to $V',$ we have   
        \begin{align*}
            E_S^{\mathcal{G}/V'} = \left(E_S \cap {V \setminus V' \choose 2}\right) \cup \{\{v, v'\} \; : &\; \{v, w\} \in E_S \text{ for some } w \in V'\} \\
            E_B^{\mathcal{G}/V'}=\left(E_B \cap {V \setminus V' \choose 2}\right) \cup \{\{v, v'\} \;:\; & \{v, w\} \in E_B \text{ for some } w \in V', \\
            &\text{ and, for all } w \in V', \{v,w\}\not\in E_S\}.
        \end{align*}
\end{definition}

Finally, we can define

\begin{definition}[restricted $Y$-class, c.f. \cite{yu2009computing}]
\label{def:restrictedY}
	Let $\mathcal{G} = (V, E_S, E_B)$ be a class graph and $\Gamma(\mathcal{G})$ the graphs of class $\mathcal{G}.$    
    Then $\Gamma(\mathcal{G})$ is \emph{restricted $Y$-class} if $\mathcal{G}$ can be constructed according to the following two rules:
	\begin{itemize}
		\item[R1] $\mathcal{G}$ is a class graph on three vertices, or one of the complete class graphs on four vertices pictured in  Figure \ref{fig:four_vertices_class_graphs}, or
		\item[R2] Take $V^1, \dots, V^k \subseteq V$ so that $V^i \cap V^j = \varnothing$ if $i \neq j,$ and so that for each $i=1, \dots, k,$  $\Gamma(\mathcal{G}|_{V^i})$ is a restricted $Y$-class. Write $\phi(\mathcal{G})$ for the class graph obtained by sequentially contracting $\mathcal{G}$ along each $V^i$. If $\alpha(\phi(\mathcal{G}))$ is a tree and $\omega(\phi(\mathcal{G}))$ is a complete graph, then $\Gamma(\mathcal{G})$ is a restricted $Y$-class, where $\alpha(\phi(\mathcal{G}))$ and $\omega(\phi(\mathcal{G}))$ are as in Definition \ref{def:alpha(G) and omega(G)}.
	\end{itemize} 
\end{definition}

Our definition of restricted $Y$-classes involves a subset of the rules for producing \emph{$Y$-graphs} in \cite{yu2009computing}, so every restricted $Y$-class is a $Y$-graph. Thus, we obtain the following corollary to \cite[Corollary 18]{yu2009computing}.
\begin{cor}\label{cor:prob_Y}
    Suppose $\mathcal{G} = (V, E_S, E_B)$ is a class graph so that $\alpha(\mathcal{G})$ is a tree and $\Gamma(\mathcal{G})$ is a restricted $Y$-class. Let $G \sim G(n, r, A)$ be a geometric random graph in $\mathcal{G}$. The probability that there is some full subgraph of $H \subseteq G$ with $|V|$ vertices so that $H \cong K$ for some $K \in \Gamma(\mathcal{G})$ is $(\pi r^2 / A)^{|V|-1}.$
\end{cor}

In order to count eulerian magnitude homology classes, we are interested in certain class graphs $\mathcal{G}$ for which graphs of the type pictured in Figure \ref{fig:subgraph_emh} appear in $\Gamma(\mathcal{G}).$

\begin{lemma}
	\label{lem:Ygraph}
	Let $G$ be a graph and $k \geq 2$. Suppose $\bar{x}\in ET_{k,k}(G)$ is a  $k$-trail in $G$ for which $\partial_{k,k}\bar{x}=0.$ Let $H(\bar{x})$ be the graph from the proof of Lemma \ref{lem:vanishing threshold_single tuple}. Then there is a class graph $\mathcal{G}$ such that $G \in \mathcal{G}$, $\alpha(\mathcal{G})$ is a tree, $\Gamma(\mathcal{G})$ is a restricted $Y$-class, and $H \cong K$ for some $K \in \Gamma(\mathcal{G}).$
\end{lemma}

\begin{proof} 
    Let $\bar{x} = (x_0, x_1,\dots x_k)$ and recall that $H(\bar{x})$ is the graph on vertices $V = \{x_0, \dots, x_k\}$ with edges $E_H = \{\{x_{i-1},x_{i+1}\}\}_{i=0}^{k-1} \cup \{\{x_{i-1}, x_{i+1}\}_{i=1}^{k-1},$ as illustrated in Figure \ref{fig:subgraph_emh} for $k=5.$	We will proceed to build the required restricted $Y$-class containing $H(\bar{x})$ by (very similar) cases in $k$.
 
    If $k=2,$ $H$ is the complete graph on three vertices. This is a member of the restricted $Y$-class $\Gamma\left((\{x_0, x_1, x_2\}, \{\{x_0, x_1\}, \{x_0, x_2\}\}, \{\{x_1, x_2\}\})\right),$ whose minimal element is a tree. 
    
    If $k=3,$ $H$ is a graph on four vertices with five edges, and so is isomorphic to a graph in each of the $\Gamma(\mathcal{G}_i)$ for $\mathcal{G}_i$ found in Figure \ref{fig:four_vertices_class_graphs}. In particular, if we take $\Gamma(\mathcal{G}_7),$ the minimal element is a tree, and we have constructed the required class.
	
	Assume now that $k\geq 4$. In each case, we will partition the vertices of $H(\bar{x})$ to construct the required restricted $Y$ class using $R2$ of Definition \ref{def:restrictedY}. 
 
    If $k = 4\ell + 1$ for some $\ell > 1,$ let $V^i = \{v_{1+4i}, v_{2+4i}, v_{3+4i}, v_{4+4i}\},$ $i=0, \dots, \ell-1,$ so the $V^i$ are disjoint and all vertices but $v_0$ and $v_k$ are contained in some $V^i.$ Each of the subgraphs $H(\bar{x})|_{V^i}$ is again a graph on four vertices with five edges. Define class graphs $$\mathcal{H}_i = \left(V^i, E_S^i = {V^i \choose 2} \setminus \{\{v_{2+4i}, v_{4+4i}\}\}, E_B^i = \{\{v_{2+4i}, v_{4+4i}\}\}\right)$$ and observe that $H(\bar{x})|_{V^i} = \alpha(\mathcal{H}_i),$ as in Figure \ref{fig:construction_Ygraph}(a). Define $\mathcal{H}$ to be the class graph $\mathcal{H} = (V, E_S^\mathcal{H} = E_H, E_B^\mathcal{H} = {V \choose 2} \setminus E_H),$ so again $H(\bar{x}) = \alpha(\mathcal{H}),$ and $\omega(\mathcal{H})$ is the complete graph on $V.$ Contracting $\mathcal{H}$ along each of the $\mathcal{H}_i$ sequentially, and calling the vertices corresponding to $\mathcal{H}_i$ in the contraction by $h_i,$ we obtain the new class graph $\phi(\mathcal{H}) =  (V^\phi, E_S^\phi, E_B^\phi)$ with \begin{align*}
    V^\phi &= \{v_0, h_1, \dots, h_\ell, v_{k}\}, \\
    E_S^\phi &= \{\{v_0, h_1\}, \{h_\ell, v_k\}\} \cup \{\{h_i, h_{i+1}\}_{i=1}^{\ell-1}, \\
    E_B^\phi &= {V^\phi \choose 2} \setminus E_S^\phi.
    \end{align*}
    Here, $\alpha(\phi(\mathcal{H}))$ is a line graph, thus a tree, and $\omega(\phi(\mathcal{H}))$ is a complete graph, as required. Thus, $\mathcal{H}$ is a restricted $Y$-class.
   	
	\begin{figure}[t]
		\begin{minipage}{\linewidth}
			\begin{minipage}{0.45\linewidth}
				\centering
				\begin{tikzpicture}[node distance={15mm}, thick, main/.style = {draw, circle}]
					\node[main] (0) {$x_0$}; 
					\node[main] (1) [above right of=0] {$x_1$};  
					\node[main] (2) [below right of=1] {$x_2$};
					\node[main] (3) [above right of=2] {$x_3$};
					\node[main] (4) [below right of=3] {$x_4$};
					\node[main] (5) [above right of=4] {$x_5$};
					\draw (0) -- (1);
					\draw[red] (1) -- (2);
					\draw[red] (2) -- (3);
					\draw[red] (3) -- (4);
					\draw (4) -- (5);
					
					\draw (0) -- (2);
					\draw[red] (1) -- (3);
					\draw[red] (2) -- (4);
					\draw (3) -- (5);
                    \draw[dashed, red] (1) -- (4);
				\end{tikzpicture} 
				\subcaption{In red, class graph  $\mathcal{H}_1$ visualized inside $H(\bar{x}).$ To expand this to the class graph $\mathcal{H},$ we add all non-pictured edges to the set $E_B^\mathcal{H}.$ }
			\end{minipage}
			\hfill
			\begin{minipage}{0.45\linewidth}
				\centering
				\begin{tikzpicture}[node distance={15mm}, thick, main/.style = {draw, circle}]
					
					\node[main] (0) {$x_0$};
					\node[main] (1) [above right=15mm of 0] {$x_1$};
					\node[main, below= 1mm of 1] (4) {$x_4$};
					\node[main, below left= -1mm and 3mm of 1] (2) {$x_2$};
					\node[main, below right=-1mm and 3mm of 1] (3) {$x_3$};
					
					\node[circle, draw=red, fit=(1) (2) (3) (4), inner sep=-1mm] (all) {};
					
					\node[main] (5) [right=5mm of all] {$x_5$};
					
					\draw (0) -- (all);
                    \path[every node/.style={font=\sffamily\small}]
			(0) edge [bend right=50, dashed] node {} (5);
				    \draw (all) -- (5);				
				\end{tikzpicture}
				\subcaption{The contraction $\varphi(\mathcal{H})$. Here the large red circle is the vertex $h_1,$ representing the quotient by the set of original vertices $\{x_1, \dots, x_4\}.$ supporting $\mathcal{H}_1$.}
			\end{minipage}
		\end{minipage}
		\caption{Demonstration of construction of restricted $Y$-class containing $H(\bar{x})$ in the proof of Lemma \ref{lem:Ygraph}.}
		\label{fig:construction_Ygraph}
	\end{figure}
	
    The other cases follow from this same framework with slight modifications. When $k = 4\ell,$ we obtain a partition of all vertices besides $v_0$ into sets of four vertices and obtain $\phi(\mathcal{H})$ as before, but missing vertex $v_k.$ When $k = 4\ell + 3$, we partition the vertices $\{v_1, \dots, v_{4\ell -1}\}$ as before into sets $V^1, \dots V^{\ell -2}$ with four vertices. We then add the set $W^{\ell-1} = \{v_{4\ell+1}, v_{4\ell+2}, v_{4\ell+3}\},$ and observe that $H(\bar{x})|_{W^{\ell-1}}$ is a complete graph on three vertices, so as in the $k=2$ case we can take the corresponding singleton restricted $Y$-class. The construction of $\mathcal{H}$ now proceeds as before, again omitting the vertex $v_k.$ When $k = 4\ell+2,$ we shift the partition so we have $V^0 = \{v_0, v_1, v_2\},$ $V^i = \{v_{-1+4i}, v_{4i}, v_{1+4i}, v_{2+4i}\} ,i=1, \dots, \ell,$ and proceed as before.	
\end{proof}

We are now able to show the following.

\begin{theorem}
\label{thm:vanishing threshold RGG}
Fix $k$ and $q>\frac{k+1}{2k}.$ Then as $n\to \infty,$ \[\E[|EMH_{k,k}(G(n, n^{-q}, A)|] \to 0.\]
\end{theorem} 

\begin{proof}
Let $G \sim G(n, n^{-q}, A),$ and suppose $\bar{x} = (x_0,\dots,x_k)\in ET_{k,k}(G)$ is a  $k$-trail in $G,$ for which $\partial_{k,k} \bar{x} = 0.$ Let $\mathcal{H}$ be the resulting restricted $Y$-class obtained in Lemma \ref{lem:Ygraph}.
By Corollary \ref{cor:prob_Y}, the probability of finding any subgraph $H\subseteq G|_{\{v_0, \dots, v_k\}}$ isomorphic to an element of $\Gamma(\mathcal{H})$ is $\left(\frac{\pi r^2}{|A|} \right)^{(k+1)-1 } = \left(\frac{\pi r^2}{|A|} \right)^{k}$. Thus, this is an upper bound for the probability of observing a graph isomorphic to $H_k$ on any subgraph on $(k+1)$ vertices. Proceeding similarly to the proof of Lemma \ref{lem:vanishing threshold_single tuple} for ER graphs, call $N_{H_k}(G(n, n^{-q},A))$ the number of subgraphs isomorphic to $H_k$ expected to appear in a graph $G\sim G(n, n^{-q},A),$ and $a_{H_k}$ for the largest number of graphs isomorphic to $H_k$ supported on a collection of $(k+1)$ vertices in $G.$ We have

\begin{align*}
	N_{H_k}(G(n, n^{-q}, A)) & \leq \binom{n}{k+1} a_H \left(\frac{\pi r^2}{|A|} \right)^{k} \\
	&\sim  \frac{n^{k+1}}{(k+1)!} a_H\left(\frac{\pi }{|A|} \right)^{k} n^{-2q k} \\
	&= \frac{a_H}{(k+1)!} \left(\frac{\pi }{|A|} \right)^{k} n^{(k+1)-2q k} \xrightarrow{n\to \infty}
	\begin{cases}
		0, \text{ if } q > \frac{k+1}{2k} \\
		\infty, \text{ if } 0 < q < \frac{k+1}{2k}.
	\end{cases}
\end{align*}

As in the ER case, the presence of $H_k$ is necessary to support an element in the kernel of $\partial_{k,k}$. As these vanish, so must $EMH_{k,k}(G).$

As in the Erd\H{o}s-R\'{e}nyi case, the subgraph induced by a combination $\sum_{i=1}^m [(x_0,\dots,x_k)_i] \in EMH_{k,k}(G)$ is less likely to appear then the one induced by a single tuple.
Therefore, showing that there is a restricted $Y$-class associated with linear combinations $\sum_{i=1}^ma_i\bar{x}^i$ suffices to prove that the vanishing threshold obtained from the count of single tuples $(x_0,\dots,x_k)$ generating a homology cycle holds for the whole eulerian magnitude homology group $EMH_{k,k}(G)$. 
To see that this is true, we will proceed similarly as in Lemma \ref{lem:Ygraph} by building the required restricted $Y$-class containing $H(\bar{x})=\Gamma(\mathcal{H}(\bar{x}))$.

Suppose $k=2$. Then $H(\bar{x})$ is the graph on four vertices in the restricted $Y$-class $\Gamma\left((\{x_0, x_1, x_2, x_1'\}, \{\{x_0, x_1\}, \{x_1, x_2\}, \{x_0, x_1'\}, \{x_1', x_2\}\}, \{\{x_1, x_1'\}\})\right),$ which is a restricted $Y$-class as described in rule R1 (contained in the class $\mathcal{G}_4$), and whose minimal element is a tree. 
    
If $k=3$, there are two options for $H(\bar{x})$. Either it is a graph on five vertices with six edges, or it is a graph on six vertices and seven edges, as illustrated in Figure \ref{fig:class graph combination k=3}.
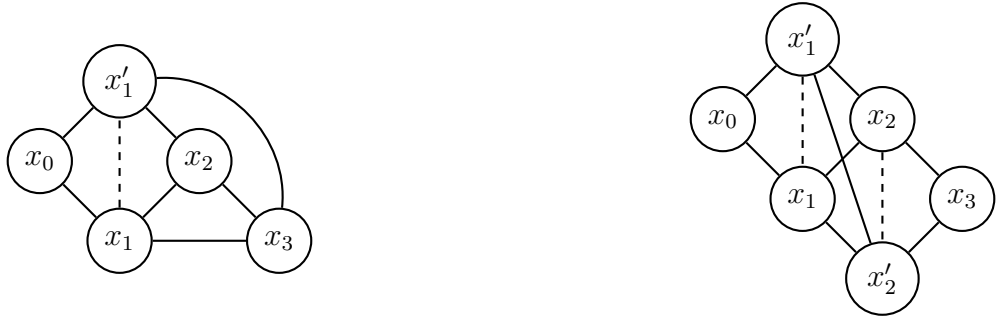
\begin{figure}[h]
	\begin{minipage}{\linewidth}
		\begin{minipage}{0.45\linewidth}		\centering			
		\begin{tikzpicture}[node distance={15mm}, thick, main/.style = {draw, circle,minimum size=0.85cm}]
			\node[main] (0) {$x_0$}; 
			\node[main] (11) [above right of=0] {$x_1'$};  
			\node[main] (2) [below right of=11] {$x_2$};
			\node[main] (1) [below right of=0] {$x_1$};
                \node[main] (3) [below right of=2] {$x_3$};
			
			\draw[dashed] (1) -- (11);
			\draw (0) -- (1);
                \draw (0) -- (11);
			\draw (1) -- (2);
			\draw (2) -- (3);
                \draw (11) -- (2);
			\draw (1) -- (3);

                \path[every node/.style={font=\sffamily\small}]
			(11) edge [bend left=50] node {} (3);
   
		\end{tikzpicture}
	\end{minipage}
    \hfill 
	\begin{minipage}{0.45\linewidth}
        \centering
		\begin{tikzpicture}[node distance={15mm}, thick, main/.style = {draw, circle,minimum size=0.85cm}]
			\node[main] (0) {$x_0$}; 
			\node[main] (11) [above right of=0] {$x_1'$}; 
                \node[main] (1) [below right of=0]{$x_1$};
			\node[main] (2) [below right of=11] {$x_2$};
			\node[main] (3) [below right of=2] {$x_3$};
			\node[main] (22) [below right of=1] {$x_2'$};
			
			\draw (0) -- (1);
                \draw (0) -- (11);
			\draw (1) -- (2);
                \draw (11) -- (2);
			\draw (2) -- (3);
			\draw (1) -- (22);
			\draw (22) -- (3);
                \draw (11) -- (22);
			
			\draw[dashed] (11) -- (1);
			\draw[dashed] (2) -- (22);
			
		\end{tikzpicture} 
	\end{minipage}
  \end{minipage}
		\captionof{figure}{(left) Subgraph $H(\bar{x})$ mandated if $H(\bar{x})$ has five vertices with six edges. (right) Subgraph $H(\bar{x})$ mandated if $H(\bar{x})$ has six vertices and seven edges.}
		\label{fig:class graph combination k=3}
	\end{figure}
In this case, we can partition the vertices of $H(\bar{x})$ to construct the required restricted $Y$ class using $R2$ of Definition \ref{def:restrictedY}.

In general, for $k\geq 4$ the subgraph $H(\bar{x})$ will look like the ones constructed in the proof of Lemma \ref{lem:vanishing threshold_combination tuples}.
Starting from the first induced $3$-subgraph or $4$-subgraph containing $x_0$, it is possible to sequentially contract (disjoint) $3$-vertex $Y$-subgraphs of $H(\bar{x})$ or $4$-vertex $Y$-subgraphs belonging to the $\mathcal{G}_4$ class.
After this process, the class graph $\phi(\mathcal{H}(\bar{x}))$ is obtained by sequentially contracting $\mathcal{H}(\bar{x})$ is such that $\alpha(\phi(\mathcal{H}(\bar{x})))$ is a tree and $\omega(\phi(\mathcal{H}(\bar{x})))$ is a complete graph.
Thus $\mathcal{H}(\bar{x})$ is a restricted $Y$-class.

\end{proof}

\begin{figure}[h]
\centering
\includegraphics[scale=0.65]{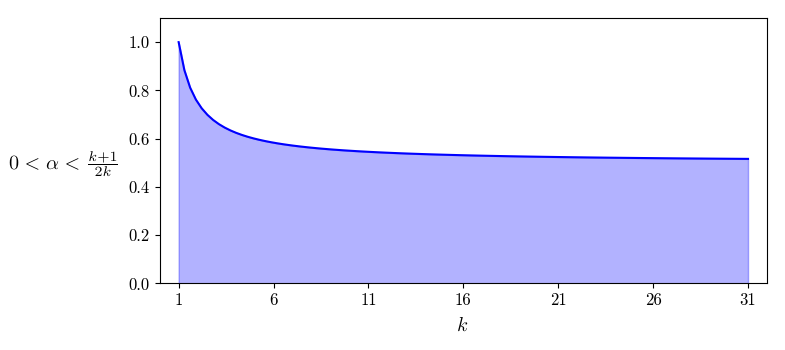}
\caption{The shaded region below the curve is the $q$ vs $k$ region for which we can have  non-vanishing $EMH_{k,k}(G)$ in expectation as $n \to \infty$ for graphs $G\sim G(n,n^{-q}, A)$. By Theorems \ref{thm:relation MH EMH DMH} and \ref{thm:iso MH DMH}, in the non-shaded region asymptotically almost surely $MH_{k,k}(G) \leq DMH_{k,k}(G)$, and if also $k \geq 5,$ then $MH_{k,k}(G) \cong DMH_{k,k}(G)$.}
\label{fig:non trivial emh_rgg}
\end{figure}

Combining this vanishing threshold with Theorem \ref{thm:relation MH EMH DMH}, we obtain the following characterization of the expected behavior of $MH_{k,k}(G(n,r,A))$ as $n \to \infty.$

\begin{cor}
Let $G=G(n,r,A)$ be a random geometric graph and take $r=n^{-q}$.
If $q > \frac{k+1}{2k}$ then the magnitude homology group $MH_{k,k}(G)$ is the subgroup of the discriminant magnitude homology group $DMH_{k,k}(G)$ generated by tuples $(x_0,\dots,x_k)$ such that the induced path always revisits the same edge $(x_i,x_j)$.
\end{cor}

Using the same combinatorics, we can again provide estimates of the expected Betti numbers $\beta_{k,k}$ for geometric random graphs as $n \to \infty$, see Figure \ref{fig:expected betti RGG}.

\begin{figure}[H]
	\centering
	\includegraphics[scale=0.65]{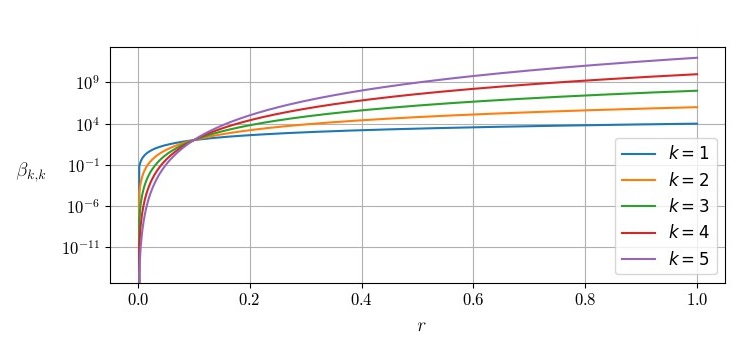}
	\caption{The expected value for the Betti numbers $\beta_{k,k}$, $k \in [1,5]$, of a random geometric graph on a flat torus of area $T_A^2 = \pi$ is plotted vertically against the radius $r$. In this example $n = 100$ and the $y$-axis is on a base 10 logarithmic scale to better visualize the large differences in values. Notice that the the curves all intersect at roughly $p=0.1=n^{-\alpha}$ with $n=100$ and $\alpha=\frac{1}{2} \sim \frac{k+1}{2k}$.}
	\label{fig:expected betti RGG}
\end{figure}

\begin{theorem}
\label{thm:asympBettiRGG}
Fix $r > 0,$ $k$ nonnegative. Write $\beta_{k,k}=\rank(EMH_{k,k}(G(n,r,A)))$. Then
\[
\lim\limits_{n\to \infty} \frac{\mathbb{E}[\beta_{k,k}]}{n^{k+1}r^{2k}} = \left(\frac{\pi}{|A|} \right)^k.
\]
\end{theorem}

\begin{proof}
 This proof is essentially identical to that of Theorem \ref{thm:asymp size BettiEMH}, with a different estimate for the probability of  $k$-trails appearing in $G\sim G(n, n^{-q}, A).$ Specifically, we have that each $k$-trail appears with probability $\left(\frac{\pi r^2}{|A|} \right)^k$ by Lemma \ref{lem:RGG_trails_indep}.
Since $EMH_{k,k}$ is free abelian by Lemma \ref{lem:LowerTriangular}, we have

\[
\lim\limits_{n\to \infty}\frac{\mathbb{E}[t_{k,k}]}{n^{k+1}r^{2k}} = \left(\frac{\pi}{|A|}\right)^k.
\]
As in the proof of Theorem \ref{thm:central limit thm}, per Lemma \ref{lem:RGG_trails_indep} it holds that
\[
\frac{\mathbb{E}[t_{k-1,k}]}{\mathbb{E}[t_{k,k}]} = 
\frac{\binom{n}{k}k! \left(1- \left(\frac{\pi r^2}{|A|} \right)^{k-1} \right) \left(\frac{\pi r^2}{|A|} \right)^k}{\binom{n}{k+1}(k+1)!\left(\frac{\pi r^2}{|A|} \right)^k} \sim
\frac{\frac{n^k}{k!} k!}{\frac{n^{k+1}}{(k+1)!}(k+1)!} =
\frac{1}{n} \xrightarrow{n \to \infty} 0.
\]
And so it follows that 
\[
\lim\limits_{n\to \infty}\frac{\mathbb{E}[t_{k,k}]}{\mathbb{E}[-t_{k-1,k}+t_{k,k}]} =1,
\]
which, together with equation (\ref{eq:morselikeineq}), proves the statement.
\end{proof}

Finally, it is possible to prove a central limit type result for the Betti numbers of the magnitude homology groups of the first diagonal $EMH_{k,k}(G)$.
The proof is identical to the result shown for the Erd\H{o}s-R\'{e}nyi random model in Theorem \ref{thm:central limit thm} by choosing $p=\left(\frac{\pi r^2}{|A|} \right)$.

\begin{theorem} \label{thm:clt_rgg}
Call $\beta_{k,k}=\rank(EMH_{k,k}(G(n,n^{-q},A)))$. Then as $n \to \infty,$
\[
\frac{\beta_{k,k} -  \mathbb{E}[\beta_{k,k}]}{\sqrt{Var(\beta_{k,k})}} \Rightarrow N(0,1).
\]
\label{thm:RGG_CLT}
\end{theorem}

\section{Conclusion}

In this paper, we investigated how the structure of a graph $G$ is represented in the magnitude homology groups of $G$. To strengthen this connection, we developed the theory of eulerian magnitude homology, built on a subcomplex of the magnitude chain complex that omits ``singular" trails that must revisit vertices. By restricting our attention to the eulerian trails, we are able to characterize classes of graphs which can support eulerian magnitude cycles  along the $k=\ell$ line. In this case, the image of the differential is zero for dimension reasons and so understanding these classes is sufficient to characterize eulerian magnitude homology completely. In Theorem \ref{thm:cycle_decomp} we produce a generating set for these homology groups, thus characterizing these groups in terms of existence of particular subgraphs. Due to its combinatorial complexity, we leave the matter of understanding how elements of these generating sets interact open for the present. 

Singular trails, which repeat landmarks, generate the complementary discriminant magnitude homology. The two theories are intertwined through the corresponding long exact sequence (\ref{eq:LES}), as the differential of a trail that must revisit a vertex may include trails that do not. Characterizing this interaction provides us with tools for dissecting the generators of magnitude homology, again focused along the $k = \ell$ diagonal where the combinatorics are more accessible. Here, as we saw in Theorems \ref{thm:iso MH DMH} and \ref{thm:lesnotsplit}, the existence of nonsingular subtrails is regulated by eulerian magnitude homology groups of lower degree.

In the context of Erd\"os-R\'enyi random graphs and random geometric graphs on the standard torus, where boundary effects are removed, we leverage our understanding of classes of graphs supporting generators along the $k=\ell$ line to develop vanishing thresholds for eulerian magnitude homology in limiting expectation in Theorems \ref{thm:vanishingthreshER} and \ref{thm:vanishing threshold RGG}. Combined with the long exact sequence relating eulerian and discriminant magnitude homology, these results allow us to completely characterize the magnitude homology groups in the vanishing range. Outside of the vanishing range, we provide limiting characterizations of the $(k,k)$-Betti numbers in Theorems \ref{thm:asymp size BettiEMH} and \ref{thm:asympBettiRGG} for both classes of graphs, and corresponding central limit theorems in Theorems \ref{thm:central limit thm} and \ref{thm:clt_rgg}.

\medspace

The tools developed in this paper suggest a number of avenues for further work. Here we highlight a few that the authors find particularly interesting.

\begin{enumerate}
\item We heavily leverage equality between the length $\ell$ and number of landmarks $k$ in a trail to obtain our combinatorial characterization of generators in Section \ref{sec:EMH}. We believe that these results can in turn be leveraged to iteratively study the combinatorics witnessed by the $\ell = k+i$ lines for increasing $i$, providing more insight into the graph-theoretic meaning of the higher magnitude homology groups. 

\item The expected values of Betti numbers of both Erd\H{o}s-R\'{e}nyi and random geometric graphs are computed here using an approach which heavily depends on the possibility of computing explicitly the probability that each edge appears in the graph.
There are examples in the literature \cite{kannan2019persistent, ayala2009discrete, bobrowski2015topology} of limit results for Betti numbers of random structures exploiting discrete Morse theory: would it be possible to use similar techniques to extend the results present in this work to more general classes of graphs?

\item The vanishing threshold and expected Betti numbers for random geometric graphs, computed in Theorem \ref{thm:vanishing threshold RGG} and \ref{thm:asympBettiRGG} respectively, are rather coarse upper bounds.
Indeed, we obtained them by bounding the number of $(k+1)$-tuples of vertices inducing a path of length $k$ with the quantity
\[
\binom{n}{k+1} \left(\frac{\pi r^2}{|A|} \right)^{k},
\]   
which is impossible to attain in some regimes, e.g. relatively sparse and sparse regimes \cite{duchemin2023random}.
Is it possible to establish more refined distribution-dependent thresholds and Betti numbers estimates by taking into account the specific regime?

\item The random geometric graphs studied here are embedded in the torus $\mathbb{T}^2$ mostly as a matter of convenience.
Would a similar proof for the vanishing threshold of $EMH_{k,k}(G)$ work (possibly with some restrictions on the acceptable regimes) in more general settings? 
\end{enumerate}

\subsection*{Acknowledgments}

GM would like to thank Simon Willerton, for pointing out a mistake in the first draft of the paper concerning the relation between magnitude homology and eulerian magnitude homology groups. The authors would also like to express their gratitude to the anonymous reviewers for their valuable and constructive suggestions, which have substantially improved the paper.

\bibliographystyle{amsplain}
\bibliography{sn-bib}

\end{document}